\documentclass{amsart}
\usepackage{geometry}
\usepackage{amscd,amssymb,verbatim,xcolor}
\usepackage{longtable, multirow}
\usepackage{soul,cancel,tikz}
\date{\today}

\newcommand\C[1]{{\mathcal #1 }}

\newcommand\la{{\lambda}}

\newtheorem{theorem}{Theorem}[section]

\newtheorem{assumption}[theorem]{Assumption}
\newtheorem{corollary}[theorem]{Corollary}

\newtheorem{definition}[theorem]{Definition}
\newtheorem{example}[theorem]{Example}
\newtheorem{lemma}[theorem]{Lemma}
\newtheorem{proposition}[theorem]{Proposition}
\newtheorem{remark}[theorem]{Remark}

\begin{document}
\title{The genuine unitary dual of complex spin groups}
\author{Kayue Daniel Wong}

\address{School of Science and Engineering, The Chinese University of Hong Kong, Shenzhen, Guangdong 518172, China }
\email{kayue.wong@gmail.com}

\author{Hongfeng Zhang}

\address{Department of Mathematics, University of Hong Kong, China }
\email{zhanghongf@pku.edu.cn}

\begin{abstract}
In this paper, we give a classification of all irreducible, unitary representations of complex spin groups.
\end{abstract}

\maketitle
\setcounter{tocdepth}{1}


\section{Introduction}\label{sec:intro}
\subsection{The unitary dual} A major unsolved problem in representation theory of real reductive Lie groups $G$ is to classify all irreducible, unitarizable $(\mathfrak{g},K)$-modules, i.e. the unitary dual $\widehat{G}$. A good understanding $\widehat{G}$ would be useful for various problems in automorphic forms, harmonic analysis, noncommutative geometry and so on. 


\smallskip
It is conjectured that for real reductive groups groups, the representations in $\widehat{G}$ can be obtained from one of the following:
\begin{enumerate}
    \item[(i)] a finite set of unitary representations called \emph{unipotent representations}; or
    \item[(ii)] parabolically inducing a module consisting of unitary characters and a unipotent representation; or
    \item[(iii)] some continuous deformations of the parabolically induced modules obtained in (ii), i.e. the \emph{complementary series}.
\end{enumerate}

Following the above conjecture, we classify all irreducible unitary, genuine representations of $Spin(m,\mathbb{C})$ treated as a real Lie group, i.e. representations that do not factor through $SO(m,\mathbb{C})$. Combined with the results of \cite{B89}, we obtain a full description of the unitary dual of $Spin(m,\mathbb{C})$.

\subsection{Unipotent representations} \label{subsec-unip}
As seen above, unipotent representations play an important role in the classification of $\widehat{G}$. 
In this section, we recall some aspects of unipotent representations for real reductive Lie groups for archimedean local field $\mathbb{F}$.

\smallskip

Motivated by the study of automorphic forms, Arthur proposed that there should be a \emph{packet} of representations associated to the $^\vee G$-equivalence classes of homomorphisms
$$
\Phi:\C W_{\mathbb{F}}\times SL(2,\mathbb{C})\longrightarrow\ ^\vee G
$$
where $\C W_{\mathbb{F}}$ is the Weil group of $\mathbb{F}$, and $^\vee G$ is the $L$-group of $G$. These packets (called \emph{Arthur packets}) are conjectured to possess certain properties such as unitarity and twisted endoscopy.

\smallskip
Special attention is given to cases when $\Phi|_{\C W_{\mathbb{F}}} = triv$. Under such hypothesis, the corresponding Arthur packets are parametrized by the image of $SL(2,\mathbb{C})$, which are parametrized by the nilpotent orbits $\mathcal{O}^{\vee}$ of $^\vee \mathfrak{g}$. The irreducible representations appearing in these packets are called (weak) unipotent packets attached to $\mathcal{O}^\vee$. Furthermore, if the orbit $\mathcal{O}^\vee$ is special in the sense of Lusztig, the representations in these packets are called \emph{special unipotent representations} attached to $\mathcal{O}^\vee$.

\smallskip
For $\mathbb{F} = \mathbb{C}$, Barbasch-Vogan \cite{BV85} gave a detailed study of special unipotent representations. For instance, they proved that these representations satisfy the endoscopic properties as conjectured by Arthur and, in the case of classical groups, Barbasch \cite{B89} proved these representations are all unitary. As for $\mathbb{F} = \mathbb{R}$, Barbasch-Ma-Sun-Zhu \cite{BMSZ21}, \cite{BMSZ22} constructed all special unipotent representations for classical groups, and proved that all these representations are unitary. 

\smallskip
However, if one is interested in the classification of $\widehat{G}$, it is well-known that special unipotent representations are not enough to constitute the whole unitary dual. For instance, the Segal-Shale-Weil representation in $Sp(2n,\mathbb{C})$ is not in any weak unipotent packet, and is not parabolically induced from any other representations (hence it is non-Arthur). 
Therefore one needs to generalize to \emph{unipotent representations} in the study of $\widehat{G}$. Apart from their importance in the study of the unitary dual, it is believed that such generalizations may have implications on the study of automorphic forms for certain covering groups of $G$ over other local fields. For instance, at $\mathbb{F} = \mathbb{R}$, the Segal-Shale-Weil representations are genuine representations on the nonlinear metaplectic group $Mp(2n,\mathbb{R})$. They play an important role in the study of endoscopic transfer between virtual representations of $Mp(2n,\mathbb{R})$ and $SO(n,n+1)$ (c.f. \cite{A98}, \cite{R98}).  

\smallskip
Here are some progress in the study of unipotent representations for complex Lie groups: For classical groups, Barbasch \cite{B89} defined and constructed all unipotent representations, and proved that they are all unitary. As for spin groups, Brega \cite{Br99} constructed an interesting collection of genuine unitary representations
for $G = Spin(2n,\mathbb{C})$ (Remark \ref{rmk-brega}). Later on, Barbasch \cite[Section 7]{B17} generalized Brega's construction to $Spin(2n+1,\mathbb{C})$ (Remark \ref{rmk-inclusion}). In Appendix \ref{sec-unipotent}, we will show in full detail that the genuine representations of $Spin(2n,\mathbb{C})$ studied by Brega
are all unipotent representations in the sense of \cite[Section 2.3]{B17}. Its proof can be easily translated to the case of $Spin(2n+1,\mathbb{C})$.

\subsection{Main results}
Here is the main theorem of this paper:
\begin{theorem}[Theorem \ref{thm-bv} and Theorem \ref{Main_odd}] \label{thm-11}
All irreducible unitary, genuine representations
of $G = Spin(m,\mathbb{C})$ are the lowest $K$-type subquotients of parabolically induced modules, whose inducing modules consist of genuine unitary characters of $\widetilde{GL}(a,\mathbb{C})$, or Stein complementary series tensored with a genuine unitary character in $\widetilde{GL}(2b,\mathbb{C})$, or a genuine unipotent representation of the same type as $G$ with lower rank given by \cite{Br99} (for $m$ even) or \cite{B17} (for $m$ odd), whose Langlands parameters are given explicitly in Remark \ref{rmk-brega} (for $m$ even) and Remark \ref{rmk-unipb} (for $m$ odd).
\end{theorem}

We end this section by offering some possible directions for further research. Firstly, if the inequalities of Remark \ref{rmk-brega} and Remark \ref{rmk-unipb} are strict, then the corresponding unipotent representations are {\it isolated} in the unitary dual $\widehat{G}$. As conjectured in \cite{T22}, we believe that these representations play an important role in the study of Arthur packets for (some covers of) classical groups over any local field.
Moreover, by some character formula calculations, one observes that certain linear combinations of the unitary representations in Theorem \ref{thm-11} 
satisfy some endoscopic identities 
similar to that of \cite[Corollary 12.4]{BV85}. One therefore hopes that this will shed some light on the study of automorphic forms and Arthur representations for nonlinear covering groups over other local fields. 
Finally, one observes that there is a {\it Shimura correspondence} on the unitary duals between $SO(2n,\mathbb{C})$ and $Spin(2n,\mathbb{C})$, along with the unitary duals between $Sp(2n,\mathbb{C})$ and $Spin(2n+1,\mathbb{C})$. Similar works of such correspondence can be found \cite{T96}, \cite{AH97} and, more recently, \cite{ABPTV07}. We will pursue along this direction in a forthcoming work.

\subsection{Structure of the manuscript}
In Section \ref{sec-prelim}, we recall some basic notions of irreducible, admissible $(\mathfrak{g},K)$-modules for complex group $G$ treated as a real group. A full proof of the theorem for $G=Spin(2n,\mathbb{C})$ is given in Section \ref{sec-mainsec} -- Section \ref{sec-nonunit}. Since it is clear from the results of \cite{Br99} and \cite{B17} that all representations appearing in the theorem are unitary (Section \ref{sec-unitary}), a substantial part of this paper is to prove that these representations exhaust $\widehat{G}$. To do so, we use \emph{bottom layer $K$-type} arguments introduced in Section \ref{sec-bottom} to reduce our study to representations whose lowest $K$-type is the spinor module. In Section \ref{sec-in} and \ref{sec-nonunit}, we
make heavy use of the intertwining operators studied
in \cite{D75} and more recently \cite{B10} to conclude that all representations with spinor lowest $K$-types that are not covered in
Theorem \ref{thm-11} must have an indefinite Hermitian form on some
\emph{spin-relevant $K$-types} (Definition \ref{def-relevant}). Since these $K$-types are always bottom layer,
this finishes the exhaustion proof. Finally, we will explain in Section \ref{sec-odd} how the proofs for $Spin(2n,\mathbb{C})$ can be modified to give the same result for $Spin(2n+1,\mathbb{C})$.

\subsection{Acknowledgements}
This work is supported by the National Natural Science Foundation of China (grant
no. 12371033) and Shenzhen Science and Technology Innovation Committee (grant no.
20220818094918001).

\section{Preliminaries} \label{sec-prelim}
\subsection{Representations for complex groups}
Let $G$ be a connected reductive complex Lie group viewed as a real Lie
group. Fix a maximal compact subgroup $K$, and a pair $(B,H=TA)$
where $B$ is a real Borel subgroup and $H$ is a $\theta-$stable Cartan
subgroup such that $T=B\cap H$, and $A$ the complement stabilized by $\theta$.
We write $\mathfrak{g}_0$, $\mathfrak{k}_0$, $\mathfrak{b}_0$, $\mathfrak{h}_0$,
$\mathfrak{t}_0$, $\mathfrak{a}_0$ as their respective Lie algebras, and their
complexifications are denoted by removing the subscript $\bullet_{0}$. Let $W$ denote the Weyl group $W(\mathfrak{g}_0, \mathfrak{h}_0)$.

The {\bf Langlands-Zhelobenko parameter} of any irreducible module is a pair $(\la_L;\la_R) \in \mathfrak{h}^* \cong \mathfrak{h}_0^* \times \mathfrak{h}_0^*$ such that $\mu:=\la_L-\la_R$ is the parameter of a character of $T$ in the
  decomposition of the $\theta-$stable
  Cartan subalgebra $H=T\cdot A,$ and $\nu:=\la_L+\la_R$ the
  $A-$character. Then the \textit{principal series representation}
  associated to $(\la_L;\la_R)$ is the $(\mathfrak{g}, K)-$module
$$
X(\lambda_L;\lambda_R) = X(\mu,\nu) = \mathrm{Ind}_B^G(e^{\mu}\otimes e^\nu \otimes
1)_{K-finite}
$$
with infinitesimal character $(\la_L;\la_R)$
(here the symbol $\mathrm{Ind}$ refers to normalized Harish-Chandra induction).

Let $J(\mu,\nu) = J(\la_L;\la_R)$ be the unique irreducible subquotient of
$X(\mu,\nu) = X(\la_L;\la_R)$ containing the lowest $K$-type $\mathcal{V}_{\{\mu\}}$ with extremal weight
$\mu=\la_L-\la_R$ (here $\{\omega\}$ is defined to the unique Weyl conjugate of $\omega$ such that
$\{\omega\}$ is dominant). This is called the {\bf Langlands subquotient}.

\begin{proposition}[Parthasarathy-Rao-Varadarajan,  Zhelobenko] \label{prop:BVprop}
Assume $G$ is a complex connected reductive group viewed as a real group, and
let $(\la_L;\la_R)$ and $(\la_L';\la_R') $ be parameters described above. Then the
following are equivalent:
\begin{itemize}
\item $(\la_L;\la_R)=(w\la_L';w\la_R')$ for some $w\in W.$
\item  $X(\la_L;\la_R)$ and $X(\lambda_L'; \lambda_R')$ have
  the same composition factors with the same multiplicities.
\item The Langlands subquotient of $X(\la_L;\la_R)$, written as
$J(\la_L;\la_R)$, is the same as that of $X (\lambda_L'; \lambda_R')$.
\end{itemize}
Furthermore, every irreducible $(\mathfrak{g}, K)$-module is equivalent to some $J(\la_L;\la_R)$.
\end{proposition}

\subsection{Intertwining operators} \label{sec-intertwine1}
The irreducible subquotient $J(\la_L;\la_R) = J(\mu,\nu)$ of the principal series representations
$X(\la_L;\la_R) = X(\mu,\nu)$ in Proposition \ref{prop:BVprop}
can be constructed using intertwining operators as follows: Let $w \in W$ be a Weyl group element. Then there is a standard intertwining operator
\begin{equation} \label{eq-intertwine}
i_w(\mu,\nu): X (\mu, \nu) \to X (w \mu, w\nu),
\end{equation}
which can be normalized to be the identity on
the lowest $K$-type. It is a meromorphic function of $\nu$ defined almost everywhere.

Suppose $X(\mu,\nu)$ is such that
$Re(\nu)$ is dominant with respect to the roots in $B$. Then
then the image of $i_{w_0}(\mu,\nu)$ in \eqref{eq-intertwine}, where $w_0$ is the longest element in the
Weyl group, is equal to $J(w_0\mu,w_0\nu) \cong J(\mu,\nu)$.

In Section \ref{sec-in}, we will have a detailed study of $i_w(\mu,\nu)$ on some
isotypic subspaces of $X(\mu,\nu)$  in $Spin(2n,\mathbb{C})$ when the lowest $K$-type $\mathcal{V}_{\{\mu\}}$ is the spinor module.
This constitutes the main tool in determining the (non-)unitarity of $J(\mu,\nu)$'s.

\subsection{Hermitian dual} \label{sec-herm}
Let $\pi$ be an admissible $(\mathfrak{g},K)$-module. Its Hermitian dual $\pi^h$ (cf. \cite[Definition 2.10]{V84}) is the $(\mathfrak{g},K)$-module
such that there is a 
nondegenerate hermitian pairing $\langle , \rangle: \pi \times \pi^h \to \mathbb{C}$ satisfying:
\[\langle (U_1,U_2)\cdot x, f \rangle = \langle x, -(\overline{U_2},\overline{U_1})\cdot f \rangle, \quad \quad \langle k\cdot x, f \rangle = \langle x, k^{-1}\cdot f\rangle\]
for $x \in \pi$, $f \in \pi^h$, $(U_1,U_2) \in \mathfrak{g}_0 \oplus \mathfrak{g}_0 \cong \mathfrak{g}$ (here $\overline{U}$ is the complex conjugate with the complex structure of $\mathfrak{g}_0$), and $k \in K$. 
Most notably, $\pi$ has a non-degenerate invariant Hermitian form
if and only if $\pi \cong \pi^h$ (cf. \cite[Corollary 2.15]{V84}).

By \cite[Section 2.4]{B89}, one has $J(\mu,\nu)^h = J(\mu,-\overline{\nu})$
(the same holds also for principal series representations). From now on, we will write 
$$(\mu^h,\nu^h) := (\mu,-\overline{\nu})$$
for the Langlands parameters of the Hermitian dual.
Therefore, $J(\mu,\nu)$ has a non-degenerate invariant Hermitian form if and only if
there exists $w_h \in W$ such that
\begin{equation} \label{eq-herm}
w_h \mu = \mu^h, \quad \quad w_h \nu = \nu^h.
\end{equation}

Thanks to the work of Harish-Chandra, the problem of classifying the unitary, irreducible
$G$-modules is equivalent to the problem of classifying the irreducible $(\mathfrak{g}, K)$-modules having a positive definite invariant Hermitian form.
From now on, we call such $(\mathfrak{g}, K)$-modules {\bf unitarizable}, and the collection of all such modules is called the {\bf unitary dual} $\widehat{G}$ of $G$.
By our above discussions, one can only focus on the Langlands parameters satisfying \eqref{eq-herm} in the classification of $\widehat{G}$.


%
%
%
%


\subsection{Reduction to real parameters} \label{sec-real}
To further narrow down the parameters to be considered for the classification of $\widehat{G}$, one has the following `reduce to real parameters' theorem:
\begin{theorem}[\cite{B89}, Proposition 2.5]
Let $\pi \in \widehat{G}$ be a unitarizable $(\mathfrak{g},K)$-module. Then there exists a parabolic subgroup $P = LN \leq G$, a unitary representation $\pi_{re} \in \widehat{L}$ which has real infinitesimal character, and a unitary character $\chi \in \widehat{L}$ such that
$$\pi = \mathrm{Ind}_P^G\left((\pi_{re} \otimes \chi) \otimes {\bf 1}\right).$$
\end{theorem}
In view of the above theorem and \eqref{eq-herm}, one only considers irreducible modules
$J(\mu,\nu)$ such that
$\nu = \overline{\nu}$, and there exists an involution $s \in W$ such that $s\mu = \mu$ and $s\nu = -\nu$.

\section{The main theorem} \label{sec-mainsec}
From now on, we focus only on $G = Spin(2n,\mathbb{C})$. Fix a choice of positive simple roots given
by the Dynkin diagram:
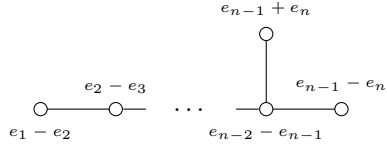
\begin{center}
\begin{tikzpicture}
\draw
    (-01,0) node[circle,draw=black, fill=white,inner sep=0pt,minimum size=5pt,label=below:{\tiny $e_1 - e_2$}] {}
 -- (0,0) node[circle,draw=black, fill=white,inner sep=0pt,minimum size=5pt,label=above:{\tiny $e_2 - e_3$}] {};

\draw (1,0) node {$\dots$};


\draw
 (2,1) node[circle,draw=black, fill=white,inner sep=0pt,minimum size=5pt,label=above:{\tiny $e_{n-1} + e_n$}] {}
--  (2,0) node[circle,draw=black, fill=white,inner sep=0pt,minimum size=5pt,label=below:{\tiny $e_{n-2} - e_{n-1}$}] {}
 -- (3,0) node[circle,draw=black, fill=white,inner sep=0pt,minimum size=5pt,label=above:{\tiny $e_{n-1} - e_n$}] {};
 \draw[thin](0.1,0)--(0.4,0);
 \draw[thin](1.6,0)--(1.9,0);
\end{tikzpicture}\end{center} 

Let $J(\mu,\nu) = J(\lambda_L;\lambda_R)$ be an
irreducible, Hermitian $(\mathfrak{g},K)$-module with real infinitesimal character. By Proposition \ref{prop:BVprop},
one can conjugate $\mu$ and $\nu$ (resp. $\lambda_L$ and $\lambda_R$)
with some $w \in W$ such that $w\mu$ is dominant. By replacing $J(\mu,\nu)$ with $J(w\mu,w\nu)$, one may assume $\mu$ is dominant
with respect to the above choice of simple roots.

Suppose $\mu$ consists only of integer coefficients, then $J(\mu,\nu)$ can be descended to an
irreducible representation of $SO(2n,\mathbb{C})$.
Therefore, we only study $J(\mu,\nu)$
with $\mu$ equal to
\begin{equation} \label{eq-mugen}
\mu := (\dots; \mu^r; \dots; \mu^2; \mu^1), \quad \quad \text{where}\quad \mu^r = (\underset{m_r}{\underbrace{\frac{2r-1}{2},\dots ,\frac{2r-1}{2}}}).
\end{equation}
It is possible that the last $\frac{1}{2}$-coordinate of $\mu^1$ is equal $\frac{-1}{2}$. However, this
only differs from \eqref{eq-mugen} by an outer automorphism.


\subsection{Bottom-layer arguments} \label{sec-bottom}
The classification of $\widehat{G}$ can be effectively simplified by using bottom layer $K$-types as detailed in
\cite{KV95}. The special case of complex groups  is in Section 2.7 of \cite{B89}.

Suppose $J(\mu,\nu)$ is an irreducible Hermitian $(\mathfrak{g},K)$-module with real infinitesimal character,
and $\mu$ is of the form \eqref{eq-mugen}. Let 
$$L := \prod_{r \geq 1} L_r, \quad \quad L_r := \begin{cases} GL(m_r,\mathbb{C}) &\text{if}\ r > 1 \\ SO(2m_1,\mathbb{C}) &\text{if}\ r =1 \end{cases}$$
be a Levi subgroup of $SO(2n,\mathbb{C})$, and $\widetilde{L} := \mathrm{pr}^{-1}(L)$ be the Levi subgroup of $G$, where 
$\mathrm{pr}: G \to SO(2n,\mathbb{C})$ is the double covering map.
Let $\widetilde{L_r} := \mathrm{pr}^{-1}(L_r)$, then the $\widetilde{L_r}$'s commute with each other, and the irreducible genuine representations of $\widetilde{L}$
are given by tensor products of irreducible genuine representations of $\widetilde{L_r}$. 

Consider the induced module
\begin{equation} \label{eq:inclusion}
I := \mathrm{Ind}_{\prod_{r \geq 2} A_{m_r-1} \times D_{m_1}}^{D_n}\left(\bigotimes_{r\geq 2}
J_{A_{m_r-1}}(\mu^r,\nu^r) \otimes
J_{D_{m_1}}(\mu^1,\nu^1)\right)
\end{equation}
containing $J(\mu,\nu)$ as its lowest $K$-type subquotient (from now on, we only specify the information
on the Lie Type of the Levi subgroup for parabolic induction whenever there is no danger of confusion).
Then a {\bf bottom layer $K$-type} $\mathcal{V}_{\beta_r}$ has highest weight of the form $\beta_r =\mu+ \mu_{L_r}$, where
$\mu_{L_r}$ are some integer combination of the roots in $L_r$, so that $\beta_r$
is $K$-dominant. In such cases, the multiplicities and
signatures of the $K \cap L$-type $\mathcal{V}_{\beta_r|_{L}}$ in the inducing module of $I$ coincide with that of the $K$-type $\mathcal{V}_{\beta_r}$ in $J(\mu,\nu)$.

\smallskip
By Section 2.7 of \cite{B89}, some bottom layer $K$-types for $I$ are
obtained by:
\begin{itemize}
\item For $r \geq 2$, one can add $(1,\dots ,1,0,\dots ,0,-1,\dots , -1)$ (equal
number of $1$'s and $-1$'s)  to $\mu^r$ in $\mu$.
\item For $r = 1$, one can add $(1,\dots ,1,0,\dots ,0)$ (even number of $1$'s)
or $(1,\dots,1,0,\dots,0,-1)$ (odd number of of $1$'s) to $\mu^1$ in $\mu$.
\end{itemize}

\subsection{(Pseudo-)spherical unitary dual of $GL(n,\mathbb{C})$} \label{sec-pseudo}
In view of \eqref{eq:inclusion}, it is essential to determine the unitarity of
$J_{A_{m_r-1}}(\mu^r, \nu^r)$, whose lowest $K = \widetilde{U}(m_r)$-type has highest weight
equal to $(\frac{2r-1}{2}, \dots,\frac{2r-1}{2})$. These representations are {\it pseudo-spherical} in the sense
of \cite{V86}, since they are all equal to the irreducible spherical representation
$J_{A_{m_r-1}}\left(0,\nu^r \right)$ tensored with
a unitary character $(\frac{\det}{|\det|})^{\frac{2r-1}{2}}$. As a consequence, one only needs to understand the spherical unitary dual of $GL(n,\mathbb{C})$, which we recall below.

\medskip
Let $J_{A_{n-1}}(0,\nu)$ to denote the irreducible subquotient of $GL(n,\mathbb{C})$ with Langlands parameter $(\lambda_L;\lambda_R)=(\frac{\nu}{2};\frac{\nu}{2})$. We describe all $\nu$ such that $J_{A_{n-1}}(0,\nu)$ is unitarizable. By the discussions in Sections \ref{sec-herm} and \ref{sec-real}, one may assume that $\nu$ is real, and there exists $w\in W(A_{n-1})$ such that $w\nu=-\nu$.

Let $\nu_1=(\nu_{1,1},\dots,\nu_{1,n_1})$ be the longest subsequence of $\nu$ such that
\[\nu_{1,j}-\nu_{1,j+1}\in 2\mathbb{N}\backslash \{0\} \quad \text{for all} \quad 1\leq j\leq n_1-1.\]
Let $\nu_2=(\nu_{2,1},\dots,\nu_{2,n_2})$ be the longest subsequence of $\nu\setminus \nu_1$ such  that $\nu_{2,j}-\nu_{2,j+1}\in 2\mathbb{N}\backslash \{0\}$ for any $1\leq j\leq n_2-1$. Repeating the process on the remaining terms, we get
$$\nu=(\nu_1; \dots; \nu_m)$$
with $\nu_i=(\nu_{i,1},\dots,\nu_{i,n_i})\in \mathbb{R}^{n_i}$ such that $\nu_{i,j}-\nu_{i,j+1}\in  2\mathbb{N}\backslash \{0\}$ for any $1\leq j\leq n_i-1$.

\medskip
The following proposition is useful to understand the structure of $J_{A_{n-1}}(0,\nu)$.
\begin{proposition}[\cite{V86}, Proposition 12.2]
Let $J_i=J_{A_{n_i-1}}(0,\nu_i)$, then
\begin{equation} \label{eq-asph}
J_{A_{n-1}}(0,\nu)\cong \mathrm{Ind}_{\prod_{i=1}^m {A_{n_i-1}}}^{A_{n-1}}\left(\bigotimes_{i=1}^m J_i\right).
\end{equation}
\end{proposition}
Since we assume $J_{A_{n-1}}(0,\nu)$ has an invariant Hermitian structure, there are two possibilities for the $J_i=J_{A_{n_i-1}}(0,\nu_i)$'s appearing above:
\begin{itemize}
\item[(a)] $J_i$ is Hermitian, i.e. $\nu_i = - w_i\nu_i$ for some $w_i \in W(A_{n_i-1})$; or
\item[(b)] $J_i$ is not Hermitian, and there is a corresponding $J_j$ equal to its Hermitian dual $J_i^h$. In other words, there exists
$w_i \in W(A_{n_i-1}) = W(A_{n_j-1})$ such that $\nu_i = -w_i\nu_j$.
\end{itemize}
\begin{lemma}[\cite{V86}, Lemma 12.6] \label{lem-dirac}
Retain the above setting. If there exists $J_k = J(0,\nu_k)$ such that it is not one-dimensional, then
$J_{A_{n-1}}(0,\nu)$ is not unitary. More precisely, the Hermitian form of $J_{A_{n-1}}(0,\nu)$ is indefinite at the trivial $K$-type and the adjoint $K$-type with highest weight $(1,0,\dots,0,-1)$.
\end{lemma}
Therefore, one must have $\nu_{i,j} - \nu_{i,j+1} = 2$ for all $i$ and $j$. In particular, the only possibility for the Hermitian $J_i$ in Case (a) above is when $J_i$ is the trivial module. As for Case (b), we are left to study the unitarity of the following:
\begin{definition} \label{def-comp}
Given $a\in \mathbb{N}_+$ and $t\in \mathbb{R}$, let $\nu_{a,t} := (a-1+t, a-3+t, \dots, -a+3+t, -a+1+t)$. For $r \in \frac{1}{2}\mathbb{N}$, let 
\begin{equation} \label{eq-compat}
\mathrm{comp}_{r}(a,t) := J((r,\dots,r), \nu_{a,t}) = J(0, \nu_{a,t}) \otimes (\frac{\det}{|\det|})^{r}.
\end{equation}
Then the $t$-complementary series of $GL(2a,\mathbb{C})$ is given by
$$\mathrm{Ind}_{A_{a-1} \times A_{a-1}}^{A_{2a-1}}\left( \mathrm{comp}_0(a,t) \otimes \mathrm{comp}_0(a,t)^h\right).$$
\end{definition}
By the discussions above and induction in stages, the module $J(0,\nu)$ in \eqref{eq-asph} is unitary implies that
it is parabolically induced from a tensor product of trivial modules and $t$-complementary series. The theorem below determines
precisely when the $t$-complementary series is unitary:
\begin{theorem}[\cite{V86}, Lemma 12.12] \label{thm-comp}
Using the setting of Definition \ref{def-comp}, the $t$-complementary series is unitary if and only if $|t| < 1$ (Stein's complementary series). Otherwise,
\begin{itemize}
\item[(a)] If $|t|$ is a positive integer, then the $t$-complementary series is reducible. Indeed, the parameters $\nu_{a,t} \cup (-\nu_{a,t})$ can be rearranged into sequences of unequal lengths; or
\item[(b)] If $q<|t|<q+1$ for some integer $1\leq q\leq a$, then the Hermitian form of the
$t$-complementary series is indefinite at the trivial $K$-type and the $K$-type with highest weight $(\underbrace{1,\dots,1}_{a-q+1},0,\dots,0,\underbrace{-1,\dots,-1}_{a-q+1})$.
\item[(c)] If $|t| > a+1$, then the Hermitian form of the $t$-complementary series is indefinite at the trivial $K$-type and the adjoint $K$-type.
\end{itemize}
\end{theorem}

%
%

\subsection{A unitarizability criterion}
With the knowledge of $GL(n,\mathbb{C})$ and the bottom-layer argument, we can give a necessary condition for an irreducible Hermitian $(\mathfrak{g},K)$-module $J(\mu,\nu)$ to be unitarizable. By the discussions in Section \ref{sec-herm}, there exists $w_r\in W(D_{m_r})$ such that $w_r\mu^r=w_r\mu^r$ and $w_r\nu^r=-\nu^r$ for any $r \geq 1$.

\begin{proposition}
Let $J(\mu,\nu)$ be an irreducible Hermitian $(\mathfrak{g},K)$-module, and $I$ be the induced module \eqref{eq:inclusion}
corresponding to $J(\mu,\nu)$. If $J_{A_{m_r-1}}(\mu^r,\nu^r)$ is not unitary for some $r \geq 2$, then so is $J(\mu,\nu)$.
\end{proposition}

\begin{proof}
By the results in the previous section, if $J_{A_{m_r-1}}(\mu^r,\nu^r)$ is non-unitary, then the Hermitian form on $J_{A_{m_r-1}}(\mu^r,\nu^r)$ will be indefinite at the lowest $K$-type and the $K$-type with highest weight
\[(\underbrace{\frac{2r+3}{2},\dots,\frac{2r+3}{2}}_{q},\frac{2r+1}{2},\dots,\frac{2r+1}{2},\underbrace{\frac{2r-1}{2},\dots,\frac{2r-1}{2}}_q),\] for some $q\in \mathbb{N}_+$.
%
%
By the bottom-layer argument in Section \ref{sec-bottom}, the Hermitian form on $I$
and its irreducible subquotient $J(\mu,\nu)$ is also indefinite at the lowest $K$-type $\mu$ and the $K$-type with highest weight
\[\beta_r = (\dots,\mu_{r+1},\underbrace{\frac{2r+3}{2},\dots,\frac{2r+3}{2}}_{q},\frac{2r+1}{2},\dots,\frac{2r+1}{2},\underbrace{\frac{2r-1}{2},\dots,\frac{2r-1}{2}}_q,\mu_{r-1},\dots,\mu_1),\]
for some $q\in \mathbb{N}_+$. Therefore, $J(\mu,\nu)$ is not unitary.
\end{proof}

\begin{corollary}
Let $J(\mu,\nu)$ be an irreducible Hermitian $(\mathfrak{g},K)$-module with $\mu = (\dots,$ $\mu^r,$ $\dots,$ $\mu^2).$
Then $J(\mu,\nu)$ is unitary if and only if each $J_{A_{m_r-1}}(\mu^r,\nu^r)$ is unitary for all $r \geq 2$.
\end{corollary}

So we are reduced to studying the Hermitian module $J_{D_{m_1}}(\mu^1,\nu^1)$.  
To determine the (non-)unitarity of these modules, the following $K$-types will play an important role:
\begin{definition} \label{def-relevant}
Let $G = Spin(2n,\mathbb{C})$. The $K$-types $\mathcal{V}_{\eta(q)}$ with highest weights
$$\eta(q) = (\underbrace{\frac{3}{2},\dots,\frac{3}{2}}_q,\underbrace{\frac{1}{2},\dots,\frac{1}{2}}_{n-q-1},\frac{(-1)^q}{2})$$
are called {\bf spin-relevant $K$-types}.
\end{definition}
  
To begin with, we recall the following:
\begin{theorem}[\cite{B89}, Corollary 2.9]\label{KL-thm}
Let $P=MN$ be a parabolic subgroup of $G$. Suppose
$(\mu+\nu,\check{\alpha})\notin 2\mathbb{Z}$ for any $\alpha\in \Delta(\mathfrak{n})$, then
\[J(\mu,\nu) \cong \mathrm{Ind}^G_{M}(J_M(\mu,\nu)).\]
\end{theorem}

Write $\nu^1 =(\nu_{1},\dots,\nu_{m_1})\in \mathbb{R}^{n}$ (recall that $\mu^1 = (\frac{1}{2}, \dots, \frac{1}{2})$). Let $T \subseteq (-1,1]$ be a finite subset such that for any $t \in T$, there exists some $1 \leq i \leq n$ satisfying $\nu_i - t \in 2\mathbb{Z}$.  

For $t \in T$, let 
$$N_t := \{\nu_i\ |\ \nu_i - t \in 2\mathbb{Z}\}, \quad \quad M_t:= (\overbrace{\frac{1}{2},\dots,\frac{1}{2}}^{|N_t|\ terms})$$
so that 
$$\mu^1 = \bigsqcup_{t \in T} M_t, \quad \quad \nu^1 = \bigsqcup_{t \in T} N_t$$ 
as sets. Since we assume $J_{D_{2m_1}}(\mu^1,\nu^1)$ to be Hermitian, one has $M_{-t} = M_t^h = M_t$ and $N_{-t} = N_t^h = -N_{t}$ for all $t \in T$. Now Theorem \ref{KL-thm} implies that $J(\mu,\nu)$ is equal to
\begin{equation} \label{eq-bv}
\mathrm{Ind}_M^G\begin{pmatrix} \displaystyle \bigotimes_{t \in T'} J_{A_{2(|N_t|+|N_{1-t}|)-1}}(M_t\cup M_{-t}\cup -M_{1-t}\cup -M_{-1+t},\ N_t\cup N_{-t}\cup N_{1-t}\cup N_{-1+t}) \otimes \\  J_{A_{|N_0|+|N_1|-1}}(M_{0} \cup -M_1,\ N_0\cup N_1) \otimes
J_{D_{2|N_{1/2}|}}(M_{1/2} \cup M_{1/2}^h,\ N_{1/2}\cup N_{1/2}^h) \end{pmatrix}\end{equation}
where
$M$ is of Type $\mathop{\prod}\limits_{t \in T'} A_{2(|N_t|+|N_{1-t}|)-1} \times A_{|N_0|+|N_1|-1} \times
 D_{2|N_{1/2}|}$, and $T' := T \cap \big((0,\frac{1}{2})\cup (\frac{1}{2},1)\big)$. If $J(\mu,\nu)$ is unitary, one has 
\begin{itemize}
\item[(i)] $J_{A_{2(|N_t|+|N_{1-t}|)-1}}(M_t\cup M_{-t}\cup -M_{1-t}\cup -M_{-1+t},\ N_t\cup N_{-t}\cup N_{1-t}\cup N_{-1+t})$ is parabolically induced from a tensor product of Stein's complementary series tensored with $(\frac{\det}{|\det|})^{\pm1/2}$.
\item[(ii)] $J_{A_{|N_0|+|N_1|-1}}(M_0 \cup -M_1,\ N_0\cup N_1)$ is parabolically induced from a tensor product of the unitary character $(\frac{\det}{|\det|})^{\pm1/2}$.
\end{itemize}
Actually, assume that (i) (resp. (ii)) is not true. By the classification of unitary representations in type A, the Hermitian form of the modules in (i) (resp. (ii)) will be indefinite at the $K$-type with highest weight
\[(\underbrace{\frac{3}{2},\dots,\frac{3}{2}}_q,\frac{1}{2},\dots,\frac{1}{2},\underbrace{\frac{-1}{2},\dots,\frac{-1}{2}}_q)\ \text{or} \ (\underbrace{\frac{1}{2},\dots,\frac{1}{2}}_q,\frac{-1}{2},\dots,\frac{-1}{2},\underbrace{\frac{-3}{2},\dots,\frac{-3}{2}}_q)\]
for some $q\in \mathbb{N}_+$. Hence, the Hermitian form of the induced module \eqref{eq-bv} or equivalently $J_{D_{2m_1}}(\mu_1,\nu_1)$ is indefinite at $\mathcal{V}_{\eta(q)}$. Since $\mathcal{V}_{\eta(q)}$ are the bottom layer for the induced module \eqref{eq:inclusion}, the Hermitian form of $J(\mu,\nu)$ will be indefinite at $\mathcal{V}_{\eta(q)}$.

So we are further reduced to studying the Hermitian module $J_{D_{2|N_{1/2}|}} (M_{1/2} \cup M_{1/2}^h,N_{1/2}\cup N_{1/2}^h)$. From now on, we assume that $(\mu, \nu) = (M_{1/2} \cup M_{1/2}^h,N_{1/2}\cup N_{1/2}^h)$, and study the unitarity of
\begin{equation} \label{eq-dd}
J_{D_{2n}}(M_{1/2} \cup M_{1/2}^h,N_{1/2}\cup N_{1/2}^h)  = \mathrm{Ind}_{D_n \times D_n}^{D_{2n}}\left(J(M_{1/2},N_{1/2}) \otimes J(M_{1/2},N_{1/2})^h\right),
\end{equation}
where the last equality is 
on the level of virtual representations:
\[\mathrm{Ind}_{D_n \times D_n}^{D_{2n}}(\pi \otimes \pi^h) := \sum_{(w, w^h) \in W(D_n \times D_n)} c(w)c^h(w^h)X(\lambda_{\pi} \cup \lambda_{\pi^h}; w\lambda_{\pi} \cup w^h\lambda_{\pi^h}),\]
where $\pi = \displaystyle \sum_{w \in W(D_n)} c(w)X_{D_n}(\lambda_{\pi}; w\lambda_{\pi})$ and $\pi^h = \displaystyle \sum_{w^h \in W(D_n)} c^h(w^h)X_{D_n}(\lambda_{\pi^h}; w^h\lambda_{\pi^h})$ are the character formulas of $\pi$ and $\pi^h$ respectively.

\subsection{Statement of the main theorem} \label{sec-main}
We now present the final reduction step in the determination of $\widehat{G}$, which leads us to the main theorem in Theorem \ref{thm-main}. Recall in the previous section that one only needs to study the unitarity of
$$\pi = J_{D_{2n}}(M_{1/2}\cup M_{1/2}^h,\ N_{1/2}\cup N_{1/2}^h).$$
In terms of $(\lambda_L;\lambda_R)$-coordinates, one has
$\lambda_L  = \left(\frac{1}{2}(M_{1/2} + N_{1/2}); \frac{1}{2}(M_{1/2}^h + N_{1/2}^h)\right)$ and 
$\lambda_R  = \left(\frac{1}{2}(-M_{1/2} + N_{1/2}); \frac{1}{2}(-M_{1/2}^h + N_{1/2}^h)\right)$.

\medskip
Consider the induced module
\begin{equation} \label{eq-reduce2}
\mathrm{Ind}_{A_{2n-1}}^{D_{2n}}\left(\mathrm{Ind}_{A_{n-1} \times A_{n-1}}^{A_{2n-1}}\left( J_{A_{n-1}}(M_{1/2},N_{1/2}) \otimes J_{A_{n-1}}(M_{1/2}^h,N_{1/2}^h) \right)\right)\end{equation}
having $\pi$ as its lowest $K$-type subquotient. By Section \ref{sec-bottom}, there exists some irreducible representations $\iota_i$ of Type $A_{n_i-1}$ such that
$$J_{A_{n-1}}(M_{1/2},N_{1/2})  = \mathrm{Ind}_{\prod_i A_{n_i-1}}^{A_{n-1}}\left(\bigotimes_i \iota_i \right) \quad \quad
J_{A_{n-1}}(M_{1/2}^h,N_{1/2}^h)  = \mathrm{Ind}_{\prod_i A_{n_i-1}}^{A_{n-1}}\left(\bigotimes_i \iota_i^h \right).$$
Since the adjoint $K$-type $\mathcal{V}_{\eta(1)}$ (c.f. Definition \ref{def-relevant}) is bottom layer with respect to the induced module \eqref{eq-reduce2}, Lemma \ref{lem-dirac} implies that if $\pi$ is unitary, then all
$\iota_i$ must be one-dimensional, and
\begin{equation} \label{eq-iota}
\iota_i \otimes \iota_i^h \cong \mathrm{comp}_{1/2}\left(n_i,\frac{2r-1}{2}\right) \otimes \mathrm{comp}_{1/2}\left(n_i,\frac{2r-1}{2}\right)^h
\end{equation}
for some $1 \leq r \leq n_i$ by Theorem \ref{thm-comp}(b). More explicitly, $\iota_i \cong \begin{cases} \mathrm{comp}_{1/2}\left(n_i,\frac{2r-1}{2}\right) &\text{if}\ n_i+r\ \text{is even}; \\
\mathrm{comp}_{1/2}\left(n_i,\frac{2r-1}{2}\right)^h & \text{if}\ n_i+r\ \text{is odd}\end{cases}.$

\medskip
In conclusion, one only needs to study $\pi = J(M_{1/2}\cup M_{1/2}^h,N_{1/2}\cup N_{1/2}^h)$, with $N_{1/2}$ equal to
$$N_{1/2} = \bigcup_{i=1}^l \left(\frac{1}{2} + 2(x_i-1),\   \dots,\   \frac{5}{2},\  \frac{1}{2},\   \frac{-3}{2},\  \dots,\  \frac{1}{2} - 2y_i \right) 
\cup \bigcup_{i=l+1}^k
           \left(\frac{1}{2} + 2(x_i-1),\  \dots,\   \frac{5}{2},\   \frac{1}{2}\right)$$ 
for two descending chains of positive integers $\{x_1 \geq x_2 \geq \dots  \geq x_k\}$, $\{y_1 \geq y_2 \geq \dots \geq y_l\}$ with $k \geq l$. We will write
\begin{equation} \label{eq-strings}
N_{1/2} \longleftrightarrow \begin{pmatrix} x_1 & x_2 & \dots & x_k \\ y_1 & y_2 & \dots & y_k \end{pmatrix}
\end{equation}
from now on, by assuming $y_i = 0$ if $i > l$ if necessary. More explicitly, each $\begin{pmatrix} x_i \\ y_i \end{pmatrix}$ in \eqref{eq-strings} corresponds to the one-dimensional module $\iota_i$ or $\iota_i^h$ in \eqref{eq-iota} by:
\begin{equation} \label{eq-complementary}
\begin{cases}
\mathrm{comp}_{1/2}(x_i+y_i,x_i-y_i-\frac{1}{2}) &\text{if}\ x_i > y_i \\
\mathrm{comp}_{1/2}(x_i+y_i,y_i-x_i+\frac{1}{2})^h &\text{if}\ x_i \leq y_i,
\end{cases}\end{equation}
In particular, if $x_i - y_i = 0$ or $1$, then $\iota_i \otimes \iota_i^h$ corresponds to the Stein's complementary series (cf. Theorem \ref{thm-comp}).

\medskip
Here is the main theorem for $G = Spin(2n,\mathbb{C})$:
\begin{theorem} \label{thm-main}
Let $\pi = J\left(M_{1/2} \cup M_{1/2}^h,\ N_{1/2} \cup N_{1/2}^h\right)$, where $N_{1/2}$ is as given in \eqref{eq-strings}. Then $\pi$ is unitary if and only if
\begin{equation} \label{eq-unitary}
x_i \geq y_i\quad \text{and} \quad y_i + 1 \geq x_{i+1}
\end{equation}
for all $1 \leq i \leq k$. Otherwise, $\pi$ has an indefinite signature on the lowest $K$-type and $\mathcal{V}_{\eta(i)}$ for some $i$. In particular, the parameter must be of the form:
$$\begin{pmatrix} x_1 & x_2 & \dots & x_l & x_{l+1} & 1 & \dots & 1 \\ y_1 & y_2 & \dots & y_l & 0 & 0 & \dots & 0 \end{pmatrix}.$$
\end{theorem}
Since $\eta(i)$ is bottom layer corresponding to the induced module \eqref{eq:inclusion}, this finishes the classification
of $\widehat{G}$ for $G = Spin(2n,\mathbb{C})$.

\begin{remark} \label{rmk-brega}
When the inequalities in \eqref{eq-unitary} are strict, i,e. $N_{1/2}$ is of the form:
$$\begin{pmatrix} x_1 & x_2 & \dots & x_l & x_{l+1} \\ y_1 & y_2 & \dots & y_l & y_{l+1} \end{pmatrix} \quad \quad x_i > y_i \geq x_{i+1},$$
it matches with that of \cite[Equation (9)]{Br99}, where the unitarity of $\pi$ is proved in \cite[Theorem 3.1]{Br99}.
More precisely, the $\lambda_1^L$ in \cite[Equation (8)]{Br99} is equal to $\frac{1}{2}(M_{1/2}+N_{1/2})$. Henceforth we call them {\bf Brega representations}.
In Appendix \ref{sec-unipotent}, we will prove that all Brega representations are
{\it unipotent} in the sense of Barbasch \cite{B17}.
\end{remark}

To end this section, we list all the representations of $Spin(16,\mathbb{C})$ whose parameters are of the form as in Theorem \ref{thm-main}. We list whether these representations are unitary and, more precisely, Brega representations.
In the case when it is not unitary, we also give which spin-relevant $K$-type has an indefinite signature in its Hermitian form.

\begin{center}
\begin{longtable}{|c|c|c|c|c|}
\hline
   Parameter & $J(\lambda_L; \lambda_R)$ &  Is Unitary?  \\  \hline
    $\begin{pmatrix} 4 \\ 0 \end{pmatrix}$ &$J\begin{pmatrix} \frac{7}{2},& \frac{5}{2},& \frac{3}{2}, & \frac{1}{2},& 0, &-1,& -2, &-3 \\ 3,&2,&1,&0,& -\frac{1}{2},&-\frac{3}{2},&-\frac{5}{2},&-\frac{7}{2} \end{pmatrix}$   &  Yes - Brega   \\ \hline
    $\begin{pmatrix} 3 & 1 \\ 0 & 0 \end{pmatrix}$ &$J\begin{pmatrix} \frac{5}{2},& \frac{3}{2},& \frac{1}{2},& \frac{1}{2},&0,&0,&-1,&-2 \\ 2,&1,&0,&0,&-\frac{1}{2},&-\frac{1}{2},&-\frac{3}{2},&-\frac{5}{2} \end{pmatrix} $ & Yes \\ \hline
    $\begin{pmatrix} 3 \\ 1 \end{pmatrix}$ &$J\begin{pmatrix} \frac{5}{2},& \frac{3}{2},& \frac{1}{2},& -\frac{1}{2},& 1,&0,&-1,&-2 \\ 2,&1,&0,&-1,&\frac{1}{2},&-\frac{1}{2},&-\frac{3}{2},&-\frac{5}{2} \end{pmatrix}$ &  Yes - Brega   \\ \hline
     $\begin{pmatrix} 2 & 2 \\ 0 & 0\end{pmatrix}$ &$J\begin{pmatrix}\frac{3}{2},& \frac{3}{2},& \frac{1}{2},& \frac{1}{2},&0,&0,&-1,&-1 \\ 1,&1,&0,&0,&-\frac{1}{2},&-\frac{1}{2},&-\frac{3}{2},&-\frac{3}{2} \end{pmatrix}$ & No - $\mathcal{V}_{\eta(2)}$\\ \hline
    $\begin{pmatrix} 2 & 1 & 1 \\ 0 & 0 & 0 \end{pmatrix}$  &$J\begin{pmatrix}\frac{3}{2},& \frac{1}{2},& \frac{1}{2},& \frac{1}{2},&0,&0,&0,&-1 \\ 1,&0,&0,&0,&-\frac{1}{2},&-\frac{1}{2},&-\frac{1}{2},&-\frac{3}{2} \end{pmatrix}$ & Yes  \\ \hline
    $\begin{pmatrix} 2 & 1 \\ 1 & 0 \end{pmatrix}$ &$J\begin{pmatrix}\frac{3}{2},& \frac{1}{2},& \frac{1}{2},& -\frac{1}{2},&1,&0,&0,&-1 \\ 1,&0,&0,&-1,& \frac{1}{2},&-\frac{1}{2},&-\frac{1}{2},&-\frac{3}{2}\end{pmatrix}$ & Yes - Brega \\ \hline
     $\begin{pmatrix} 2  \\ 2  \end{pmatrix}$ &$J\begin{pmatrix}\frac{3}{2},& \frac{1}{2},& -\frac{1}{2},& -\frac{3}{2},& 2,&1,&0,&-1 \\ 1,&0,&-1,&-2,& \frac{3}{2},&\frac{1}{2},&-\frac{1}{2},&-\frac{3}{2}\end{pmatrix}$ & Yes   \\ \hline
     $\begin{pmatrix} 1 & 1 & 1 & 1 \\ 0 & 0 &0 &0 \end{pmatrix}$ &$J\begin{pmatrix}\frac{1}{2},& \frac{1}{2},& \frac{1}{2},& \frac{1}{2},&0,&0,&0,&0 \\ 0,&0,&0,&0,&-\frac{1}{2},&-\frac{1}{2},&-\frac{1}{2},&-\frac{1}{2}\end{pmatrix}$ & Yes \\ \hline
      $\begin{pmatrix} 1 & 1 & 1 \\ 1 & 0 & 0 \end{pmatrix}$ &$J\begin{pmatrix}\frac{1}{2},& \frac{1}{2},& \frac{1}{2},& -\frac{1}{2},& 1,&0,&0,&0 \\ 0,&0,&0,&-1,& \frac{1}{2},&-\frac{1}{2},&-\frac{1}{2},&-\frac{1}{2}\end{pmatrix}$ & Yes  \\ \hline
     $\begin{pmatrix} 1 & 1 \\ 1 &  1 \end{pmatrix}$ &$J\begin{pmatrix}\frac{1}{2},& \frac{1}{2},& -\frac{1}{2},& -\frac{1}{2},& 1,&1,&0,&0 \\0,&0,&-1,&-1,& \frac{1}{2},&\frac{1}{2},&-\frac{1}{2},&-\frac{1}{2}\end{pmatrix}$ & Yes \\ \hline
      $\begin{pmatrix} 1 & 1 \\ 2 & 0 \end{pmatrix}$ &$J\begin{pmatrix}\frac{1}{2},& \frac{1}{2},& -\frac{1}{2},& -\frac{3}{2},&2,&1,&0,&0 \\ 0,&0,&-1,&-2,& \frac{3}{2},&\frac{1}{2},&-\frac{1}{2},&-\frac{1}{2}\end{pmatrix}$ & No - $\mathcal{V}_{\eta(3)}$\\ \hline
    $\begin{pmatrix}  1 \\ 3  \end{pmatrix}$ &$J\begin{pmatrix}\frac{1}{2},& -\frac{1}{2},& -\frac{3}{2},& -\frac{5}{2},&3,&2,&1,&0 \\ 0,&-1,&-2,&-3,& \frac{5}{2},&\frac{3}{2},&\frac{1}{2},&-\frac{1}{2}\end{pmatrix}$ & No - $\mathcal{V}_{\eta(3)}$ \\ \hline
\end{longtable}
\end{center}

\section{Proof of unitarity} \label{sec-unitary}
In this section, we will prove the following stronger result, which (by Remark \ref{rmk-brega}) immediately implies the unitarity statement of Theorem \ref{thm-main}:
\begin{theorem}  \label{thm-bv}
Let $\pi =  J_{D_{2n}}\left(M_{1/2} \cup M_{1/2}^h,\ N_{1/2} \cup N_{1/2}^h\right)$
be such that $N_{1/2} \longleftrightarrow \begin{pmatrix} x_1 & x_2 & \dots & x_k  \\ y_1 & y_2 & \dots & y_k \end{pmatrix}$
satisfies \eqref{eq-unitary}. Then $\pi$ is a subquotient of a parabolically induced module, whose inducing module is
a tensor product of some Stein's complementary series $\mathrm{Ind}_{A_{\ell-1} \times A_{\ell-1}}^{A_{2\ell-1}}$ $(\mathrm{comp}_{1/2}(\ell,\frac{1}{2})$ $\otimes$ $\mathrm{comp}_{1/2}(\ell,\frac{1}{2})^h)$ along with a Brega representation (Remark \ref{rmk-brega}).
\end{theorem}

\begin{proof}
We prove by induction on the number of equalities in \eqref{eq-unitary}. Note that in the base case when there are no equalities in \eqref{eq-unitary}, then $\pi$ itself is a Brega representation.

\smallskip
By the induction hypothesis, suppose the theorem holds for all $\pi_r$ with $N_{1/2}$-coordinates having exactly $r$ equalities in \eqref{eq-unitary}. Consider $\pi_{r+1}$, whose corresponding $N_{1/2}$-coordinates have $r+1$ equalities in \eqref{eq-unitary}.

\smallskip
Firstly, suppose $x_i = y_i$ for some $1 \leq i \leq l$, then $J_{D_n}(M_{1/2},N_{1/2})$ (resp. $J_{D_n}(M_{1/2}^h,N_{1/2}^h)$)
is the subquotient of the induced module
$$\mathrm{Ind}_{A_{2x_i-1} \times D_{n-2x_i}}^{D_n}\left(\mathrm{comp}_{1/2}(2x_i,\frac{1}{2})^h \otimes J(\overline{M}_{1/2}, \overline{N}_{1/2})\right)$$
(\text{resp.} $\mathrm{Ind}_{A_{2x_i-1} \times D_{n-2x_i}}^{D_n}\left(\mathrm{comp}_{1/2}(2x_i,\frac{1}{2}) \otimes J(\overline{M}_{1/2}^h, \overline{N}_{1/2}^h)\right)$), where $\mathrm{comp}_{1/2}(2x_i,\frac{1}{2})$ is the twisted one-sided $t$-complementary series \eqref{eq-compat}, and
$$\overline{N}_{1/2} \longleftrightarrow \begin{pmatrix} x_1  & \dots & x_{i-1} & x_{i+1} & \dots & x_k \\ y_1 & \dots & y_{i-1} & y_{i+1} & \dots & y_k \end{pmatrix}.$$

By induction in stages, $\pi_{r+1}$ is a subquotient of
\begin{equation} \label{eq-unit2}
\mathrm{Ind}_{A_{4x_i-1} \times D_{2n-4x_i}}^{D_{2n}}\left( \mathrm{Ind}_{A_{2x_i-1} \times A_{2x_i-1}}^{A_{4x_i-1}}\left(\mathrm{comp}_{1/2}(2x_i,\frac{1}{2}) \otimes \mathrm{comp}_{1/2}(2x_i,\frac{1}{2})^h\right) \otimes \pi^- \right)
\end{equation}
with $\pi^- := J_{D_{2n-4x_i}}\left(J(\overline{M}_{1/2} \cup \overline{M}_{1/2}^h, \overline{N}_{1/2} \cup \overline{N}_{1/2}^h\right)$
satisfying the induction hypothesis. Therefore the theorem holds for $\pi_{r+1}$ by induction in stages.


\medskip
The case for $y_i + 1 = x_{i+1}$ is identical --  in such a case, take $\mathrm{comp}_{1/2}(x_{i+1}+y_i,\frac{1}{2})$ instead of $\mathrm{comp}_{1/2}(2x_i,\frac{1}{2})$ in \eqref{eq-unit2}, and
$$\overline{N}_{1/2} \longleftrightarrow \begin{pmatrix} x_1  & \dots & x_{i-1} & x_i & x_{i+2} & \dots & x_k \\ y_1 & \dots & y_{i-1} & y_{i+1} &  y_{i+2} & \dots & y_k \end{pmatrix}$$
instead. So the theorem follows. \end{proof}

Combined with the result in Appendix \ref{sec-unipotent}, this implies that all representations in Theorem \ref{thm-main}
are subquotients of modules inducing from some (twisted) Stein's complementary series tensored with a unipotent representation.
Also, by taking Equations \eqref{eq:inclusion} and \eqref{eq-bv} into account, we have verified the conjectured structure
of all genuine representations in $\widehat{G}$ for $Spin(2n,\mathbb{C})$ described in the introduction.


\section{Intertwining operators} \label{sec-in}
In view of Section \ref{sec-unitary}, one only needs to show that the irreducible modules $J(\mu,\nu) = J(M_{1/2} \cup M_{1/2}^h, N_{1/2} \cup N_{1/2}^h)$ whose parameters does not satisfy the hypothesis of Theorem \ref{thm-main}
is not unitary. By the discussions in Section \ref{sec-intertwine1}, $J(\mu,\nu)$ can be seen as the image of the long intertwining operator. In this section, we will study certain intertwining operators which will be used in Section \ref{sec-nonunit} extensively to determine the non-unitarity of certain $J(\mu,\nu)$'s. We refer the reader to \cite[III.3]{D75} as well as \cite{B10} for some basic properties of intertwining operators on complex or more generally all quasi-split reductive groups.

\subsection{Basics of intertwining operators} \label{sec-inter}
Given $\nu\in \mathfrak{a}^*$, we define
\[\Delta^-(\nu)=\{\alpha\in \Delta^+(\mathfrak{g},\mathfrak{h})\ |\ \langle\nu,\alpha\rangle<0 \}.\]
For $w\in W(\mathfrak{g},\mathfrak{h})$, there exists an intertwining operator $i_{w}(\mu,\nu)$ such that
\[\begin{aligned}i_{w}(\mu,\nu) : X(\mu,\nu)& \to  X(w\mu,w\nu)\\  f & \mapsto  \int_{w^{-1}Nw\cap \overline{N}} f(\overline{n}w^{-1}g) d\overline{n},\end{aligned}\]
where $N$ (resp. $\overline{N}$)  is the unipotent subgroup corresponding to the set of positive (resp. negative) roots.
For convenience, we also denote it by $i_{w}$. When $\Delta^-(\nu)\subset \Delta^-(w\nu)$, the integral defining $i_{w}(\mu,\nu)$ converges absolutely. Moreover, if $w=w_1w_2$ with $l(w)=l(w_1)+l(w_2)$, then
\[i_{w}(\mu,\nu)=i_{w_1}(w_2\mu,w_2\nu)\circ i_{w_2}(\mu,\nu).\]

For any $K$-module $\Gamma$, we study how the {\bf $\Gamma$-isotropic space} $\mathrm{Hom}_K[\Gamma, X(\mu,\nu)]$ maps to the $\Gamma$-isotropic space  $\mathrm{Hom}_K[\Gamma,X(w\mu,w\nu)]$ under $i_{w}(\mu,\nu)$ -- by Frobenius reciprocity, we know that the space $\mathrm{Hom}_K[\Gamma,X(\mu,\nu)]$ is isomorphic to the $\mu\text{-weight space of}\ \Gamma,$ denoted by $\Gamma^{\mu}$.
Then the intertwining operator restricting to the $\Gamma$-isotropic space is given by
\begin{align*}\phi_{w}(\mu,\nu): \Gamma^{\mu} &  \to \Gamma^{w\mu}\\
     x &  \mapsto \int_{w^{-1}Nw\cap \overline{N}} b(\overline{n}w^{-1})^{\nu+2\rho} \Gamma(k(\overline{n}w^{-1})^{-1})x d\overline{n},\end{align*}
 where $b(\cdot), k(\cdot)$ means taking the component of the decomposition $G=BK$. We can replace $\phi_w(\mu,\nu)$ by  $(\int b(\overline{n}w^{-1})^{\nu+2\rho} d\overline{n})^{-1} \cdot \phi_w(\mu,\nu)$, so $\phi_w(\mu,\nu)$ is identity on $\Gamma^{\mu}$ when $\Gamma = \mathcal{V}_{\{\mu\}}$ is the $K$-type with extremal weight $\mu$.

Suppose $w = s_{\alpha}$ is a simple reflection, where $\alpha$ is a simple root.  Let $\mathfrak{g}_{\alpha}$ be the $\alpha$-root space of $\mathfrak{g}$. Let $H_{\alpha}\in \mathfrak{h},E_{\alpha}\in \mathfrak{g}_{\alpha}, \theta E_{\alpha}\in \mathfrak{g}_{-\alpha}$ be the $\mathfrak{sl}(2,\mathbb{C})$-triple (here $\theta$ is the restriction of the Cartan involution of $\mathfrak{g}$), i.e.
\[[H_{\alpha}, E_{\alpha}]=2E_{\alpha},\quad \quad [H_{\alpha}, \theta E_{\alpha}]=-2\theta E_{\alpha},\quad \quad
[E_{\alpha},\theta E_{\alpha}]=-H_{\alpha}.\]

Let $\mathfrak{g}(\alpha)=\mathbb{C}\{H_{\alpha},E_{\alpha},\theta E_{\alpha}\}$, and let $G(\alpha)$ be the connected subgroup of $G$ with Lie subalgebra $\mathfrak{g}(\alpha)$. Then 
$i_{s_{\alpha}}(\mu,\nu)$ can be computed over $G(\alpha)$. In particular,
 and $i_{s_{\alpha}}(\mu,\nu)|_{\Gamma\text{-isotropic space}}$ is determined by the action of the maximal compact subgroup $U(\alpha)$ of $G(\alpha)$ on $\Gamma$. So $\phi_{s_{\alpha}}(\mu,\nu)|_{\Gamma^{\mu}}$ can also be computed over $G(\alpha)$, which is exactly the intertwining operator over the group $SL(2,\mathbb{C})$, see \cite[III.3.6]{D75}.
For latter calculation, define $t_{\alpha}\in U(\alpha)$ by
\begin{equation}\label{t_a}
t_{\alpha} := \mathrm{exp}\big(\frac{\pi }{2} (E_{\alpha} + \theta E_{\alpha})\big).
\end{equation}

We make the following assumption, which is satisfied in the cases that we are studying later:
\begin{assumption}
Every irreducible component of $\Gamma|_{U(\alpha)}$ has dimension $\leq 4$ for any simple root $\alpha$.
\end{assumption}
There are two possibilities under the above assumption,  and the corresponding $\phi_{s_{\alpha}}$ (normalized to be identity on the lowest $K$-type) can be computed using \cite[III.3.6]{D75} as follows: 

\medskip
\noindent {\bf Case (i).} Assume that $\langle\mu,\alpha\rangle=0$, then every irreducible component of $\Gamma|_{U(\alpha)}$ has dimension $1$ or $3$, and $\phi_{s_{\alpha}}: \Gamma^{\mu}\to \Gamma^{\mu}$ is given by
\begin{equation}\label{phi_++}\phi_{s_{\alpha}}(x)=\left\{\begin{array}{ll} x, & t_{\alpha}x=x,\\
  \frac{2-\langle \check{\alpha},\nu\rangle}{2+\langle \check{\alpha},\nu\rangle} x, & t_{\alpha}x=-x.\end{array}\right.\end{equation} 
Here, $\check{\alpha}$ is the coroot corresponding to $\alpha$.
 Moreover,  the intertwining operator on the $\Gamma$-isotropic space is determined by the action of the Weyl group on the zero weight space of $\Gamma$.

\medskip
\noindent {\bf Case (ii).} Assume that $\langle\mu,\alpha\rangle\neq 0$, then every irreducible components of $\Gamma|_{U(\alpha)}$ has dimension $2$ or $4$. Assume $x\in \Gamma^{\mu}$ in the irreducible $U(\alpha)$-submodule $V\subset \Gamma$, then
\begin{equation}\label{phi_+-}\phi_{s_{\alpha}}(x)=\left\{\begin{array}{ll}
 t_{\alpha}x, & \dim V=2, \\
     \frac{-3+\langle \check{\alpha},\nu\rangle}{3+\langle \check{\alpha},\nu\rangle} t_{\alpha}x, & \dim V=4.\end{array}\right.\end{equation}
 
\subsection{Type A intertwining operators} \label{sec-interd}
For the rest of this paper, we use the shorthand of the $(\mu,\nu)$-parameters
by
$$(\nu_1^{\epsilon_1}, \dots, \nu_n^{\epsilon_n}) := \left(\mu = (\frac{\epsilon_1}{2},\dots, \frac{\epsilon_n}{2});\ \nu = (\nu_1, \dots, \nu_n)\right).$$
Also, write $\mathcal{A}_{\bullet}$ denotes the irreducible $\widetilde{U(n)}$-module with highest weight $\bullet$.

\medskip
The following results generalize \eqref{phi_++} and \eqref{phi_+-}. We omit the proof which are some direct computations by induction on $n$.
\begin{lemma}\label{int_of_1/2_-1/2} 
Let $\nu=(\nu_1 < \dots < \nu_{n-1};x)$ with $x \in \mathbb{R}$, $\nu_{k+1}-\nu_k=2$ for all $k$, and $\delta \in \{+,-\}$. 

\noindent (a) Consider the intertwining operator
\[i_w: \mathrm{Ind}_{\widetilde{GL(n-1) \times GL(1)}}^{\widetilde{GL(n)}}\big(J(\nu_1^{\delta}, \dots, \nu_{n-1}^{\delta}) \otimes J(x^{\delta})\big)\longrightarrow \mathrm{Ind}_{\widetilde{GL(1) \times GL(n-1)}}^{\widetilde{GL(n)}}\big( J(x^{\delta}) \otimes J(\nu_1^{\delta}, \dots, \nu_{n-1}^{\delta})\big),\]
normalized such that it is equal to $1$ on the lowest $K$-type isotypic space $\mathcal{A}_{(\frac{1}{2},\frac{1}{2},\dots,\frac{1}{2},\frac{1}{2})}$. 
Then the $\mathcal{A}_{(\frac{3}{2},\frac{1}{2},\dots,\frac{1}{2},-\frac{1}{2})}$-isotropic subspace of the induced modules is one-dimensional, and $i_w$ is equal to the scalar
$$\frac{-2+\nu_1-x}{2+\nu_{n-1}-x}$$ 
upon restricting to the subspaces.

\noindent (b) Consider the intertwining operator
\[i_w: \mathrm{Ind}_{\widetilde{GL(n-1) \times GL(1)}}^{\widetilde{GL(n)}}\big(J(\nu_1^{\delta}, \dots, \nu_{n-1}^{\delta}) \otimes J(x^{-\delta})\big)\longrightarrow \mathrm{Ind}_{\widetilde{GL(1) \times GL(n-1)}}^{\widetilde{GL(n)}}\big( J(x^{-\delta}) \otimes J(\nu_1^{\delta}, \dots, \nu_{n-1}^{\delta})\big),\]
normalized such that it is equal to $1$ on the lowest $K$-type isotypic space $\mathcal{A}_{(\frac{1}{2},\frac{1}{2},\dots,\frac{1}{2},\frac{-1}{2})}$.
Then the $\mathcal{A}_{(\frac{3}{2},\frac{1}{2},\dots,\frac{1}{2},-\frac{3}{2})}$-isotropic subspace of the induced modules is one-dimensional, and $i_w$ is equal to the scalar
$$\frac{-3+\nu_1-x}{3+\nu_{n-1}-x}$$ 
upon restricting to the subspaces.
\end{lemma}

\subsection{Injectivity of certain intertwining operators} \label{sec-inject}
We now apply the results in the previous sections to study certain intertwining operators
appearing in Section \ref{sec-nonunit}.

\begin{proposition}\label{inj_v_x}
Let $p+q=n$, and $\nu = (\nu_1 < \dots < \nu_p; \xi_1 \leq \dots \leq \xi_q)$ be such that $\nu_{k+1}-\nu_k = 2$ for all $k$. For $\delta, \epsilon_l \in \{+,-\}$, consider the intertwining operator
$$\iota: \mathrm{Ind}_{\widetilde{GL(p) \times GL(q)}}^{\widetilde{GL(n)}}\big(J(\nu_1^{\delta}, \dots, \nu_p^{\delta}) \otimes J(\xi_1^{\epsilon_1}, \dots \xi_q^{\epsilon_q})\big)\longrightarrow \mathrm{Ind}_{\widetilde{GL(q) \times GL(p)}}^{\widetilde{GL(n)}}\big(J(\xi_1^{\epsilon_1}, \dots \xi_q^{\epsilon_q}) \otimes J(\nu_1^{\delta}, \dots, \nu_p^{\delta})\big).$$
Suppose
$\xi_{l}\neq \nu_1-2$ or $\nu_p+2$ (resp. $\nu_1-3$ or $\nu_p+3$) if $\epsilon_{l}= \delta$ (resp. $\epsilon_l = -\delta$) for all $1\leq l \leq q$, then $\iota$ is well-defined and injective over the $\mathcal{A}_{\eta}$-isotropic subspace for all $\eta = (\eta_1, \dots, \eta_n)$ with $|\eta_i| \leq \frac{3}{2}$.
\end{proposition}

\begin{proof}
We prove by induction on $q$. The base case $q=0$ is trivially true. Assume that it is true for $q-1$, then for the case $q$, write $\iota$ as the composition of the following $\iota''\circ \iota'$:
\begin{align*}
X((\nu_1^{\delta} < \dots < \nu_p^{\delta}; \xi_{1}^{\epsilon_1} \leq \dots \leq \xi_q^{\epsilon_q}))
&\xrightarrow{\iota'}
X(\xi_1^{\epsilon_1}; \nu_1^{\delta} < \dots < \nu_{p}^{\delta}; \xi_{2}^{\epsilon_2} \leq \dots \leq \xi_{q}^{\epsilon_q})\\ &\xrightarrow{\iota''} X(\xi_1^{\epsilon_1}\leq \xi_{2}^{\epsilon_2} \leq \dots \leq \xi_{q}^{\epsilon_q}; \nu_1^{\delta} < \dots < \nu_{p}^{\delta}).
\end{align*}

One can restrict the domain of the above intertwining operator into 
\begin{align*}
&\mathrm{Ind}_{\widetilde{GL(p) \times GL(1) \times GL(q-1)}}^{\widetilde{GL(n)}}\big(J(\nu_1^{\delta}, \dots, \nu_p^{\delta}) \otimes J(\xi_1^{\epsilon_1}) \otimes J(\xi_2^{\epsilon_2} \dots \xi_q^{\epsilon_q})\big) \\\xrightarrow{\iota'}\  
&\mathrm{Ind}_{\widetilde{GL(1) \times GL(p) \times GL(q-1)}}^{\widetilde{GL(n)}}\big(J(\xi_1^{\epsilon_1}) \otimes J(\nu_1^{\delta}, \dots, \nu_p^{\delta}) \otimes J(\xi_2^{\epsilon_2} \dots \xi_q^{\epsilon_q})\big) \\
\xrightarrow{\iota''}\  
&\mathrm{Ind}_{\widetilde{GL(1) \times GL(q-1) \times GL(p)}}^{\widetilde{GL(n)}}\big(J(\xi_1^{\epsilon_1}) \otimes J(\xi_2^{\epsilon_2} \dots \xi_q^{\epsilon_q}) \otimes J(\nu_1^{\delta}, \dots, \nu_p^{\delta}) \big)
\end{align*}
Note that $\mathrm{Ind}_{\widetilde{GL(p) \times GL(q)}}^{\widetilde{GL(n)}}\big(J(\nu_1^{\delta}, \dots, \nu_p^{\delta}) \otimes J(\xi_1^{\epsilon_1}, \dots \xi_q^{\epsilon_q})\big)$ is a submodule of the domain in the above chain of operators (due to the fact that the $\xi$-parameters are anti-dominant), so the injectivity of $\iota$ can be obtained from that of the above operators.

By our hypothesis on $\xi$, Lemma \ref{int_of_1/2_-1/2} and induction in stages imply that $\iota'$ is well-defined on the $\mathcal{A}_{\eta}$-isotropic spaces,  for $\xi_1\neq \nu_p+2$ (resp. $\nu_p+3$) if $\epsilon_1=\delta$ (resp. $\epsilon_1=-\delta$); and injective for $\xi_1\neq \nu_1-2$ (resp. $\nu_1-3$) if $\epsilon_1=\delta$, and by induction we know that $\iota''$ is also well defined and injective over the $\mathcal{A}_{\eta}$-isotropic spaces. Hence, the lemma also holds for $\iota=\iota''\circ \iota'$.
\end{proof}

\begin{proposition}\label{type_A_inj}
Let $p+q=n$, and $\nu = (\xi_1 \leq \dots \leq \xi_q; \nu_1 < \dots < \nu_p)$ be such that $\nu_{k+1}-\nu_k = 2$ for all $k$. 
Let $z \in S_n$ be such that $z\nu$ is non-increasing. Consider
 the intertwining operator
$$\iota := \iota_z: X\left(\xi_1^{\epsilon_1}, \dots, \xi_q^{\epsilon_q}; \nu_1^{\delta}, \dots, \nu_p^{\delta}\right) \longrightarrow X\left(z(\xi_1^{\epsilon_1}, \dots, \xi_q^{\epsilon_q}; \nu_1^{\delta}, \dots, \nu_p^{\delta})\right).$$
Suppose $\xi_{l}\neq \nu_{p}+2$ (resp. $\nu_{p}+3$) if $\epsilon_{l}= \delta$ (resp. $\epsilon_l = -\delta$) for all $1\leq l \leq q$, then the restricted map 
$$\iota|_{\mathrm{Ind}_{\widetilde{GL(q) \times GL(p)}}^{\widetilde{GL(n)}}\left(J(\xi_1^{\epsilon_1}, \dots, \xi_q^{\epsilon_q}) \otimes J(\nu_1^{\delta},\dots,\nu_p^{\delta})\right)}$$
is injective over the $\mathcal{A}_{\eta}$-isotropic subspace for all $\eta = (\eta_1, \dots, \eta_n)$ with $|\eta_i| \leq \frac{3}{2}$.
\end{proposition}

\begin{proof}
We prove by induction on $q$, which is trivial for the base case $q=0$. Assume that the statement is true for $q-1$.
If $\xi_{q}\leq \nu_1$, then $(\xi_1,\dots,\xi_{q},\nu_1,\dots,\nu_{p})$ is an non-increasing sequence. So $\iota = \iota_{e}$ is the identity map and the lemma is clearly true. On the other hand, if $\xi_{q} \geq \nu_{p}$, then $\iota$ is the composition of the following $\iota'$ and $\iota''$ :
\begin{align*}
&\ \mathrm{Ind}_{\widetilde{GL(q) \times GL(p)}}^{\widetilde{GL(n)}}\left(J(\xi_1^{\epsilon_1} \leq \dots \leq \xi_{q}^{\epsilon_q})\otimes J(\nu_{1}^{\delta} < \dots <\nu_{p}^{\delta})\right)\\ \xrightarrow{\iota'}&\   
\mathrm{Ind}_{\widetilde{GL(q-1) \times GL(p)\times GL(1)}}^{\widetilde{GL(n)}}\left(J(\xi_1^{\epsilon_1} \leq \dots \leq \xi_{q-1}^{\epsilon_{q-1}})\otimes J(\nu_1^{\delta}, \dots, \nu_{p}^{\delta})\otimes J(\xi_{q}^{\epsilon_q})\right) \\ \xrightarrow{\iota''} &\
X\left(z(\xi_1^{\epsilon_1}, \dots , \xi_q^{\epsilon_q}; \nu_1^{\delta} , \dots, \nu_p^{\delta})\right).
\end{align*}
By hypothesis, $\xi_{q} \neq \nu_{p}+2$ (resp. $\nu_{p}+3$) if $
\epsilon_{q}=\delta$ (resp. $-\delta$). So
one can apply Lemma \ref{int_of_1/2_-1/2} and the `submodule of a parabolically induced module' argument in the proof of Proposition \ref{inj_v_x} to conclude that $\iota'$ is injective. Furthermore, we also know from induction that $\iota''$ is also injective over the image of $\iota'$.
Hence, the proposition holds for $\iota=\iota''\circ\iota'$.

If $\nu_{k}\leq \xi_{q}< \nu_{k+1}$ for some $1\leq k\leq p-1$, then $\xi_{1},\dots,\xi_{q-1}\leq \xi_{q}< \nu_{k+1}=\nu_{k}+2<\nu_{k}+3$. Hence, $z$ fix the last $p-k$ entries of $\nu$. By the above case, the proposition holds for   $(\xi_1^{\epsilon_1}, \dots, \xi_q^{\epsilon_q}; \nu_1^{\delta}, \dots, \nu_{k}^{\delta})$. Using parabolically induction, one gets the proposition for $(\xi_1^{\epsilon_1}, \dots, \xi_q^{\epsilon_q}; \nu_1^{\delta}, \dots, \nu_{p}^{\delta})$.
\end{proof}

\begin{remark}\label{similar_of_int_lemma}
In Section \ref{sec-nonunit}, we will apply the above results to $G = Spin(2n,\mathbb{C})$. In such a case, there are two choices of Levi subalgebras of type $A_{n-1}$, where the `other' type $A_{n-1}$ Lie algebra $\overline{\mathfrak{gl}}(n,\mathbb{C})$ has simple roots
$$\{e_i-e_{i+1}\ |\ 1\leq i\leq n-2\} \cup \{e_{n-1}+e_n\}.$$ 
\end{remark}

\subsection{Intertwining operator for nonreduced expressions} We conclude this section by the following trick, which will be used extensively in the next section.
\begin{proposition}\label{non_red_of_phi}
Let $w$ be an element in $W(D_n)$. Write
\begin{equation}\label{non-reduce_expand}w=s_{\alpha_m}\dots s_{\alpha_2}s_{\alpha_1}=s_{\beta_{m'}}\dots s_{\beta_2}s_{\beta_1},\end{equation}
for some simple reflections $s_{\alpha_{\bullet}}$ and $s_{\beta_{\bullet}}$, which are not assumed to be a reduced expression, so that the intertwining operators $\phi_{s_{\alpha_{\bullet}}}$ or $\phi_{s_{\beta_{\bullet}}}$ may not be well-defined.

Suppose there exists a subspace $V'$ of $\Gamma = \mathcal{V}_{\eta_i}$ such that
$\phi_{s_{\alpha_k}}$ (resp. $\phi_{s_{\beta_k}}$) is well-defined over $\phi_{s_{\alpha_{k-1}}\dots s_{\alpha_2} s_{\alpha_1}}(V')$ (resp. $\phi_{s_{\beta_{k-1}}\dots s_{\beta_2} s_{\beta_1}}(V')$)  for all $k$, then we have
\begin{equation}\label{expand_of_phi}
\phi_{s_{\alpha_m}}\circ\dots\circ\phi_{s_{\alpha_2}}\circ \phi_{s_{\alpha_1}}=\phi_{s_{\beta_{m'}}}\circ\dots\circ\phi_{s_{\beta_2}}\circ \phi_{s_{\beta_1}}
\end{equation}
over the subspace $V'$.
\end{proposition}



\begin{proof}
Firstly, assume that $\nu$ satisfies
\begin{equation}\label{alpha_nu_neq_2or3}
\langle \check{\alpha},\nu\rangle\neq \pm 2,\ \text{or} \pm 3,\ \forall\ \alpha\in \Delta^+(\mathfrak{g},\mathfrak{h}).
\end{equation}
Then for any root $\alpha$,
$\phi_{s_{\alpha}}$ is well-defined and invertible over the corresponding weight space of $\mathcal{V}_{\eta_i}$ by the formulas in Section \ref{sec-inter}. Notice that \eqref{expand_of_phi} holds when the two expansions in \eqref{non-reduce_expand} are both reduced. We can prove \eqref{expand_of_phi} by induction on $(m,m')$ as follows. When $(m,m')=(0,0)$, it is trivial. Assume it is true for $(<m,\leq m')$ and $(\leq m,<m')$. By the formulas in Section \ref{sec-inter} and the results of $SL(2)$, we get that $\phi_{s_{\alpha}}\circ\phi_{s_{\alpha}}=1$. For example, in the case (i) of Section \ref{sec-inter}, assume that $s_{\alpha}x=-x$, then $\phi_{s_{\alpha}}(x)=\frac{2-\langle\check{\alpha},\nu\rangle}{2+\langle\check{\alpha},\nu\rangle}x$ and
\[\phi_{s_{\alpha}}\big(\phi_{s_{\alpha}}(x)\big)=\frac{2-\langle\check{\alpha},s_{\alpha}\nu\rangle}{2+\langle\check{\alpha},s_{\alpha}\nu\rangle}\cdot \left(\frac{2-\langle\check{\alpha},\nu\rangle}{2+\langle\check{\alpha},\nu\rangle}x\right)=x.\]
Assume that $s_{\alpha_m}\dots s_{\alpha_2}s_{\alpha_1}$ is not reduced. Let $1\leq k\leq m$ be the smallest number such that the length of $s_{\alpha_{k}}\dots s_{\alpha_2} s_{\alpha_1}$ is less than $s_{\alpha_{k-1}}\dots s_{\alpha_2} s_{\alpha_1}$, so $s_{\alpha_{k}}\dots s_{\alpha_2} s_{\alpha_1}=s_{\alpha_{k-2}'}\dots s_{\alpha_2'} s_{\alpha_1'}$ for some simple reflections $s_{\alpha'_{\bullet}}$. We have
\[\begin{aligned}
& \phi_{s_{\alpha_m}}\circ\dots\circ\phi_{s_{\alpha_2}}\circ \phi_{s_{\alpha_1}}
\\=&\ \phi_{s_{\alpha_m}}\circ\dots\circ\phi_{s_{\alpha_{k+1}}}\circ \phi_{s_{\alpha_{k-2}'}}\circ\dots\circ\phi_{s_{\alpha_2'}}\circ\phi_{s_{\alpha_1'}}
\\=&\ \phi_{s_{\beta_{m'}}}\circ\dots\circ\phi_{s_{\beta_2}}\circ \phi_{s_{\beta_1}},\end{aligned}\]
where the first equality is due to $\phi_{s_{\alpha_k}}\circ\phi_{s_{\alpha_k}}=1$ and \eqref{expand_of_phi} holds for reduced expansions, and the second equality is due to the induction for $(m-2,m')$. Therefore, the proposition holds for all $\nu$ satisfying condition \eqref{alpha_nu_neq_2or3}.

As for general $\nu$ such that \eqref{expand_of_phi} is well-defined, notice that the formulas of $\phi_{s_{\alpha}}$ in Section \ref{sec-inter} are rational functions over $\nu$. So \eqref{expand_of_phi} holds for these values by continuity arguments, and the proposition follows.
%
%
%
%
%
\end{proof}

\section{Proof of non-unitarity} \label{sec-nonunit}
In this section, we will prove that the Hermitian modules $\pi = J(M_{1/2} \cup M_{1/2}^h,\ N_{1/2} \cup N_{1/2}^h)$ not appearing in Theorem \ref{thm-main} are not unitary. As mentioned in the previous section, we will understand $\pi$ by studying the long intertwining operator on the principal series representation. More explicitly, let
$N_{1/2} \longleftrightarrow \begin{pmatrix}x_1&\dots&x_k
\\ y_1&\dots&y_k \end{pmatrix}$, and $w \in W(D_n)$ (resp. $w_h \in W(D_n)$) be such that $\phi := wN_{1/2}$ (resp. $\phi^h := w_hN_{1/2}^h$) is dominant. We also fix $\theta := wM_{1/2}$ and $\theta^h := w_hM_{1/2}^h$ throughout this section, so that
the image of the long intertwining operator
$$\iota_{w_0}: X(\theta,\phi) \longrightarrow X(w_0\theta,w_0\phi)  \quad \quad (\text{resp.}\ \iota_{w_0}: X(\theta^h,\phi^h) \longrightarrow X(w_0\theta^h,w_0\phi^h)).$$
is equal to $J(\theta,\phi) = J(M_{1/2},N_{1/2})$ (resp. $J(\theta^h,\phi^h) = J(M_{1/2}^h,N_{1/2}^h)$).

\begin{example}
    Let $N_{1/2} = \left(\frac{9}{2}, \frac{5}{2}, \frac{1}{2}, \frac{-3}{2}, \frac{-7}{2}; \frac{5}{2},\frac{1}{2} \right) \longleftrightarrow \begin{pmatrix} 3 & 2 \\ 2 & 0 \end{pmatrix}$. Then 
    {\small \begin{align*} 
        (M_{1/2}; N_{1/2}) = \left(\frac{9}{2}^+, \frac{5}{2}^+, \frac{1}{2}^+, \frac{-3}{2}^+, \frac{-7}{2}^+, \frac{5}{2}^+,\frac{1}{2}^+ \right),\ (M_{1/2}^h; N_{1/2}^h) = \left(\frac{-9}{2}^+, \frac{-5}{2}^+, \frac{-1}{2}^+, \frac{3}{2}^+, \frac{7}{2}^+, \frac{-5}{2}^+, \frac{-1}{2}^+ \right)
    \end{align*}}
and
\begin{align*}
        (\theta; \phi) = \left(\frac{9}{2}^+, \frac{7}{2}^-, \frac{5}{2}^+, \frac{5}{2}^+, \frac{3}{2}^-, \frac{1}{2}^+, \frac{1}{2}^+ \right), \
    (\theta^h; \phi^h) = \left(\frac{9}{2}^-, \frac{7}{2}^+, \frac{5}{2}^-, \frac{5}{2}^-, \frac{3}{2}^+, \frac{1}{2}^-, \frac{-1}{2}^+ \right)
    \end{align*}
\end{example}

In the next two subsections, we will study two basic occasions of non-unitarity (``Case I'' and ``Case II'') for some specific choices of $N_{1/2}$. In both cases, we will split the long intertwining operator $\iota_{w_0}$ into 
$$X(\theta,\phi) \xrightarrow{\iota_{w_1}} X(w_1\theta,w_1\phi) \xrightarrow{\iota_{w_2}} X(w_2w_1\theta,w_2w_1\phi) \xrightarrow{\iota_{w_3}} X(w_0\theta,w_0\phi)$$ 
for some choices of $w_1$, $w_2$, $w_3$ such that $w_3w_2w_1 = w_0$, and study the injectivity of restricted map $\iota_{w_{3}}|_{\mathrm{im}(\iota_{w_2})}$ (as well as their Hermitian dual counterparts).

\begin{remark} \label{rmk-inclusion}
In order to study the injectivity of $\iota_{w_{3}}|_{\mathrm{im}(\iota_{w_2})}$, one write $w_2=\tau_m\dots\tau_{1}$ for some choices of $\tau_i$'s, and study the injectivity of $\iota_{\tau_i}|_{\mathrm{im}(\iota_{\tau_{i-1}}\circ\dots\iota_{\tau_{1}}\circ\iota_{w_2})}$. Assume that  $\iota_{\tau_i}$ is a composition of some type A intertwining operators, we apply the same strategy as in the proof of Proposition \ref{inj_v_x} and Proposition \ref{type_A_inj} by embedding 
$\mathrm{im}(\iota_{\tau_{i-1}}\circ\dots\iota_{\tau_{1}}\circ\iota_{w_2}) \subseteq I_i$ into an induced module 
\[I_i := \mathrm{Ind}_{M_i}^G(J_{M_i}(\tau_{i-1}\dots\tau_{1}w_2w_1\theta,\tau_{i-1}\dots\tau_{1}w_2w_1\phi))\] 
for some Levi subgroup $M_i := \cdots \times A_{p-1} \times A_{q-1} \times \cdots$ of $G$. In such a case, the results in Section \ref{sec-interd} and \ref{sec-inject} can be applied to to conclude that $\iota_{\tau_{i}}|_{I_i}$ (and hence $\iota_{\tau_{i}}|_{\mathrm{im}(\iota_{\tau_{i-1}}\circ\dots\iota_{\tau_{1}}\circ\iota_{w_2})}$) is injective on certain spin-relevant $K$-types $\mathcal{V}_{\eta}$.

Assume that $\tau_{i-1}\dots\tau_{1}w_2w_1\phi$ is anti-dominant of $M_i$-type. To use the above strategy, one needs to check that $\mathrm{im}(\iota_{\tau_{i-1}}\circ\dots\iota_{\tau_{1}}\circ\iota_{w_2}) \subseteq I_i$.  Let $w_{M_i}$ be the longest element of the weyl group of $M_i$, since we focus on the spin-relevant $K$-types, it suffices to construct a well-defined intertwining operator
\[X(w_1\theta,w_1\phi)\to X\left(w_{M_i}^{-1}(\tau_{i-1}\dots\tau_{1}w_2w_1\phi),w_{M_i}^{-1}(\tau_{i-1}\dots\tau_{1}w_2w_1\phi)\right)\]
over all spin-relevant $K$-types. It would be clear that such an intertwining operator exists when this strategy appears in the proof of the next two subsections. 
\end{remark}

\subsection{Base Case I} \label{sec-u}
\begin{proposition}\label{single_mult}
Let $n = a+b$, and $N_{1/2} \longleftrightarrow \begin{pmatrix}a \\ b \end{pmatrix}$ with $b > a$. 
Then the representations $\pi = J(M_{1/2} \cup M_{1/2}^h,\ N_{1/2} \cup N_{1/2}^h)$ is the lowest $K$-type subquotient of the induced representation $\mathrm{Ind}_{A_{2n-1}}^{D_{2n}} \big(J_{A_{2n-1}}(M_{1/2} \cup M_{1/2}^h,\ N_{1/2} \cup N_{1/2}^h)\big),$
both having the same multiplicities of $\mathcal{V}_{\eta(i)}$ for $0 \leq i\leq n$.
\end{proposition}

\begin{proof}
The proposition holds more generally for $b \geq a-1$, which we will prove below. Consider the decomposition of the long intertwining operator $(\theta, \phi) \longrightarrow (w_0\theta,w_0\phi)$:
\begin{equation} \label{eq-lema1}
\begin{aligned}
  &\left((2b-\frac{1}{2})^-, (2b-\frac{5}{2})^-, \dots, (2a-\frac{1}{2})^-, (2a-\frac{3}{2})^+, (2a-\frac{5}{2})^-,\dots,\frac{3}{2}^-,(\frac{\epsilon(b)}{2})^{\epsilon(b)}\right) \\
 \xrightarrow{\kappa}
 &\left((2a-\frac{3}{2})^+, \dots, \frac{5}{2}^+, \frac{1}{2}^+, \frac{-3}{2}^+, \dots, (\frac{1}{2}-2b)^+\right) \\
\xrightarrow{\zeta} 
&\left((\frac{1}{2}-2b)^+, \dots, \frac{-3}{2}^+, \frac{1}{2}^+, \frac{5}{2}^+, \dots, (2a-\frac{3}{2})^+\right) \\ 
\xrightarrow{\omega}  & \left((\frac{1}{2}-2b)^+, (\frac{5}{2}-2b)^+, \dots, (\frac{1}{2}-2a)^+, (\frac{3}{2}-2a)^-, (\frac{5}{2}-2a)^+,\dots,-\frac{3}{2}^+,(\frac{\epsilon(a+1)}{2})^{\epsilon(a+1)}\right)
\end{aligned}
\end{equation} 
where $\epsilon(k) := (-1)^k$, and that on the Hermitian dual $(\theta^h, \phi^h) \longrightarrow (w_0\theta^h, w_0\phi^h)$:
{\allowdisplaybreaks
\begin{equation} \label{eq-lema12}
\begin{aligned}
  &\left((2b-\frac{1}{2})^+, (2b-\frac{5}{2})^+, \dots, (2a-\frac{1}{2})^+, (2a-\frac{3}{2})^-, (2a-\frac{5}{2})^+,\dots,\frac{3}{2}^+,(\frac{\epsilon(a)}{2})^{\epsilon(a+1)}\right) \\
 \xrightarrow{\kappa^h}
 &\left((2b-\frac{1}{2})^+, \dots, \frac{3}{2}^+, \frac{-1}{2}^+, \frac{-5}{2}^+,\dots, (\frac{3}{2}-2a)^+\right) \\
\xrightarrow{\zeta^h} 
&\left((\frac{3}{2}-2a)^+, \dots, \frac{-5}{2}^+, \frac{-1}{2}^+, \frac{3}{2}^+, \dots, (2b-\frac{1}{2})^+\right) \\ 
\xrightarrow{\omega^h}  & \left((\frac{1}{2}-2b)^-, (\frac{5}{2}-2b)^-, \dots, (\frac{1}{2}-2a)^-, (\frac{3}{2}-2a)^+, (\frac{5}{2}-2a)^-,\dots,\frac{-3}{2}^-,(\frac{\epsilon(b+1)}{2})^{\epsilon(b)}\right) \end{aligned}\end{equation}}
Note that $\zeta: X\left((2a-\frac{3}{2})^+,\dots, (\frac{1}{2} - 2b)^+\right) \to X\left((\frac{1}{2}-2b)^+,\dots, (2a-\frac{3}{2} )^+\right)$, is the longest type $A_{n-1}$ intertwining operator, so its image is precisely $\mathrm{Ind}_{A_{n-1}}^{D_n}(J_{A_{n-1}}(M_{1/2},N_{1/2}))$. Similarly, the image of $\zeta^h$ 
is equal to $\mathrm{Ind}_{A_{n-1}}^{D_n}(J_{A_{n-1}}(M_{1/2}^h,N_{1/2}^h))$.
By \eqref{eq-dd} and induction in stages, one therefore only needs to show $\kappa$ and $\kappa^h$ are surjective, 
as well as $\omega$ and $\omega^h$ are injective on $\mathrm{im}(\zeta)$ and $\mathrm{im}(\zeta^h)$ respectively.
However, one can check that $\kappa$ and $\omega^h$ (resp. $\kappa^h$ and $\omega$) are Hermitian adjoints of each other. So it suffices to prove that
\begin{center}
    $\omega$ and $\omega^h$ are injective on $\mathrm{im}(\zeta)$ and $\mathrm{im}(\zeta^h)$ respectively.
\end{center}
For simplicity, we present the proof when $a$ and $b$ are both even. This fixes the signs of $\frac{\epsilon}{2}$ in \eqref{eq-lema1} and \eqref{eq-lema12}. The general case can be obtained from the proof below by replacing some $\mathfrak{gl}$-intertwining operators into the $\overline{\mathfrak{gl}}$ ones (or vice versa).

\medskip
\noindent {\bf Injectivity of $\omega$:}  By assumption $a\leq b+1$, one has $(\frac{1}{2}-2b)\leq (\frac{5}{2}-2a)$. Consider the intertwining operator $\omega_1$ from parameter $(M_{1/2},N_{1/2})$ to $\omega_1(M_{1/2},N_{1/2})$:
\[\begin{aligned} &\left((\frac{1}{2}-2b)^+, \dots, (\frac{1}{2}-2a)^+,(\frac{5}{2}-2a)^+, \dots, (2a-\frac{7}{2})^+, \underline{(2a-\frac{3}{2})^+}\right)\\
\xrightarrow{\omega_1} & \left((\frac{1}{2}-2b)^+, \dots, (\frac{1}{2}-2a)^+,\underline{-(2a-\frac{3}{2})^-}, (\frac{5}{2}-2a)^+, \dots, (2a-\frac{11}{2})^+, -(2a-\frac{7}{2})^-\right).\end{aligned}\]
Note that $\mathrm{im}(\zeta)$ lies inside $\mathrm{Ind}_{A_{b-a}\times \overline{A_{2a-2}}}^{D_n}(J_{A_{b-a}\times \overline{A_{2a-2}}}(M_{1/2},N_{1/2}))$, and $\omega_1$ can be seen as an $\overline{\mathfrak{gl}}(2a-2) \times \overline{\mathfrak{gl}}(1)$-intertwining operator on the sub-parameter $\left((\frac{5}{2}-2a)^+, \dots, (2a-\frac{7}{2})^+\right)\left(\underline{(2a-\frac{3}{2})^+}\right)$.
Now Proposition \ref{inj_v_x} implies that $\omega_1$ is injective over $\mathrm{im}(\zeta)$. 

Similarly, let $\omega_2, \dots, \omega_{a-1}$ be the intertwining operators given by:  
\[\begin{aligned}  
& \left((\frac{1}{2}-2b)^+, \dots, (\frac{1}{2}-2a)^+,(\frac{3}{2}-2a)^-, (\frac{5}{2}-2a)^+, \dots, (2a-\frac{11}{2})^+, (\frac{7}{2}-2a)^-\right) \\
\xrightarrow{\omega_2} & \left((\frac{1}{2}-2b)^+, \dots, (\frac{1}{2}-2a)^+,-(2a-\frac{3}{2})^-, (\frac{5}{2}-2a)^+, (\frac{7}{2}-2a)^-,(\frac{9}{2}-2a)^-,\dots, (2a-\frac{11}{2})^+\right) \\
\dots &  \\
\xrightarrow{\omega_{a-1}} & \left((\frac{1}{2}-2b)^+, (\frac{5}{2}-2b)^+, \dots, (\frac{1}{2}-2a)^+, (\frac{3}{2}-2a)^-, (\frac{5}{2}-2a)^+,\dots,-\frac{3}{2}^+,(\frac{-1}{2})^{-}\right).
\end{aligned}\]
For a similar reason to $\omega_1$, one has $\omega_i$ is injective over the image of $\omega_{i-1}\circ\dots\circ\omega_1\circ \zeta$. Notice that $\omega=\omega_{a-1}\circ\dots\circ \omega_1$, so $\omega$ is injective over $\mathrm{im}(\zeta)$.

\medskip
\noindent {\bf Injectivity of $\omega^h$}: Write $\omega^h$ as the composition of the intertwining operators $\gamma$ and $\delta$ as :
{\allowdisplaybreaks \begin{align*}
&\left((\frac{3}{2}-2a)^+, \dots, \frac{-5}{2}^+, \frac{-1}{2}^+, \frac{3}{2}^+, \dots, (2b-\frac{9}{2})^+, (2b-\frac{5}{2})^+,(2b-\frac{1}{2})^+\right) \\ 
\xrightarrow{\omega_1^h}
&\left((\frac{3}{2}-2a)^+, \dots, \frac{-5}{2}^+, \frac{-1}{2}^+\right) \left((\frac{1}{2}-2b)^-\right) \left( \frac{3}{2}^+, \dots, (2b-\frac{9}{2})^+,(\frac{5}{2}-2b)^- \right) \\ 
\xrightarrow{\omega_2^h}
&\left((\frac{3}{2}-2a)^+, \dots, \frac{-5}{2}^+, \frac{-1}{2}^+\right) \left((\frac{1}{2}-2b)^-,,(\frac{5}{2}-2b)^-\right) \left( \frac{3}{2}^+, \dots, (2b-\frac{9}{2})^+\right) \\
\dots &  \\
\xrightarrow{\omega_{b-1}^h} &\left((\frac{3}{2}-2a)^+, \dots, \frac{-5}{2}^+, \frac{-1}{2}^+\right) \left((\frac{1}{2}-2b)^-,\dots,\frac{-7}{2}^-,\frac{-3}{2}^{-} \right)
\\
\xrightarrow{\omega_b^h}  & \left((\frac{1}{2}-2b)^-, (\frac{5}{2}-2b)^-, \dots, (\frac{1}{2}-2a)^-, (\frac{3}{2}-2a)^+, (\frac{5}{2}-2a)^-,\dots,\frac{-3}{2}^-,(\frac{-1}{2})^{+}\right).\end{align*}}
For $1 \leq i \leq b-1$, the operators $\omega_i^h$ are similar to the $\omega_{i}$'s above. So $\omega_i^h$ are injective over $\mathrm{im}(\omega_{i-1}^h\circ\dots\circ\omega_1^h \circ\zeta^h)$.
Finally, by applying Proposition \ref{type_A_inj}, one can see $\omega_b^h$ is injective over $\mathrm{im}(\omega_{b-1}^h \circ \dots \circ \omega_1^h \circ\zeta^h)$. This finishes the proof of the statement.
\end{proof}

\begin{example} Consider $N_{1/2} = \left(\frac{5}{2}, \frac{1}{2}, \frac{-3}{2}, \frac{-7}{2}, \frac{-11}{2}, \frac{-15}{2}\right) \longleftrightarrow \begin{pmatrix} 2\\ 4 \end{pmatrix}$. Here, $\omega$ is given by:
\begin{align*}\omega: \left(\frac{-15}{2}^+, \frac{-11}{2}^+, \frac{-7}{2}^+, \frac{-3}{2}^+, \frac{1}{2}^+, \frac{5}{2}^+\right) \longrightarrow \left(\frac{-15}{2}^+, \frac{-11}{2}^+, \frac{-7}{2}^-, \frac{-5}{2}^-, \frac{-3}{2}^-, \frac{-1}{2}^-\right), \end{align*}
which is a $\overline{\mathfrak{gl}}(2) \times \overline{\mathfrak{gl}}(1)$ operator on the last three coordinates. By applying Proposition \ref{inj_v_x} and Remark \ref{rmk-inclusion} with $(\nu_1^{\delta},\nu_2^{\delta}) = (\frac{-3}{2}^+, \frac{1}{2}^+)$ and $(\xi_1^{\epsilon_1}) = (\frac{-5}{2}^-)$ (here we take $(\frac{-5}{2}^-)$ instead of $(\frac{5}{2}^+)$ since we work on $\overline{\mathfrak{gl}}(3)$ instead of $\mathfrak{gl}(3)$), one can conclude that $\omega$ is injective on the level of spin-relevant $K$-types.

Now we study $\omega^h:$
\begin{align*}\omega^h: \left(\frac{-5}{2}^+, \frac{-1}{2}^+, \frac{3}{2}^+, \frac{7}{2}^+, \frac{11}{2}^+, \frac{15}{2}^+\right) &\xrightarrow{\omega_1^h} \left(\frac{-5}{2}^+, \frac{-1}{2}^+, \frac{-15}{2}^-, \frac{-11}{2}^-, \frac{-7}{2}^-, \frac{-3}{2}^-\right) \\ &\xrightarrow{\omega_2^h} \left(\frac{-15}{2}^-, \frac{-11}{2}^-, \frac{-7}{2}^-, \frac{-5}{2}^+, \frac{-3}{2}^-, \frac{-1}{2}^+\right)
\end{align*}
More explicitly, $\omega_1^h$ consists of several $\overline{\mathfrak{gl}}(2)$-intertwining operators by changing the $\nu$-coordinates into negatives, and the usual $\mathfrak{gl}(2)$-intertwining operators swapping the negative $\nu$-coordinates to the left. All these operators have no kernel on the level of $\mathcal{V}_{\eta(i)}$'s, since the difference (resp. sum) of the $\nu$-coordinates in each  $\mathfrak{gl}(2)$-operator is not equal to $2$ or $3$. So $\omega_1^h$ is injective on the spin-relevant $K$-types.

As for $\omega_2^h$, one can directly apply Proposition \ref{type_A_inj} and Remark \ref{rmk-inclusion} with $(\xi_1^{\epsilon_1},\xi_2^{\epsilon_2}) = \left(\frac{-5}{2}^+, \frac{-1}{2}^+\right)$ and  $(\nu_1^{\delta},\nu_2^{\delta},\nu_3^{\delta},\nu_4^{\delta}) = \left(\frac{-15}{2}^-, \frac{-11}{2}^-, \frac{-7}{2}^-, \frac{-3}{2}^-\right)$
to conclude that it is also injective on the spin-relevant $K$-types.
\end{example}

\begin{example}
    It would be beneficial to see how the injectivity argument fails for some parameters not covered by the hypothesis in the proposition -- let $N_{1/2} = \left(\frac{13}{2}, \frac{9}{2}, \frac{5}{2}, \frac{1}{2}, \frac{-3}{2}, \frac{-7}{2}\right) \longleftrightarrow \begin{pmatrix} 4\\ 2 \end{pmatrix}$, and 
\begin{align*}\omega: \left(\frac{-7}{2}^+, \frac{-3}{2}^+, \frac{1}{2}^+, \frac{5}{2}^+, \frac{9}{2}^+, \frac{13}{2}^+\right) \longrightarrow \left(\frac{-13}{2}^-, \frac{-9}{2}^-, \frac{-7}{2}^+, \frac{-5}{2}^-, \frac{-3}{2}^+, \frac{-1}{2}^-\right).
\end{align*}
In particular, one has to pass $\frac{-13}{2}^-$ to the leftmost coordinate. However, since $\frac{-7}{2} - 3 = \frac{-13}{2}$, one cannot apply Proposition \ref{inj_v_x}. In order to apply Proposition \ref{type_A_inj} instead, one may consider $(\nu_1^{\delta}, \nu_2^{\delta}) = (\frac{-13}{2}^-,\frac{-9}{2}^-)$. However, 
$\frac{-9}{2} + 3 = \frac{-3}{2}$ which appears on the left of $\frac{9}{2}$, so the proposition does not apply either. Similarly, if one considers $(\nu_1^{\delta}, \nu_2^{\delta}, \nu_3^{\delta}) = (\frac{-13}{2}^-,\frac{-9}{2}^-,\frac{-5}{2}^-)$, then one has $\frac{-5}{2}+3=\frac{1}{2}$ which appears on the left of $\frac{5}{2}$. Finally, if one considers $(\nu_1^{\delta}, \nu_2^{\delta}, \nu_3^{\delta},\nu_4^{\delta}) = (\frac{-13}{2}^-,\frac{-9}{2}^-,\frac{-5}{2}^-,\frac{-1}{2}^-)$, the intertwining operator 
$$\left(\frac{1}{2}^+,\frac{5}{2}^+\right) \longrightarrow \left(\frac{-5}{2}^-,\frac{-1}{2}^-\right)$$
has kernel on the spin-relevant $K$-types. 

In conclusion, Proposition \ref{inj_v_x} and \ref{type_A_inj} does not guarantee injectivity of $\omega$ (and surjectivity of $\kappa^h$). In fact, one can apply \texttt{atlas} to check that the multiplicities of $J(M_{1/2},N_{1/2})$ and the induced module $\mathrm{Ind}_{A_5}^{D_6}(J_{A_5}(M_{1/2}, N_{1/2}))$ are different on the level of spin-relevant $K$-types (see Appendix A for an example of atlas calculations). In view of Corollary \ref{cor-u} below, such discrepancy is expected since $\pi$ is known to be a (unitary) Brega representation.
\end{example}

\begin{corollary} \label{cor-u}
Let $n = a+b$ and $N_{1/2} \longleftrightarrow \begin{pmatrix} a  \\ b  \end{pmatrix}$ with $a < b$. Then
$\pi_I := J(N_{1/2} \cup N_{1/2}^h, N_{1/2} \cup N_{1/2}^h)$
is not unitary at the $K$-type with highest weight $\eta(2a+1)$.
\end{corollary}
\begin{proof}
By Proposition \ref{single_mult}, $\pi_I$ is a subquotient of
$\mathrm{Ind}_{A_{2n-1}}^{D_{2n}} \big(J_A(M_{1/2} \cup M_{1/2}^h, N_{1/2} \cup N_{1/2}^h)\big),$
and the multiplicities of the $K$-type $\mathcal{V}_{\eta(i)}$ are equal for both modules for all $0 \leq i \leq n$.
Note that Theorem \ref{thm-comp} implies that
$$J_A(M_{1/2} \cup M_{1/2}^h, N_{1/2} \cup N_{1/2}^h) =
\mathrm{Ind}_{A_{n-1} \times A_{n-1}}^{A_{2n-1}}\left(\mathrm{comp}_{1/2}(n,b-a+\frac{1}{2}) \otimes \mathrm{comp}_{1/2}(n,b-a+\frac{1}{2})^h\right)$$
is Hermitian, and the Hermitian form is indefinite on the $\widetilde{U}(2n)$-type with highest weight $(\frac{1}{2},\dots, \frac{1}{2}) +(\overbrace{1,\dots,1}^{2a+1},0,\dots,0,\overbrace{-1,\dots,-1}^{2a+1})$.
So the induced module above has an indefinite form on
$\mathcal{V}_{\eta(2a+1)}$ (cf. \cite[Proposition 10.5]{V86}), and the result follows.
\end{proof}

\subsection{Base Case II} \label{sec-v}
\begin{proposition} \label{two_string_mult}
Let $n = c+d+e+f$ and $N_{1/2}\longleftrightarrow \begin{pmatrix} c &  d \\ e & f\end{pmatrix}$ be such that $d> e+1$. Set 
$$A_{1/2} \longleftrightarrow \begin{pmatrix} d \\ e \end{pmatrix}, \quad \quad \quad D_{1/2} \longleftrightarrow \begin{pmatrix} c \\ f \end{pmatrix}.$$ 
Then $\pi=J_{D_{2n}}(M_{1/2} \cup M_{1/2}^h,N_{1/2}\cup N_{1/2}^h)$ is the lowest $K$-type subquotient of the induced representation
\[\mathrm{Ind}_{A_{2(d+e)-1}\times D_{2(c+f)}}^{D_{2n}} \big(J_A(A_{1/2}\cup A_{1/2}^h)\otimes J_D(D_{1/2} \cup D_{1/2}^h)\big)
\]
(we omit the $\mu$-parameters in the above expression, whose coordinates are all equal to $\frac{1}{2}$), both having the same multiplicity of $\mathcal{V}_{\eta(i)}$  for all $0\leq i\leq n$.
\end{proposition}
\begin{proof} As in the previous case, the proposition holds more generally for $d \geq e$, which we will show below. Consider as in Proposition \ref{single_mult} the intertwining operators:
\begin{align*}
&(\theta,\phi)\\
\xrightarrow{\kappa} &\left((2d-\frac{3}{2})^+,\dots,(\frac{1}{2}-2e)^+\right)\left((2c-\frac{3}{2})^+,\dots, (2f+\frac{1}{2})^+,(2f-\frac{1}{2})^-,\dots,\frac{5}{2}^+,\frac{3}{2}^-,\frac{\epsilon(f)}{2}^{\epsilon(f)} \right)\\
\xrightarrow{\zeta} &\left((\frac{1}{2}-2e)^+,\dots,(2d-\frac{3}{2})^+\right)\left((\frac{3}{2}-2c)^-,\dots, (\frac{-1}{2}-2f)^-,(\frac{1}{2}-2f)^+,\dots,\frac{-5}{2}^-,\frac{-3}{2}^+,\frac{\epsilon(c+1)}{2}^{\epsilon(c+1)} \right)\\
\xrightarrow{\omega} &(w_0\theta,w_0\phi),
\end{align*}
where $\zeta$ is the long $A_{d+e-1} \times D_{c+f}$-intertwining operator, along with its Hermitian dual:
\begin{align*}
&(\theta^h,\phi^h)\\
\xrightarrow{\kappa^h} &\left((2e-\frac{1}{2})^+,\dots,(\frac{3}{2}-2d)^+\right) \left((2c-\frac{3}{2})^-,\dots, (2f+\frac{1}{2})^-,(2f-\frac{1}{2})^+,\dots,\frac{5}{2}^-,\frac{3}{2}^+,\frac{\epsilon(c)}{2}^{\epsilon(c+1)}\right)\\
\xrightarrow{\zeta^h} &\left((\frac{3}{2}-2d)^+,\dots,(2e-\frac{1}{2})^+\right) \left((\frac{3}{2}-2c)^+,\dots, (\frac{-1}{2}-2f)^-,(\frac{1}{2}-2f)^-,\dots,\frac{-5}{2}^+,\frac{-3}{2}^-,\frac{\epsilon(f+1)}{2}^{\epsilon(f)}\right)\\
\xrightarrow{\omega^h} &(w_0\theta^h,w_0\phi^h)
\end{align*}
One needs to show that $\ker(\omega|_{\mathrm{im}(\zeta)})=0$ 
and $\ker(\omega^h|_{\mathrm{im}(\zeta^h)})=0$ on the $\mathcal{V}_{\eta(i)}$-isotropic spaces. 

\bigskip
\noindent{\bf Step 1 - Reduction to $\begin{pmatrix} c & d \\ e & 0 \end{pmatrix}$:} As before, we assume $c, d, e, f$ are even for simplicity. Rather than studying $\omega$ and $\omega^h$ directly, consider the intertwining operators
\[\begin{aligned}
&\left((2d-\frac{3}{2})^+,\dots,(\frac{1}{2}-2e)^+\right) \left((2c-\frac{3}{2})^+,\dots, (2f+\frac{1}{2})^+,(2f-\frac{1}{2})^-,\dots,\frac{5}{2}^+,\frac{3}{2}^-,\frac{1}{2}^{+}\right)\\
 \xrightarrow{\Sigma}&
\left((\frac{1}{2}-2e)^+,\dots,(2d-\frac{3}{2})^+\right) \left((\frac{1}{2}-2f)^+,\dots,\frac{-3}{2}^+\right) \left((\frac{3}{2}-2c)^-,\dots,\frac{-1}{2}^{-} \right)\\
\xrightarrow{\Delta} &
(w_0\theta,w_0\phi),\end{aligned}\]
and 
\[\begin{aligned}
&\left((2e-\frac{1}{2})^+,\dots,(\frac{3}{2}-2d)^+\right) \left((2c-\frac{3}{2})^-,\dots, (2f+\frac{1}{2})^-,(2f-\frac{1}{2})^+,\dots,\frac{5}{2}^-,\frac{3}{2}^+,\frac{1}{2}^{-}\right)\\
 \xrightarrow{\Sigma^h}&
\left((\frac{3}{2}-2d)^+,\dots,(2e-\frac{1}{2})^+\right) \left((\frac{1}{2}-2f)^-,\dots,\frac{-3}{2}^-\right) \left((\frac{3}{2}-2c)^+,\dots,\frac{-1}{2}^{+} \right)\\
\xrightarrow{\Delta^h} &
(w_0\theta^h,w_0\phi^h),\end{aligned}\]
Then $\Delta \circ \Sigma = \omega \circ \zeta$, and
$\Sigma$ is part of the intertwining operator 
$\zeta$ (and similarly for $\Delta^h$, $\Sigma^h$, $\omega^h$, $\zeta^h$). Hence it suffices to show that $\ker(\Delta|_{\mathrm{im}(\Sigma)})=0$ and $\ker(\Delta^h|_{\mathrm{im}(\Sigma^h)})=0$.

Write $\Delta = \iota_3 \circ \iota_2 \circ \iota_1$ and $\Delta^h = (\iota_3^h)' \circ (\iota_2^h)' \circ (\iota_1^h)'$, where:
\begin{align*}
&\left((\frac{1}{2}-2e)^+,\dots,(2d-\frac{3}{2})^+\right) \left((\frac{1}{2}-2f)^+,\dots,\frac{-3}{2}^+\right) \left((\frac{3}{2}-2c)^-,\dots,\frac{-1}{2}^{-}\right) \\
\xrightarrow{\iota_1} &
\left((\frac{1}{2}-2f)^+,\dots,\frac{-3}{2}^+\right) \left((\frac{1}{2}-2e)^+,\dots,(2d-\frac{3}{2})^+\right) \left((\frac{3}{2}-2c)^-,\dots,\frac{-1}{2}^{-}\right) \\
\xrightarrow{\iota_2} & \left((\frac{1}{2}-2f)^+,\dots,\frac{-3}{2}^+\right)\left((\frac{1}{2}-2e)^+,\dots,\frac{-3}{2}^+\right) \left((\frac{3}{2}-2c)^-,\dots,\frac{-1}{2}^{-}\right) \left((\frac{3}{2}-2d)^-,\dots,\frac{-1}{2}^{-} \right) \\
\xrightarrow{\iota_3} &
\ (w_0\theta,w_0\phi).
\end{align*}
and 
\begin{align*}
&\left((\frac{3}{2}-2d)^+,\dots,(2e-\frac{1}{2})^+\right) \left((\frac{1}{2}-2f)^-,\dots,\frac{-3}{2}^-\right) \left((\frac{3}{2}-2c)^+,\dots,\frac{-1}{2}^{+} \right) \\
\xrightarrow{\iota_1^h} &
\left((\frac{1}{2}-2f)^-,\dots,\frac{-3}{2}^-\right) \left((\frac{3}{2}-2d)^+,\dots,(2e-\frac{1}{2})^+\right) \left((\frac{3}{2}-2c)^+,\dots,\frac{-1}{2}^{+} \right) \\
\xrightarrow{\iota_2^h} &
\left((\frac{1}{2}-2f)^-,\dots,\frac{-3}{2}^-\right)\left((\frac{3}{2}-2d)^+,\dots,\frac{-1}{2}^{+}\right) \left((\frac{3}{2}-2c)^+,\dots,\frac{-1}{2}^{+}\right) \left((\frac{1}{2}-2e)^{-},\dots,\frac{-3}{2}^{-}\right)  \\
\xrightarrow{\iota_3^h} &
\ (w_0\theta^h,w_0\phi^h),
\end{align*} 
We first study $\iota_1$ in full detail, and the $\iota_1^h$ case is similar: Note that $\iota_1$ may conjugate the $\nu$-parameters into a more dominant form, so it may not be well-defined. However, by Proposition \ref{inj_v_x}, $\iota_1$ is well-defined and injective over $\mathrm{im}(\Sigma)$. Indeed, $\iota_1$ is given by:
\small
\[\left((\frac{1}{2}-2e)^+,\dots,(2d-\frac{3}{2})^+\right) \left((\frac{1}{2}-2f)^+,\dots,\frac{-3}{2}^+\right) \to
\left((\frac{1}{2}-2f)^+,\dots,\frac{-3}{2}^+\right) \left((\frac{1}{2}-2e)^+,\dots,(2d-\frac{3}{2})^+\right),\]
\normalsize so for all $x \in \{\frac{1}{2}-2f,\dots,\frac{-3}{2}\}$,
$x \neq (\frac{1}{2}-2e)-2$ or $(2d-\frac{3}{2}) + 2$ since $e \geq f$, and the hypotheses of Proposition \ref{inj_v_x} is satisfied. Meanwhile, Proposition \ref{non_red_of_phi} implies that
$$\Delta=\iota_3 \circ \iota_2 \circ\iota_1$$
over the $\mathcal{V}_{\eta(i)}$-isotropic subspaces of $\mathrm{im}(\Sigma)$. 
Now we focus on $\iota_3$, and similarly on $\iota_3^h$. Indeed, $\iota_3$ can be seen as the composition of $3$ type A intertwining operators given in Proposition \ref{type_A_inj}. In each case, since $\frac{-3}{2} + 2$ or $\frac{-3}{2} + 3$ (or $\frac{-1}{2} + 2$ or $\frac{-1}{2} + 3$) are positive, it cannot be equal to any of the negative coordinates to its left. Therefore, $\iota_3$ is injective over $\mathrm{im}(\iota_2\circ\iota_1\circ\Sigma)$.  
So we are left to show $\iota_2$ is injective over $\mathrm{im}(\iota_1 \circ\Sigma)$, that is, one can omit the $\left((\frac{1}{2}-2f)^+,\dots,\frac{-3}{2}^+\right)$-component in the above intertwining operators. In other words, we assume $f = 0$ and study the injectivity of $\Delta = \omega = \iota_3 \circ \iota_2$ over $\mathrm{im}(\Sigma) = \mathrm{im}(\zeta)$.

\bigskip
\noindent{\bf Step 2 - Reduction to $\begin{pmatrix} 2 & 2 \\ 0 & 0 \end{pmatrix}$ for $\omega$:}
We first focus on $\omega$, and further reduce the study of injectivity of $\omega$ from the case of $\begin{pmatrix} c & d \\ e & 0 \end{pmatrix}$ to that of $\begin{pmatrix} 2 & 2 \\ 0 & 0 \end{pmatrix}$. 
For simplicity, we assume $c$ and $d$ are even so that one always has $\frac{\epsilon}{2}^{\epsilon} = \frac{-1}{2}^-$ in the intertwining operators.

\medskip
Consider $\omega$ as a composition of the following operators:
{\allowdisplaybreaks \begin{align*}
&\left((\frac{1}{2}-2e)^+,\dots,\frac{1}{2}^+, \frac{5}{2}^+, \frac{9}{2}^+, \dots, (2d-\frac{3}{2})^+\right) \left((\frac{3}{2}-2c)^-,\dots,\frac{-5}{2}^-,\frac{-1}{2}^{-}\right) \\
\xrightarrow{\omega_1} & \left((\frac{1}{2}-2e)^+,\dots,\frac{1}{2}^+, \frac{5}{2}^+\right) \left((\frac{3}{2}-2c)^-,\dots,\frac{-5}{2}^-,\frac{-1}{2}^{-}\right)\left(\frac{9}{2}^+, \dots, (2d-\frac{3}{2})^+\right)  \\
\xrightarrow{\omega_2} & \left((\frac{1}{2}-2e)^+,\dots,\frac{1}{2}^+, \frac{5}{2}^+\right) \left((\frac{3}{2}-2d)^-, \dots, \frac{-9}{2}^- \right) \left((\frac{3}{2}-2c)^-,\dots,\frac{-5}{2}^-,\frac{-1}{2}^{-}\right)  \\
\xrightarrow{\omega_3} & \left((\frac{1}{2}-2e)^+,\dots,\frac{-3}{2}^+\right) \left((\frac{3}{2}-2d)^-, \dots, \frac{-9}{2}^- \right) \left(\frac{1}{2}^+, \frac{5}{2}^+\right)  \left((\frac{3}{2}-2c)^-,\dots,\frac{-5}{2}^-,\frac{-1}{2}^{-}\right)  \\
\xrightarrow{\omega_4} & \left((\frac{1}{2}-2e)^+,\dots,\frac{-3}{2}^+\right) \left((\frac{3}{2}-2d)^-, \dots, \frac{-9}{2}^- \right) \left((\frac{3}{2}-2c)^-,\dots,\frac{-9}{2}^-\right) \left(\frac{1}{2}^+, \frac{5}{2}^+\right)  \left(\frac{-5}{2}^-,\frac{-1}{2}^{-}\right)  \\
\xrightarrow{\omega_5} & \left((\frac{1}{2}-2e)^+,\dots,\frac{-3}{2}^+\right) \left((\frac{3}{2}-2d)^-, \dots, \frac{-9}{2}^- \right) \left((\frac{3}{2}-2c)^-,\dots,\frac{-9}{2}^-,\frac{-5}{2}^-,\frac{-1}{2}^{-} \right) \left(\underline{\frac{-5}{2}^-, \frac{-1}{2}^-}\right)   \\
\xrightarrow{\omega_6} & \left((\frac{1}{2}-2e)^+,\dots,\frac{-3}{2}^+\right) \left((\frac{3}{2}-2d)^-, \dots, \frac{-9}{2}^-, \underline{\frac{-5}{2}^-, \frac{-1}{2}^-} \right) \left((\frac{3}{2}-2c)^-,\dots,\frac{-9}{2}^-,\frac{-5}{2}^-,\frac{-1}{2}^{-} \right)   \\
\xrightarrow{\omega_7} & \ (w_0\theta,w_0\phi).
\end{align*}}
Let's show the injectivity of $\omega_i$ over the image of $\zeta$ for $i\neq 5$: For $\omega_1$, it is a type A intertwining operators given in Proposition \ref{type_A_inj}, and is injective since for all $x \in \{\frac{9}{2},\dots,2d-\frac{3}{2}\}$,
$x \neq \frac{-1}{2}+3$. By similar reasoning, $\omega_3$, $\omega_4$, $\omega_7$ are all injective over the image of $\zeta$. For $\omega_2$, one can write it as the composition of the intertwining operator $\left(\frac{9}{2}^+, \dots, (2d-\frac{3}{2})^+\right)$ to $\left((\frac{3}{2}-2d)^-,\dots, \frac{-9}{2}^-\right)$, which is injective since its decomposition of intertwining operators corresponding to the simple roots are injective by \eqref{phi_+-},  and a type A intertwining operators given in Proposition \ref{type_A_inj}, which also can be checked injective by $c\geq d$. For $\omega_6$, which moves the underlying parameter $(\underline{\frac{-5}{2}^-,\frac{-1}{2}^{-}})$ from right to left, it is injective over the image of $\zeta$ due to Proposition \ref{type_A_inj}, which is similar to the $\iota_3$. Hence, it only remains to study injectivity of $\omega_5$, which is part of the $\omega$-operator for $\begin{pmatrix} 2 & 2 \\ 0 & 0 \end{pmatrix}$.

\bigskip
\noindent{\bf Step 3 - Reduction to $\begin{pmatrix} 2 & 1 \\ 1 & 0 \end{pmatrix}$ for $\omega^h$:}
As in Step 2, we will reduce the study of injectivity of $\omega^h$ from the case of $\begin{pmatrix} c & d \\ e & 0 \end{pmatrix}$ to that of $\begin{pmatrix} 2 & 1 \\ 1 & 0 \end{pmatrix}$. 

Assume $e$ is even. Write $\omega^h$ as
\small
{\allowdisplaybreaks
\begin{align*}
& \left((\frac{3}{2}-2d)^+,\dots,\frac{-1}{2}^+,\frac{3}{2}^+,\frac{7}{2}^+,\dots,(2e-\frac{1}{2})^+\right) \left((\frac{3}{2}-2c)^+,\dots,\frac{-1}{2}^{+} \right) 
\\ \xrightarrow{\omega_1^h} & \left((\frac{3}{2}-2d)^+,\dots,\frac{-1}{2}^+,\frac{3}{2}^+\right) \left((\frac{3}{2}-2c)^+,\dots,\frac{-1}{2}^{+} \right) \left(\frac{7}{2}^+,\dots,(2e-\frac{1}{2})^+\right)
\\ \xrightarrow{\omega_2^h} & \left((\frac{3}{2}-2d)^+,\dots,\frac{-5}{2}^+,\frac{-1}{2}^+,\frac{3}{2}^+\right) \left((\frac{1}{2}-2e)^-,\dots,\frac{-7}{2}^-\right) \left((\frac{3}{2}-2c)^+,\dots,\frac{-1}{2}^{+} \right) 
\\ \xrightarrow{\omega_3^h} &
\left((\frac{3}{2}-2d)^+,\dots,(\frac{-1}{2}-2e)^+,(\frac{1}{2}-2e)^-,\dots,\frac{-7}{2}^-,\frac{-5}{2}^{+},\frac{-1}{2}^{+},\frac{3}{2}^+\right) \left((\frac{3}{2}-2c)^+,\dots,\frac{-5}{2}^+,\frac{-1}{2}^{+} \right)
\\ \xrightarrow{\omega_4^h} &
\left((\frac{3}{2}-2d)^+,\dots,(\frac{-1}{2}-2e)^+,(\frac{1}{2}-2e)^-,\dots,\frac{-7}{2}^-,\frac{-5}{2}^{+}\right) \left((\frac{3}{2}-2c)^+,\dots,\frac{-9}{2}^+\right) \left(\frac{-1}{2}^{+},\frac{3}{2}^+\right) \left(\frac{-5}{2}^+,\frac{-1}{2}^{+} \right)
\\ \xrightarrow{\omega_5^h} &
\left((\frac{3}{2}-2d)^+,\dots,(\frac{-1}{2}-2e)^+,(\frac{1}{2}-2e)^-,\dots,\frac{-7}{2}^-,\frac{-5}{2}^{+}\right) \left((\frac{3}{2}-2c)^+,\dots,\frac{-9}{2}^+,\frac{-5}{2}^+,\frac{-1}{2}^{+}\right) \left(\frac{-3}{2}^{-},\frac{1}{2}^-\right) 
\\
\xrightarrow{\omega_6^h} &
\ (w_0\theta^h,w_0\phi^h),
\end{align*}}
For $\omega_1^h$, $\omega_4^h$ and $\omega_6^h$, their injectivities are due to Proposition \ref{type_A_inj} similarly as above. For $\omega_2^h$, it is similar to $\omega_2$ above. For $\omega_3^h$, regard $\omega_3^h$ as the intertwining operator of another version of Proposition \ref{type_A_inj} about ``$\nu=(\nu_1<\dots<\nu_p;\xi_1\leq \dots\leq \xi_q)$". Since $d \geq e$, for all $x\in \{\frac{1}{2}-2e,\dots,-\frac{7}{2}\}$, $x\neq \frac{3}{2}-2d-3$. Hence, $\omega_3^h$ is injective. So it remains to show the injectivity of $\omega_5^h$, which is part of the $\omega^h$ operator for $\begin{pmatrix} 2 & 1 \\ 1 & 0 \end{pmatrix}$. 



\bigskip
\noindent{\bf Step 4: Proof of $\begin{pmatrix} 2 & 2 \\ 0 & 0 \end{pmatrix}$ and $\begin{pmatrix} 2 & 1 \\ 1 & 0 \end{pmatrix}$.} 
As discussed above, it suffices to show that $\omega$ (resp. $\omega^h$) is injective for $N_{1/2} \longleftrightarrow \begin{pmatrix} 2 & 2 \\ 0 & 0 \end{pmatrix}$ (resp. $\begin{pmatrix} 2 & 1 \\ 1 & 0 \end{pmatrix}$). Equivalently, one needs to check the proposition holds for these $N_{1/2}$-parameters, whose proofs are postponed to Appendix A.
\end{proof} 

\begin{corollary} \label{cor-v} 
Let $n = c+d+e+f$ and $N_{1/2} \longleftrightarrow \begin{pmatrix} c & d  \\ e & f  \end{pmatrix}$ with $d > e+1$. Then
$\pi_{II} := J(M_{1/2} \cup M_{1/2}^h, N_{1/2} \cup N_{1/2}^h)$
is not unitary at the $K$-type with highest weight $\eta(2e+2)$.
\end{corollary}
\begin{proof}
By Proposition \ref{two_string_mult}, $\pi_{II}$ is a subquotient of $\mathrm{Ind}_{A_{2(d+e)-1}\times D_{2(c+f)}}^{D_{2n}} \big(J_A(A_{1/2}\cup A_{1/2}^h)\otimes J_D(D_{1/2} \cup D_{1/2}^h)\big)$
and the multiplicities of the $K$-type $\mathcal{V}_{\eta(2e+2)}$ are equal for both modules.
On the other hand, Theorem \ref{thm-comp} implies that
$$J_A(A_{1/2}\cup A_{1/2}^h) =
\mathrm{Ind}_{A_{d+e-1} \times A_{d+e-1}}^{A_{2(d+e)-1}}\left(\mathrm{comp}_{1/2}(d+e, d - e - \frac{1}{2} ) \otimes \mathrm{comp}_{1/2}(d+e, d-e - \frac{1}{2} )^h\right)$$
has an indefinite form on $(\frac{1}{2},\dots,\frac{1}{2})+(\overbrace{1,\dots,1}^{2e+2},0,\dots,0,\overbrace{-1,\dots,-1}^{2e+2})$,
and the module $J_D(D_{1/2} \cup D_{1/2}^h)$ is (unitary) Brega representation by Section \ref{sec-unitary}, since $c \geq d > e+1 \geq f+1 > f$. So the induced module above has an indefinite form on $\mathcal{V}_{\eta(2e+2)}$, and the result follows.
\end{proof}

\subsection{Induction by complementary series} \label{sec-comp}
In this section, we generalize the results in Section \ref{sec-u} and Section \ref{sec-v} by inducing
some Stein complementary series to the modules in Corollary \ref{cor-u} and Corollary \ref{cor-v}. We will
show that their non-unitarity is preserved upon induction by complementary series (Corollary \ref{cor-nonunit2}).
\begin{lemma}\label{case_first_is_comp}
Let $n = (s+t)+ \overline{n}$, where $\overline{n} := \sum_{i=1}^p (x_i+y_i)$ and $s-t = 0$ or $1$. Consider the parameters  
$N_{1/2}\longleftrightarrow \left(\begin{array}{cccc} s & x_1 & \dots & x_p \\ t & y_1 & \dots & y_p\end{array}\right)$, $\overline{N}_{1/2} \longleftrightarrow\left(\begin{array}{ccc}  x_1 & \dots & x_p \\ y_1 & \dots & y_p\end{array}\right)$ and their corresponding representations: 
$$\pi=J(M_{1/2}\cup M_{1/2}^h,N_{1/2}\cup N_{1/2}^h),\quad \quad \overline{\pi}=J(\overline{M}_{1/2}\cup \overline{M}_{1/2}^h,\overline{N}_{1/2}\cup \overline{N}_{1/2}^h).$$
Then $\pi$ and
\[\mathrm{Ind}^{D_{2n}}_{A_{2(s+t)-1}\times D_{2\overline{n}}}\Big(\mathrm{Ind}_{A_{s+t-1} \times A_{s+t-1}}^{A_{2(s+t)-1}}\left(\mathrm{comp}_{1/2}(s+t,\frac{1}{2})\otimes \mathrm{comp}_{1/2}(s+t,\frac{1}{2})^h\right)\otimes\overline{\pi}'\Big)\]
have the same multiplicity of the $K$-types of highest weights $\eta(i)$ for all $0 \leq i \leq n$.
\end{lemma}
\begin{proof}
Let $(\overline{\theta},\overline{\phi})$ be the dominant form of $(\overline{M}_{1/2},\overline{N}_{1/2})$, and its anti-dominant form is given by
$(\overline{w_0}\overline{\theta},\overline{w_0}\overline{\phi}) = (\gamma_1^{\epsilon_1}, \dots, \gamma_{\overline{n}}^{\epsilon_{\overline{n}}})$. As before, we assume $\overline{n}$ is even by simplicity, and check the injectivity of the following intertwining operators:
\begin{align*}
    \left((\frac{1}{2}-2t)^+, \dots, \frac{-3}{2}^+, \frac{1}{2}^+, \dots, (2s - \frac{3}{2})^+\right) \Big(\gamma_1^{\epsilon_1}, \dots, \gamma_{\overline{n}}^{\epsilon_{\overline{n}}}\Big) &\xrightarrow{\omega} (w_0\theta,w_0\phi) \\
    \left((\frac{3}{2}-2s)^+, \dots, \frac{-1}{2}^+, \frac{3}{2}^+, \dots, (2t - \frac{1}{2})^+\right) \Big(\gamma_1^{-\epsilon_1}, \dots,  \gamma_{\overline{n}}^{-\epsilon_{\overline{n}}}\Big) &\xrightarrow{\omega^h} (w_0\theta^h,w_0\phi^h),
\end{align*}
We present the proof for $\omega$ with $s = t$ is even. The proof for $s = t+1$, as well as that of $\omega^h$ are almost identical.

Split $\omega$ into the following operators:
\begin{align*}
&\left((\frac{1}{2}-2s)^+, \dots, \frac{-3}{2}^+, \frac{1}{2}^+, \dots, (2s - \frac{3}{2})^+\right) \Big(\gamma_1^{\epsilon_1}, \dots, \gamma_{\overline{n}}^{\epsilon_{\overline{n}}}\Big)\\ \xrightarrow{\omega_1}\  & \Big(\gamma_1^{\epsilon_1}, \dots, \gamma_{\overline{n}}^{\epsilon_{\overline{n}}}\Big) \left((\frac{1}{2}-2s)^+, \dots, \frac{-3}{2}^+, \frac{1}{2}^+, \dots, (2s - \frac{3}{2})^+\right) \\
\xrightarrow{\omega_2}\  & \Big(\gamma_1^{\epsilon_1}, \dots, \gamma_{\overline{n}}^{\epsilon_{\overline{n}}}\Big)   \left((-2s +\frac{3}{2})^-, \dots, \frac{-1}{2}^-\right) \left((\frac{1}{2}-2s)^+, \dots, \frac{-3}{2}^{+} \right) \\
\xrightarrow{\omega_3}\ &(w_0\theta,w_0\phi)
\end{align*}
We now study each operator appearing above individually: Firstly, note that $\omega_1$ may not make the $\nu$-parameter more anti-dominant. However, note that $s \geq x_i, y_i$ for all $i$ by hypothesis, and hence
$(\frac{1}{2} -2s) - 2$, $(\frac{1}{2} - 2s) - 3$, $(2s-\frac{3}{2})+2$, $(2s-\frac{3}{2})+3$ are not equal to any of $\gamma_j$ for $1 \leq j \leq \overline{n}$. Consequently, one can apply Proposition \ref{inj_v_x} along with Proposition \ref{non_red_of_phi} as before, implying that $\omega = \omega_3 \circ \omega_2 \circ \omega_1$ is well-defined and injective for all spin-relevant $K$-types.

As for $\omega_2$, it comprises of the following intertwining operators:
\begin{align*}
&\left((\frac{1}{2}-2s)^+, \dots, \frac{-3}{2}^+, \frac{1}{2}^+, \dots, (2s-\frac{11}{2})^+, (2s-\frac{7}{2})^+, (2s - \frac{3}{2})^+\right) \\
\longrightarrow\  &\left((-2s + \frac{3}{2})^- \right)  \left((\frac{1}{2}-2s)^+, \dots, \frac{-3}{2}^+, \frac{1}{2}^+, \dots, (2s-\frac{11}{2})^+, (-2s+\frac{7}{2})^-\right)\\
\longrightarrow
\ &\left((-2s + \frac{3}{2})^-, (-2s+\frac{7}{2})^- \right) \left((\frac{1}{2}-2s)^+, \dots, \frac{-3}{2}^+, \frac{1}{2}^+, \dots, (2s-\frac{11}{2})^+\right)\\
\longrightarrow\ &\dots \\
\longrightarrow\ & \left((-2s + \frac{3}{2})^-, (-2s+\frac{7}{2})^-, \dots, \frac{-1}{2}^- \right)  \left((\frac{1}{2}-2s)^+, \dots,(\frac{-3}{2})^{+}\right)
\end{align*}
By Proposition \ref{inj_v_x} and Proposition \ref{non_red_of_phi} again, all of the above maps are well-defined and injective. Now $\omega_3$ is injective by applying Proposition \ref{type_A_inj} twice, as in the proofs of the previous cases. So the result follows.
\end{proof}

\begin{lemma} \label{case_non_unitary_and_small_comp}
Let $N_{1/2} \longleftrightarrow \begin{pmatrix}  * & s_{1} & \dots & s_q  \\  ** & t_{1} & \dots & t_q \end{pmatrix}$ be such that
\begin{itemize}
\item $s_i - t_i = 0$ or $1$; and
\item $\begin{pmatrix} * \\ ** \end{pmatrix} = \begin{pmatrix} a \\ b \end{pmatrix}$ or $\begin{pmatrix} c & d \\ e & f \end{pmatrix}$ are as given in Corollary \ref{cor-u} and Corollary \ref{cor-v}.
\end{itemize}
Consider $\pi = J_{D_{2n}}(M_{1/2} \cup M_{1/2}^h, N_{1/2} \cup N_{1/2}^h)$, and the induced module
\begin{equation} \label{eq-nonunitind}
\begin{aligned}
&\mathrm{Ind}_{A_{2j-1} \times D_{2n-2j}}^{D_{2n}}\left(J_A(\widetilde{A}_{1/2}\cup {\widetilde{A}_{1/2}}^h) \otimes   J_D(\widetilde{D}_{1/2}\cup {\widetilde{D}_{1/2}}^h)\right),
\end{aligned}
\end{equation}
(the omitted $\mu$-parameters have coordinates all equal to $\frac{1}{2}$), where
\begin{equation} \label{eq-uv}
\left(\widetilde{A}_{1/2},\widetilde{D}_{1/2}\right) \longleftrightarrow \begin{cases} \left(\begin{pmatrix} a \\ b \end{pmatrix},\ \begin{pmatrix}  s_{1} & \dots & s_q  \\  t_{1} & \dots & t_q \end{pmatrix}\right), &\text{if}\ \begin{pmatrix} * \\ ** \end{pmatrix} = \begin{pmatrix} a \\ b \end{pmatrix},\\
 \left(\begin{pmatrix} d \\ e \end{pmatrix},\ \begin{pmatrix}  c & s_{1} & \dots & s_q  \\  f & t_{1} & \dots & t_q \end{pmatrix}\right)  &\text{if}\ \begin{pmatrix} * \\ ** \end{pmatrix} = \begin{pmatrix}  c & d \\ e & f \end{pmatrix}. \end{cases}
\end{equation}
Then $\pi$ and the induced module \eqref{eq-nonunitind} have the same multiplicities for all $K$-types with highest weight $\eta(i)$ for all $0\leq i\leq n$.
\end{lemma}

\begin{proof}
When $\begin{pmatrix} * \\ ** \end{pmatrix} = \begin{pmatrix} a \\ b \end{pmatrix}$, the argument is similar to the
Lemma \ref{case_first_is_comp}. Let us focus on the case $\begin{pmatrix} * \\ ** \end{pmatrix} = \begin{pmatrix} c & d \\ e & f \end{pmatrix}$:
For simplicity, assume that $c$, $f$, and $\overline{n}:=\sum_{i=1}^q(s_i+t_i)$ are even integers. Let $(\gamma_1^{\epsilon_1},\dots,\gamma_{\overline{n}}^{\epsilon_{\overline{n}}})$ be the anti-dominant of the parameter $(\overline{M}_{1/2},\overline{N}_{1/2})$ with $\overline{M}_{1/2} \longleftrightarrow \begin{pmatrix}  s_{1} & \dots & s_q  \\  t_{1} & \dots & t_q \end{pmatrix}$.
As before, one needs to show the injectivity of $\omega$ and $\omega^h$:
\begin{align*}
\left((\frac{1}{2}-2d)^+,\dots,(2e-\frac{3}{2})^+\right) z'\left(\gamma_1^{\epsilon_1},\dots,\gamma_{\overline{n}}^{\epsilon_{\overline{n}}};(2c-\frac{3}{2})^+,\dots,(\frac{1}{2}-2f)^+ \right)
& \xrightarrow{\omega}
(w_0\theta,w_0\phi) \\
\left((\frac{3}{2}-2e)^+,\dots,(2d-\frac{1}{2})^+\right) (z')^h\left(\gamma_1^{-\epsilon_1},\dots,\gamma_{\overline{n}}^{-\epsilon_{\overline{n}}}; (2f-\frac{1}{2})^+,\dots,(\frac{3}{2}-2c)^+ \right)
& \xrightarrow{\omega^h}
(w_0\theta^h,w_0\phi^h), 
\end{align*}
where $z', (z')^h \in W$ is chosen such that its conjugates in the above equations are anti-dominant,  over $\mathrm{Ind}_{A_{j-1} \times D_{n-j}}^{D_{n}}\left(J_A(\widetilde{A}_{1/2}) \otimes   J_D(\widetilde{D}_{1/2}) \right)$ and $\mathrm{Ind}_{A_{j-1} \times D_{n-j}}^{D_{n}}\left(J_A( {\widetilde{A}_{1/2}}^h) \otimes   J_D({\widetilde{D}_{1/2}}^h)\right)$  for all spin-relevant $K$-types.

We only study the case of $\omega$. As in Step 1 in the proof of Proposition \ref{two_string_mult}, rather than studying $\omega$, it suffices to study the injectivity of $\Delta$ over the image of $\Sigma$:
\begin{align*}
& \left(((2e-\frac{3}{2})^+,\dots,\frac{1}{2}-2d)^+\right) z'\left(\gamma_1^{\epsilon_1},\dots,\gamma_{\overline{n}}^{\epsilon_{\overline{n}}}; (2c-\frac{3}{2})^+,\dots,(\frac{1}{2}-2f)^+\right)\\
\xrightarrow{\Sigma}  & \left((\frac{1}{2}-2d)^+,\dots,(2e-\frac{3}{2})^+\right) \Big(\gamma_1^{\epsilon_1},\dots,\gamma_{\overline{n}}^{\epsilon_{\overline{n}}}\Big)\left((\frac{1}{2}-2f)^+,\dots,\frac{-3}{2}^+\right)\left((\frac{3}{2}-2c)^-,\dots,\frac{-1}{2}^{-}\right)\\
\xrightarrow{\Delta} & \ (w_0\theta,w_0\phi), \end{align*}
Write $\Delta$ as composition of $\Delta_3\circ\Delta_2\circ\Delta_1$:
\small
\begin{align*}
& \left((\frac{1}{2}-2d)^+,\dots,(2e-\frac{3}{2})^+\right) \Big(\gamma_1^{\epsilon_1},\dots,\gamma_{\overline{n}}^{\epsilon_{\overline{n}}}\Big)\left((\frac{1}{2}-2f)^+,\dots,\frac{-3}{2}^+\right)\left((\frac{3}{2}-2c)^-,\dots,\frac{-1}{2}^{-}\right)\\
\xrightarrow{\Delta_1}
&\Big(\gamma_1^{\epsilon_1},\dots,\gamma_{\overline{n}}^{\epsilon_{\overline{n}}}\Big) \left((\frac{1}{2}-2d)^+,\dots,(2e-\frac{3}{2})^+\right) \left((\frac{1}{2}-2f)^+,\dots,\frac{-3}{2}^+\right)\left((\frac{3}{2}-2c)^-,\dots,\frac{-1}{2}^{-}\right)  \\
\xrightarrow{\Delta_2}
&\Big(\gamma_1^{\epsilon_1},\dots,\gamma_{\overline{n}}^{\epsilon_{\overline{n}}}\Big) \left((\frac{1}{2}-2d)^+,\dots,\frac{-3}{2}^+\right)\left((\frac{1}{2}-2f)^+,\dots,\frac{-3}{2}^+\right)\left((\frac{3}{2}-2c)^-,\dots,\frac{-1}{2}^{-}\right)  \left((\frac{3}{2}-2e)^-,\dots,\frac{-1}{2}^-\right) \\
\xrightarrow{\Delta_3}&\ (w_0\theta,w_0\phi),
\end{align*}
\normalsize where $\Delta_1$ is injective and well-defined by Proposition \ref{inj_v_x} as usual. The $\Delta_2$ is injective due to the injectivity of the ``$\Delta$" in Case II. As for $\Delta_3$, it is injective by applying Proposition \ref{type_A_inj} four times. So $\Delta_3 \circ \Delta_2 \circ \Delta_1$ is well defined and injective over $\mathrm{im}(\Sigma)$ for all spin-relevant $K$-types.
\end{proof}

Combining Lemma \ref{case_first_is_comp} and Lemma \ref{case_non_unitary_and_small_comp}, we have:
\begin{proposition} \label{prop-nonunit2}
Let $N_{1/2} \longleftrightarrow \begin{pmatrix} s_1 & \dots & s_p & * & s_{p+1} & \dots & s_q  \\ t_1 & \dots & t_p & ** & t_{p+1} & \dots & t_q \end{pmatrix}$
be such that 
\begin{itemize}
\item $s_i - t_i = 0$ or $1$; and
\item $\begin{pmatrix} * \\ ** \end{pmatrix} = \begin{pmatrix} a \\ b \end{pmatrix}$ or $\begin{pmatrix} c & d \\ e & f \end{pmatrix}$ are as given in Corollary \ref{cor-u} and Corollary \ref{cor-v}.
\end{itemize}
Consider $\pi := J_{D_{2n}}(M_{1/2} \cup M_{1/2}^h,N_{1/2} \cup N_{1/2}^h)$, and the induced module
\begin{equation} \label{eq-nonunitind2}
\mathrm{Ind}_{\prod_{i = 1}^{p} A_{2j_i-1} \times A_{2j-1} \times D_m}^{D_{2n}}\begin{pmatrix} \bigotimes_{i=1}^p \mathrm{Ind}_{A_{j_i-1}  \times A_{j_i-1}}^{A_{2j_i-1}}\left(\mathrm{comp}_{1/2}(j_i, \frac{1}{2}) \otimes \mathrm{comp}_{1/2}(j_i, \frac{1}{2})^h\right) \otimes \\
J_A (\widetilde{A}_{1/2} \cup \widetilde{A}_{1/2}^h)\ \otimes\   J_D(\widetilde{D}_{1/2} \cup \widetilde{D}_{1/2}^h)\end{pmatrix},
\end{equation}
where $j_i := s_i+t_i$, and $\widetilde{A}_{1/2}$, $\widetilde{D}_{1/2}$ are as given in \eqref{eq-uv}. Then $\pi$ and \eqref{eq-nonunitind2} have the same multiplicities for all $K$-types with highest weight $\eta(i)$ for all $0\leq i\leq n$.
\end{proposition} 

\begin{corollary} \label{cor-nonunit2}
The $(\mathfrak{g},K)$-module $\pi$ in Proposition \ref{prop-nonunit2} is not unitary, and the Hermitian form is indefinite on the $K$-type with highest weights $\eta(2a+1)$ (in Case I) or $\eta(2e+2)$ (in Case II) respectively.
\end{corollary}
\begin{proof}
By Corollary \ref{cor-u} and Corollary \ref{cor-v}, the induced module \eqref{eq-nonunitind2} is nonunitary on the $K$-type $\mathcal{V}_{\eta(2a+1)}$ or $\mathcal{V}_{\eta(2e+2)}$ respectively, since the type A factors of \eqref{eq-nonunitind2}
other than $J_A (\widetilde{A}_{1/2} \cup \widetilde{A}_{1/2}^h)$ are all Stein complementary series, and the type D factor $J_D(\widetilde{D}_{1/2} \cup \widetilde{D}_{1/2}^h)$ is also unitary
by the results in Section \ref{sec-unitary}. Consequently, the result follows from Proposition \ref{prop-nonunit2}.
\end{proof}

\subsection{General case} We now begin the proof of non-unitarity in the general case:

\begin{lemma} \label{lem-nonunit3} \mbox{}\\
\begin{itemize}
\item[(a)] Let $n = \sum_i (a_i+b_i)$, and
$$N_{1/2} \longleftrightarrow \begin{pmatrix} a_1 & \dots & a_k \\ b_1 & \dots & b_k  \end{pmatrix}, \quad \quad a_i < b_i, \forall \ 1\leq i\leq k.$$
Write $\Pi := J_{D_{n}}(M_{1/2}, N_{1/2})$. Then there exists $G_{I} = Spin(2n_I,\mathbb{C})$, and a Levi subgroup $L_I = \prod_r A_{u_r-1} \times D_n$ of $G_I$ such that the induced module $\mathrm{Ind}_{L_I}^{G_I}\left(\bigotimes_r \mathrm{comp}_{1/2}(u_r;\frac{1}{2}) \otimes \Pi\right)$
has an irreducible lowest $K$-type subquotient $J_{D_{n_I}}\left(M_{1/2}^I,N_{1/2}^I\right)$ with
$$N_{1/2}^I \longleftrightarrow \begin{pmatrix}\widetilde{a}_1 & \dots & \widetilde{a}_{n_I} \\ \widetilde{a}_1 & \dots & \widetilde{a}_{n_I} \ \end{pmatrix}, \quad \quad \widetilde{a}_1 = b_1\ \ \text{and}\ \ \widetilde{a}_{n_I} = b_n.$$

\item[(b)] Let $n = \sum_i (c_i+e_i)$, and 
$$N_{1/2} \longleftrightarrow \begin{pmatrix} c_1 & \dots & c_{\ell} \\ e_1 & \dots & e_{\ell}  \end{pmatrix}, \quad \quad e_i +1 < c_{i+1}, \forall 1\leq i\leq l-1.$$
Write $\Psi := J_{D_n}(M_{1/2}, N_{1/2})$. Then there exist $G_{II} = Spin(2n_{II},\mathbb{C})$ and a Levi subgroup $L_{II} = \prod_r A_{v_r-1} \times D_n$ of $G_{II}$ such that the induced module
$\mathrm{Ind}_{L_{II}}^{G_{II}}\left(\bigotimes_r \mathrm{comp}_{1/2}(v_r;\frac{1}{2}) \otimes \Psi\right)$
has an irreducible lowest $K$-type subquotient $J_{D_{n_{II}}}\left(M^{II}_{1/2},N_{1/2}^{II}\right)$ with
$$N_{1/2}^{II} \longleftrightarrow \begin{pmatrix} \widetilde{c}_1 & \dots & \widetilde{c}_{n_{II}}\\ \widetilde{c}_1 -1 & \dots & \widetilde{c}_{n_{II}}-1 \ \end{pmatrix},\quad \quad \widetilde{c}_1 = c_1\ \  \text{and} \ \ \widetilde{c}_{n_{II}} = c_l.$$
\end{itemize}
Moreover, the above results hold analogously for the Hermitian duals $\Pi^h$
and $\Psi^h$.
\end{lemma}
\begin{proof}
We only present the proof for (a). Let $N_{1/2} \longleftrightarrow \begin{pmatrix} a_1 & \dots & a_k \\ b_1 & \dots & b_k  \end{pmatrix}$ be as given above. Let $\mathrm{comp}_{1/2}(2a_1+1,\frac{1}{2})$ be the one-sided twisted complementary series of type $A_{2a_1}$ with $\nu$-parameter corresponding to $\begin{pmatrix} a_1+1 \\ a_1  \end{pmatrix}$ (cf. \eqref{eq-complementary}). Then the induced module
$$\mathrm{Ind}_{A_{2a_1}\times D_n}^{D_{n+2a_1+1}}\left(\mathrm{comp}_{1/2}(2a_1+1,\frac{1}{2}) \otimes \widetilde{\Pi}\right)$$
has lowest $K$-type subquotient corresponding to the $\nu$-parameter
$$\begin{pmatrix} {\bf a_1+1} & a_1 & \dots & a_{i-1} &a_i & a_{i+1} & \dots & a_k \\ b_1 & b_2 & \dots & b_{i} & {\bf a_1} & b_{i+1} & \dots & b_k  \end{pmatrix},$$
where $1\leq i\leq l$ such that $b_{i} > a_1 \geq b_{i+1}$.

By continuously inducing $\mathrm{comp}_{1/2}(2a_1+3,\frac{1}{2})$, $\dots$, $\mathrm{comp}_{1/2}(2b_1-1,\frac{1}{2})$ to the above induced module, the resulting module has $\sigma_L$-coordinate of the form
$$\begin{pmatrix} b_1 & a_1' & \dots  & a_{k'}' \\ b_1 & b_1' & \dots & b_{k'}'  \end{pmatrix},$$
with $a_{k'}', b_{k'}' \leq b_k$ by construction, noting that $\begin{pmatrix} a_1' & \dots  & a_{k'}' \\ b_1' & \dots & b_{k'}'  \end{pmatrix}$ satisfies $a_{i}'\leq b_{i}'$ for all $1\leq i\leq k'$ and $\sum_{i=1}^{k'} b_i'-\sum_{i=1}^{k'} a_i'<\sum_{i=1}^{k} b_i-\sum_{i=1}^{k} a_i$. It may no longer satisfy the hypothesis in (a). Suppose $1 \leq \ell \leq k'$ is the smallest number such that $a_m' < b_m'$. Then one automatically has $a_1' = b_1'$, $\dots$, $a_{m-1}' = b_{m-1}'$. Repeat the same induction on $a_m' < b_m'$, after inducing $(\sum_{i=1}^{k} b_i-\sum_{i=1}^{k} a_i)$ complementary series as above, the $a$-parameters will be equal to $b$-parameters.  
\end{proof}

\begin{example} \label{eg-ind}
We present an example for (b). Let $N_{1/2} \longleftrightarrow \begin{pmatrix} 5 & 4 & 4 \\ 2 & 2 &  0 \end{pmatrix}$, the strings of inductions needed to get to $\begin{pmatrix} \widetilde{c}_1 & \dots & \widetilde{c}_{n_{II}} \\ \widetilde{c}_1 -1 & \dots & \widetilde{c}_{n_{II}}-1 \ \end{pmatrix}$ is listed below. At each induction step, we use $\stackrel{a}{\longrightarrow}$ to denote the induction of a type $A_{2a-1}$-factor $\mathrm{comp}_{1/2}(2a,\frac{1}{2})$ with infinitesimal character corresponding to $\begin{pmatrix} a \\ a \end{pmatrix}$ by Equation \eqref{eq-complementary}:
\begin{align*}
&\begin{pmatrix} 5 & 4 & 4 \\ 2 & 2 &  0 \end{pmatrix}  \stackrel{3}{\to}
\begin{pmatrix} 5 & 4 & 4 & {\bf 3} \\ {\bf 3} & 2 & 2 & 0  \end{pmatrix}  \stackrel{4}{\to}
\begin{pmatrix} 5 &  {\bf 4} &  4 & 4 & 3 \\ {\bf 4} & 3 & 2 & 2 & 0  \end{pmatrix}  \stackrel{3}{\to}
\begin{pmatrix} 5 &  4 &  4 & 4 & {\bf 3} & 3 \\ 4 & 3 & {\bf 3} & 2 & 2 &  0 \end{pmatrix} \stackrel{3}{\to} \\
  & \begin{pmatrix} 5 &  4 &  4 & 4 & {\bf 3} & 3 & 3 \\ 4 & 3 & 3 & {\bf 3} & 2 & 2 & 0  \end{pmatrix} \stackrel{2}{\to}
\begin{pmatrix} 5 &  4 &  4 & 4 & 3 & 3 & 3 & {\bf 2} \\ 4 & 3 & 3 & 3 & 2 & 2 & {\bf 2} & 0 \end{pmatrix}  \stackrel{1}{\to}
\begin{pmatrix} 5 &  4 &  4 & 4 & 3 & 3 & 3 & 2 & {\bf 1} \\ 4 & 3 & 3 & 3 & 2 & 2 & 2 & {\bf 1} & 0\end{pmatrix}
\end{align*}
\end{example}

\begin{remark} \label{rmk-Xcoord}
In the case when $\sigma_L \longleftrightarrow \begin{pmatrix} x_1 & \dots & x_k \\ y_1 & \dots & y_k  \end{pmatrix}$
satisfies \eqref{eq-unitary}, the above induction arguments is used in \cite{Br99} to prove the unitarity of $\pi = J_{D_{2n}}J\left(M_{1/2} \cup M_{1/2}^h, N_{1/2} \cup N_{1/2}^h \right)$. More precisely, it is proved that at each step of parabolic induction in the proof of Lemma \ref{lem-nonunit3}, the resulting induced module is irreducible.

The irreducibility result relies heavily on the associated variety theorem and multiplicity one theorem in \cite[Proposition 3.1 -- 3.2]{Br99}. If \eqref{eq-unitary} is not satisfied, neither of them would necessarily hold.
\end{remark}

We are now in the position to prove Theorem \ref{thm-main}. Let $\pi = J_{D_{2n}}\left(M_{1/2} \cup M_{1/2}^h, N_{1/2} \cup N_{1/2}^h\right)$ such that $N_{1/2}$ does not satisfy \eqref{eq-unitary}, then it must be of the form:
$$N_{1/2} \longleftrightarrow \begin{pmatrix} x_1 & \dots & x_a & z_1 & \dots & z_b & x_{a+1} & \dots & x_c & z_{b+1} & \dots & z_d & \dots \\  y_1 & \dots & y_a & w_1 & \dots & w_b & y_{a+1} & \dots & y_c & w_{b+1} & \dots & w_d  & \dots \end{pmatrix},$$
where
\begin{itemize}
\item Each $\begin{pmatrix} x_1 & \dots & x_a \\  y_1 & \dots & y_a \end{pmatrix}$, $\begin{pmatrix} x_{a+1} & \dots & x_c \\  y_{a+1} & \dots & y_c \end{pmatrix}$, $\dots$ satisfies \eqref{eq-unitary};
\item $y_a +1 \geq z_1$, $y_c+1 \geq z_{b+1}$, $\dots$; and
\item Each $\begin{pmatrix} z_1 & \dots & z_b \\  w_1 & \dots & w_b \end{pmatrix}$, $\begin{pmatrix} z_{b+1} & \dots & z_d \\  w_{b+1} & \dots & w_d \end{pmatrix}$, $\dots$ is of the form given in Lemma \ref{lem-nonunit3}(a) or Lemma \ref{lem-nonunit3}(b). In particular, we
further require $w_b+1 < x_{a+1}$ (respectively $w_d+1 < x_{c+1}$, $\dots$) if the $z_i, w_i$-coordinates is in the latter case.
\end{itemize}

\medskip
We divide the proof into 3 steps:

\noindent {\bf Step 1:} By applying Lemma \ref{lem-nonunit3} on the parameters $\begin{pmatrix} z_{b+1} & \dots & z_d \\ w_{b+1} & \dots & w_d  \end{pmatrix}$, $\dots$, one can induce $J(M_{1/2},N_{1/2})$ (resp. $J(M_{1/2}^h,N_{1/2}^h)$) and hence $\pi$ to a larger $Spin$ group $G^+ = Spin(4n^+,\mathbb{C})$ by some $\mathrm{comp}_{1/2}(j,\frac{1}{2})$ (resp. $\mathrm{comp}_{1/2}(j,\frac{1}{2})^h$) such that
$$\mathrm{Ind}_{M^+}^{G^+}\left( \bigotimes_r \mathrm{Ind}_{A_{j_r}-1 \times A_{j_r}-1}^{A_{2j_r-1}}\left(\mathrm{comp}_{1/2}(j_r+1,\frac{1}{2}) \otimes \mathrm{comp}_{1/2}(j_r+1,\frac{1}{2})^h\right) \otimes \pi\right)$$
has the lowest $K$-type subquotient $\pi^+ =  J_{D_{2n^+}}(M^+_{1/2} \cup (M^+_{1/2})^h, N^+_{1/2} \cup (N^+_{1/2})^h)$ satisfying
$$N^+_{1/2} \longleftrightarrow \begin{pmatrix} x_1 & \dots & x_a & z_1 & \dots & z_b & x_{a+1} & \dots & x_c & z_{b+1}^+ & \dots & z_{d^+}^+ & \dots \\  y_1 & \dots & y_a & w_1 & \dots & w_b & y_{a+1} & \dots & y_c & w_{b+1}^+ & \dots & w_{d^+}^+  & \dots \end{pmatrix},$$
where all $z_q^+ - w_q^+$ are equal to $0$ or $1$ for all $q \geq b+1$.

\medskip
\noindent {\bf Step 2:}  By applying the proof of Lemma \ref{lem-nonunit3} to the parameters $\begin{pmatrix} z_1 & \dots & z_b \\ w_1 & \dots & w_b  \end{pmatrix}$, one can induce $J(M^+_{1/2}, N^+_{1/2})$ (resp. $J((N^+_{1/2},(N^+_{1/2})^h))$ and hence $\pi_+$ by some $\mathrm{comp}_{1/2}(j,\frac{1}{2})$ (resp. $\mathrm{comp}_{1/2}(j,\frac{1}{2})^h$) to obtain an induced module in $G^{++} = Spin(4n^{++},\mathbb{C})$, whose lowest $K$-type subquotient $\pi^{++} =  J_{D_{2n^{++}}}(M^{++}_{1/2} \cup (M^{++}_{1/2})^h, N^{++}_{1/2} \cup (N^{++}_{1/2})^h)$ satisfying
\begin{align*}
&N^{++}_{1/2} \longleftrightarrow \\
&\begin{pmatrix} x_1 & \dots & x_a & z_1' & \dots & z_u' & * & z_{u+1}' & \dots & z_{b^+}' & x_{a+1} & \dots & x_c & z_{b+1}^+ & \dots & z_{d^+}^+ & \dots \\  y_1 & \dots & y_a & w_1' & \dots & w_u' & ** & w_{u+1}' & \dots & w_{b^+}' & y_{a+1} & \dots & y_c & w_{b+1}^+ & \dots & w_{d^+}^+  & \dots \end{pmatrix},\end{align*}
where all $z_i', w_i'$ satisfy \eqref{eq-unitary}, and $\begin{pmatrix} * \\ ** \end{pmatrix} = \begin{pmatrix} a \\ b \end{pmatrix}$ or $\begin{pmatrix} c & d \\ e & f \end{pmatrix}$ as in Corollary \ref{cor-u} or Corollary \ref{cor-v}.

\medskip
\noindent {\bf Step 3:}  By applying Remark \ref{rmk-Xcoord} to the $x_i, y_i$ and $z_i', w_i'$ coordinates in Step 2,
one can induce $\pi^{++}$ further to get a induced module whose lowest $K$-type subquotient $\pi^{+++} = J_{D_{2n^{+++}}}(M^{+++}_{1/2} \cup (M_{1/2}^{+++})^h,N_{1/2}^{+++} \cup (N_{1/2}^{+++})^h)$ satisfying
$$N^{+++}_{1/2} \longleftrightarrow \begin{pmatrix} s_1 & \dots & s_p & * & s_{p+1} & \dots & s_q \\  t_1 & \dots & t_p & ** & t_{p+1} & \dots & t_q  \end{pmatrix},$$
as in the hypothesis of Proposition \ref{prop-nonunit2}.

\medskip
To conclude, $\pi^{+++}$ is the lowest $K$-type subquotient of the induced module
\begin{equation} \label{eq-nonunit4}
\mathrm{Ind}_{M^{+++}}^{G^{+++}}\left( \bigotimes_w \mathrm{Ind}_{A_{j_w}-1 \times A_{j_w}-1}^{A_{2j_w-1}}\left(\mathrm{comp}_{1/2}(j_w+1,\frac{1}{2}) \otimes \mathrm{comp}_{1/2}(j_w+1,\frac{1}{2})^h\right) \otimes \pi\right).
\end{equation}
Suppose on the contrary that $\pi$ is unitary, then the inducing module of \eqref{eq-nonunit4} is unitary, since the other constituents
are all Stein complementary series. Consequently,
all the irreducible constituents of
\eqref{eq-nonunit4} are unitary.
However, Corollary \ref{cor-nonunit2} implies that $\pi^{+++}$ is not unitary. So we have a contradiction, implying that $\pi$
is not unitary. \qed

\section{Genuine unitary dual of $Spin(2n+1,\mathbb{C})$} \label{sec-odd}
In this section, we deal with the genuine unitary dual of $G=Spin(2n+1,\mathbb{C})$ similarly to the case of $Spin(2n,\mathbb{C})$. Fix a choice of positive simple roots
given by the Dynkin diagram:
\begin{center}
\begin{tikzpicture}
\draw
    (-01,0) node[circle,draw=black, fill=white,inner sep=0pt,minimum size=5pt,label=below:{\tiny $e_1 - e_2$}] {}
 -- (0,0) node[circle,draw=black, fill=white,inner sep=0pt,minimum size=5pt,label=above:{\tiny $e_2 - e_3$}] {};

\draw (1,0) node {$\dots$};


\draw
(2,0) node[circle,draw=black, fill=white,inner sep=0pt,minimum size=5pt,label=below:{\tiny $e_{n-2} - e_{n-1}$}] {}
 (3,0) node[circle,draw=black, fill=white,inner sep=0pt,minimum size=5pt,label=above:{\tiny $e_n$}] {};
\draw[thin](2.07,0.05)--(2.93,0.05);
\draw[thin](2.07,-0.05)--(2.93,-0.05);
\draw[thin](0.1,0)--(0.4,0);
\draw[thin](1.6,0)--(1.9,0);
\draw (2.5,0) node {$>$}; 
\end{tikzpicture}\end{center} 

\medskip
Let $J(\mu,\nu)=J(\lambda_L;\lambda_R)$ be an irreducible Hermitian $(\mathfrak{g},K)$-module such that $\nu$ has real coordinates. 
By the bottom-layer argument as in type D, the classification problem can be reduced to the case of $\mu=(\frac{1}{2},\dots,\frac{1}{2})$.

Similar to the type D case, we make the following definition:
\begin{definition}
For $0\leq i\leq n$, let $\eta(i)=(\underbrace{\frac{3}{2},\dots,\frac{3}{2}}_i,\underbrace{\frac{1}{2},\dots,\frac{1}{2}}_{n-i})$. The $K$-types $\mathcal{V}_{\eta(i)}$ are called \textbf{spin-relevant}\ $K$-\textbf{types}.
\end{definition}

For any $t\in (-1,1]$, let $N_t$ be all the coordinates $\nu_i$ of $\nu$ satisfying $\nu_i - t\in 2\mathbb{Z}$, and let $M_t$ be the corresponding $\mu$-parameter. Assume that $J(\mu, \nu)$ is unitary, with
$$(\mu,\nu) = \mathop{\bigsqcup}\limits_{t\in T} (M_t, N_t)$$
for a finite subset $T \subseteq (-1,1]$. If $t\in [0,\frac{1}{2})\cup (\frac{1}{2},1]$, then by similar arguments as in the paragraphs below Equation \eqref{eq-bv}, we have 
\[\begin{cases}J_A(M_t\cup M_{-t} \cup -M_{1-t}\cup -M_{-1+t},N_t\cup N_{-t} \cup N_{1-t}\cup N_{-1+t}), & t\in (0,\frac{1}{2})\cup (\frac{1}{2},1)\\ 
J_A(M_0\cup -M_{1},N_0\cup N_{1}), & t = 0, 1\end{cases}\]
must be Stein complementary series of type A. So we can further reduce our study of unitarity to that of 
\begin{equation} \label{eq-bhalf}
    J(M_{1/2} \cup M_{1/2}^h,N_{1/2} \cup N_{1/2}^h) = \mathrm{Ind}^{B_{2n}}_{B_n \times B_n}(J(M_{1/2},N_{1/2}) \otimes J(M_{1/2}^h,N_{1/2}^h)).
\end{equation}

\subsection{Intertwining operators for type B}
In this section, we study the intertwining operator of
principal series modules over
$\mathcal{V}_{\eta(i)}$. 
 The results are
analogous to that of Section \ref{sec-nonunit}, and we will be brief on some of the proofs in this section. 

Notice that all the results in section \ref{sec-in} are applicable in type B. In particular, by \eqref{phi_+-}, the intertwining operator on the short root in type B
\begin{equation} \label{eq-tau}
\tau: \left(a^{\epsilon}\right) \longrightarrow \left(-a^{-\epsilon}\right)
\end{equation}
is isomorphic for all spin-relevant $K$-types except when $a = \pm \frac{3}{2}$.

\subsection{A key reduction}
In Section \ref{sec-main}, an important ingredient  in the statement of the main theorem is that for $\mu = (\frac{1}{2},\dots,\frac{1}{2})$, the type D spin-relevant $K$-type $\eta(1) = (\frac{3}{2},\frac{1}{2},\dots,\frac{1}{2},-\frac{1}{2})$ is type $A_{n-1}$ bottom-layer. However, this is no longer true for
type B. 

To remedy this issue, we begin by generalizing Proposition \ref{inj_v_x} and \ref{type_A_inj}:
\begin{lemma} \label{lem-eta1}
Let $p+q = n$, and $(\nu_1^{\delta}, \dots, \nu_p^{\delta})$ and $(\xi_1^{\epsilon_1}, \dots, \xi_q^{\epsilon_q})$ be the $(\mu,\nu)$-parameters
in Proposition \ref{inj_v_x} and \ref{type_A_inj}. Suppose $\epsilon_i = -\delta$ for all $1 \leq i \leq p$. Then the intertwining operators in Proposition \ref{inj_v_x} and Proposition \ref{type_A_inj} is always well-defined and injective on the $\mathcal{V}_{\eta(1)}$-isotropic subspaces.
\end{lemma}
\begin{proof}
    In both cases, the first occurrence (if at all) of non-injectivity or non-well definedness is at the $\mathcal{A}_{\eta}$-isotropic subspace with $\eta = (\frac{1}{2}, \dots, \frac{1}{2},\frac{-1}{2},\dots,\frac{-1}{2}) + (1,\dots,1,0,\dots,0,-1,\dots,-1)$ with an equal positive number ($\leq \min\{p,q\}$) of $1$'s and $-1$'s. So $\eta$ has an even number of $\pm \frac{3}{2}$'s, and the result follows from induction in stages.
\end{proof}

\begin{proposition}\label{Theta_1}
Let $J(M_{1/2},N_{1/2})$ be as given in Equation \eqref{eq-bhalf}. Construct the strings of coordinates $N^{A}_{1/2}:=\alpha$, $N^{B}_{1/2}:=(\beta_1,\dots,\beta_k)$ as follows:
\begin{itemize} 
\item[(i)] $\beta_1$ is the longest subsequence of $N_{1/2}$ such that the coordinates of $\beta_1$ contain $\frac{-3}{2}$ and, upon rearranging into {\bf ascending} order, all adjacent coordinates of $\beta_1$ differ by $2$. 
\item[(ii)] Remove a copy of the coordinates in (i), and repeat (i) on the remaining coordinates to obtain $\beta_2,\dots,\beta_k$, until there is no $\frac{-3}{2}$ coordinates left.
\item[(iii)] Let $\alpha$ be the remaining coordinates with ascending order. 
\end{itemize}
Let $M^A_{1/2}$, $M^B_{1/2}$ be  $\mu$-parameters corresponding to $N^{A}_{1/2}$, $N^{B}_{1/2}$. Then the irreducible module $J(M_{1/2} \cup M_{1/2}^h, N_{1/2}\cup N_{1/2}^h)$  and 
\[\mathrm{Ind}_{M}^G\big(J_A(M^A_{1/2} \cup M^A_{-1/2},N^A_{1/2}\cup N^A_{-1/2})\otimes J_B(M^B_{1/2}\cup M^B_{-1/2}, N^B_{1/2}\cup N^B_{-1/2})\big)\] 
have the same multiplicities on the first spin-relevant $K$-type $\mathcal{V}_{\eta(1)}$.
\end{proposition}

\begin{example} \label{eg-typeb}
    Let $(M_{1/2},N_{1/2}) = (\frac{13}{2}^+,\frac{5}{2}^+,\frac{1}{2}^+,\frac{-3}{2}^+,\frac{-7}{2}^+,\frac{-7}{2}^+,\frac{-11}{2}^+)$, then
    $$N^B_{1/2}=\beta_1 = (\frac{-11}{2},\frac{-7}{2},\frac{-3}{2},\frac{1}{2},\frac{5}{2}); \quad N^A_{1/2} = (\frac{-7}{2},\frac{13}{2}).$$
\end{example}

\begin{proof}
As of the notations in Section \ref{sec-nonunit}, let $(\theta,\phi)$ and $(\theta_{\beta},\phi_{\beta})$ be the dominant form of $(M_{1/2},N_{1/2})$ and $(M_{1/2}^B,N_{1/2}^B)$. Also, let $(\theta_{\alpha},\phi_{\alpha})$ be the type A dominant form of $(M_{1/2}^A,N_{1/2}^A)$ (similar notations are respectively defined for their Hermitian duals). As in Section \ref{sec-u} and Section \ref{sec-v}, one needs to show the injectivity of $\omega$ (resp. $\omega^h$ ) over the image of $\zeta$ (resp. $\zeta^h$): 
\begin{align*}
\left((\theta_{\alpha},\phi_{\alpha});(\theta_{\beta},\phi_{\beta})\right) \xrightarrow{\zeta} & 
\left((w_A\theta_{\alpha},w_A\phi_{\alpha}); (w_B\theta_{\beta}, w_B\phi_{\beta})\right)
\xrightarrow{\omega}
(w_0\theta,w_0\phi),\\ 
\left((\theta_{\alpha}^h,\phi_{\alpha}^h);(\theta_{\beta}^h,\phi_{\beta}^h) \right)\xrightarrow{\zeta^h} & 
\left((w_A\theta_{\alpha}^h,w_A\phi_{\alpha}^h); (w_B\theta_{\beta}^h, w_B\phi_{\beta}^h)\right)
\xrightarrow{\omega^h}
(w_0\theta^h,w_0\phi^h)
\end{align*}
where $w_A$ and $w_B$ are the longest Weyl elements of type A and B respectively. 

As in the proof of Proposition \ref{two_string_mult}, we study the following operators instead of $\omega \circ \beta$:
\[
\left((\theta_{\alpha},\phi_{\alpha});(\theta_{\beta},\phi_{\beta})\right)
\xrightarrow{\Sigma}  
\left(\alpha^+; (\beta_1^+,\dots,\beta_k^+)\right)
\xrightarrow{\Delta}
(w_0\theta,w_0\phi)\]
where $\bullet^{\delta}$ is a $(\mu,\nu)$-parameter with $\mu = (\frac{\delta}{2},\dots,\frac{\delta}{2})$ and $\nu = \bullet$.
One can check that $\Sigma$ is part of the intertwining operator $\zeta = \iota_{(w_A,w_B)}$, and $\mathrm{im}(\Sigma)$ lies in the induced module $$\mathrm{Ind}_{M'}^G\left(J_A(\alpha^+) \otimes J_A({\beta}_1^+) \otimes \dots \otimes J_A({\beta}_k^+)\right).$$
Therefore, it suffices to show that $\Delta$ is injective over the above induced module on the $\mathcal{V}_{\eta(1)}$-isotropic subspaces.

There are two types of $\alpha^+$'s:
\begin{align*}
\overline{\alpha}^+ & = (W^+,\dots,W^+,(W+2)^+,\dots,(W+2)^+,\dots),\quad (W \geq  \frac{1}{2});\\ \underline{\alpha}^+ & = (\dots, (w-2)^+,\dots,(w-2)^+,w^+,\dots,w^+),\quad (w \leq  \frac{-7}{2}) \end{align*}
so $\alpha^+=(\underline{\alpha}^+,\overline{\alpha}^+)$. Also, for later purposes, we define $\widetilde{\beta}_i^{\pm}$ by the following: if $\beta_i := ((\frac{1}{2} - 2y_i), \dots, \frac{-3}{2}, \frac{1}{2}, \dots, (2x_i - \frac{3}{2}))$, then 
$$\widetilde{\beta}_i^+ := \left((\frac{1}{2} - 2y_i)^+, \dots, \frac{-3}{2}^+\right)\quad \quad
\widetilde{\beta}_i^- := \left((\frac{3}{2} - 2x_i)^-, \dots, \frac{-1}{2}^-\right).$$

Split $\Delta$ into the following:
\begin{align*}
(\underline{\alpha}^+)( \overline{\alpha}^+)(\beta_1^+,\dots,\beta_k^+)  
&\xrightarrow{\iota_1}
(\underline{\alpha}^+)(\beta_1^+,\dots,\beta_k^+)(\overline{\alpha}^+)\\ 
&\xrightarrow{\iota_2}
(\underline{\alpha}^+)(\beta_1^+,\dots,\beta_k^+)(-w_A\overline{\alpha}^-)\\
&\xrightarrow{\iota_3}
(\underline{\alpha}^+) (-w_A\overline{\alpha}^-) (\beta_1^+,\dots,\beta_k^+)\\
&\xrightarrow{\iota_4}
(\underline{\alpha}^+) (-w_A\overline{\alpha}^-) (\widetilde{\beta}_1^+,\dots,\widetilde{\beta}_k^+)(\widetilde{\beta}_1^-,\dots,\widetilde{\beta}_k^-)\\
&\xrightarrow{\iota_5}
z\left((\underline{\alpha}^+)(-w_A\overline{\alpha}^+)\right)(\widetilde{\beta}_1^+,\dots,\widetilde{\beta}_k^+)(\widetilde{\beta}_1^-,\dots,\widetilde{\beta}_k^-)\\
&\xrightarrow{\iota_6}
(w_0\theta,w_0\phi)
\end{align*}
where $w_A$ is the long Weyl group element of type A, so that $-w_A\overline{\alpha}$ is of {\bf descending} order. 

For $\iota_1$, it can be decomposed as the product of the intertwining operators 
\[(x^+)\left((\frac{1}{2}-2y_i)^+,\dots,(2x_i-\frac{3}{2})^+\right)\to \left((\frac{1}{2}-2y_i)^+,\dots,(2x_i-\frac{3}{2})^+\right)(x^+) \]
for $x\in \overline{\alpha}$. By definition of $\beta_i$'s, $x\neq (\frac{1}{2}-2y_i)-2, (2x_i-\frac{3}{2})+2$, so the intertwining operators are well defined and injective by Proposition \ref{inj_v_x}. 

For $\iota_2$, it can be written as composition of the intertwining operators $\tau$ and 
 intertwining operators given in Lemma \ref{lem-eta1} which are injective on $\mathcal{V}_{\eta(1)}$. Since $\overline{\alpha}$ contains no $\frac{3}{2}$, $\tau$ is injective on $\mathcal{V}_{\eta(1)}$ by \eqref{eq-tau}.  

For $\iota_3$ and $\iota_5$, they are some intertwining operators given in (b) of Lemma \ref{int_of_1/2_-1/2}. Similar  reason to $\iota_2$ implies they are injective on $\mathcal{V}_{\eta(1)}$. 

For $\iota_4$, omitting ``$(\underline{\alpha}^+) (-w_A\overline{\alpha}^-)$", write it as the composition of
\begin{align*} 
\bigcup_{i} \left((\frac{1}{2}-2y_i)^+,\dots,(2x_i-\frac{3}{2})^+\right)  
\to & \bigcup_i \left((\frac{1}{2}-2y_i)^+,\dots,\frac{-3}{2}^+\right)  \cup \bigcup_{i}\left(\frac{1}{2}^+,\dots,(2x_i-\frac{3}{2})^+\right)\\
\to & \bigcup_i \left((\frac{1}{2}-2y_i)^+,\dots,\frac{-3}{2}^+\right)  \cup \bigcup_i \left((\frac{3}{2}-2x_i)^-,\dots,\frac{-1}{2}^-\right).
\end{align*}
Then the first $\to$ is well defined and injective by Proposition \ref{inj_v_x}. The second $\to$ are some intertwining operators given in Lemma \ref{lem-eta1}, which are injective on $\mathcal{V}_{\eta(1)}$. So is $\iota_4$.  

For $\iota_6$, its injectivity is due to Proposition \ref{type_A_inj} as usual. 

\medskip
As for $\omega^h \circ \beta^h$, consider
\[
\left((\theta_{\alpha}^h,\phi_{\alpha}^h);(\theta_{\beta}^h,\phi_{\beta}^h)\right)
\xrightarrow{\Sigma^h}  
\left((\alpha^h)^+; (\beta_1^-,\dots,\beta_k^-)\right)
\xrightarrow{\Delta^h}
(w_0\theta^h,w_0\phi^h)\]
It suffices to show the injectivity of $\Delta^h|_{\Sigma^h}$, which is completely similar to that of $\Delta|_{\Sigma}$ . Now the lemma follows. \end{proof}

\begin{example} \label{eg-balpha}
We continue with Example \ref{eg-typeb} to illustrate the proof of the Proposition \ref{Theta_1}. In this case,
one needs to study the injectivity of $\Delta$ and $\Delta^h$ given by:
\begin{align*}
      (\frac{-7}{2}^+)(\frac{13}{2}^+)(\frac{-11}{2}^+,\frac{-7}{2}^+,\frac{-3}{2}^+,\frac{1}{2}^+,\frac{5}{2}^+) &\xrightarrow{\Delta} (w_0\theta,w_0\phi)\\
      (\frac{-13}{2}^+)(\frac{7}{2}^+)(\frac{-11}{2}^-,\frac{-7}{2}^-,\frac{-3}{2}^-,\frac{1}{2}^-,\frac{5}{2}^-) &\xrightarrow{\Delta^h} (w_0\theta^h,w_0\phi^h)
      \end{align*}
We first study $\Delta$ - it consists of the following operators:
\begin{align*}
      (\frac{-7}{2}^+)(\frac{13}{2}^+)(\frac{-11}{2}^+,\frac{-7}{2}^+,\frac{-3}{2}^+,\frac{1}{2}^+,\frac{5}{2}^+) &\to (\frac{-7}{2}^+)(\frac{-11}{2}^+,\frac{-7}{2}^+,\frac{-3}{2}^+,\frac{1}{2}^+,\frac{5}{2}^+)(\frac{13}{2}^+) \\
      &\to (\frac{-7}{2}^+)(\frac{-11}{2}^+,\frac{-7}{2}^+,\frac{-3}{2}^+,\frac{1}{2}^+,\frac{5}{2}^+)(\frac{-13}{2}^-) \\
       &\to (\frac{-7}{2}^+)(\frac{-13}{2}^-)(\frac{-11}{2}^+,\frac{-7}{2}^+,\frac{-3}{2}^+)(\frac{1}{2}^+,\frac{5}{2}^+) \\
      &\to (\frac{-13}{2}^-)(\frac{-7}{2}^+)(\frac{-11}{2}^+,\frac{-7}{2}^+,\frac{-3}{2}^+)(\frac{-5}{2}^-,\frac{-1}{2}^-) \\  
      &\to (\frac{-13}{2}^-,\frac{-11}{2}^+,\frac{-7}{2}^+,\frac{-7}{2}^+,\frac{-5}{2}^-,\frac{-3}{2}^+,\frac{-1}{2}^-)
\end{align*}
As for $\Delta^h$, we study the following operators:
{\allowdisplaybreaks \begin{align*}
    (\frac{-13}{2}^+)(\frac{7}{2}^+)(\frac{-11}{2}^-,\frac{-7}{2}^-,\frac{-3}{2}^-,\frac{1}{2}^-,\frac{5}{2}^-) &\to  (\frac{-13}{2}^+)(\frac{-11}{2}^-,\frac{-7}{2}^-,\frac{-3}{2}^-,\frac{1}{2}^-,\frac{5}{2}^-)(\frac{7}{2}^+) \\
    &\to (\frac{-13}{2}^+)(\frac{-11}{2}^-,\frac{-7}{2}^-,\frac{-3}{2}^-,\frac{1}{2}^-,\frac{5}{2}^-) (\frac{-7}{2}^-) \\
    &\to (\frac{-13}{2}^+) (\frac{-7}{2}^-) (\frac{-11}{2}^-,\frac{-7}{2}^-,\frac{-3}{2}^-)( \frac{1}{2}^-,\frac{5}{2}^-)  \\
    &\to  (\frac{-13}{2}^+) (\frac{-7}{2}^-) (\frac{-11}{2}^-,\frac{-7}{2}^-,\frac{-3}{2}^-)(\frac{-5}{2}^+,\frac{-1}{2}^+)\\
    &\to  (\frac{-13}{2}^+,\frac{-11}{2}^-,\frac{-7}{2}^-,\frac{-7}{2}^-,\frac{-5}{2}^+,\frac{-3}{2}^-,\frac{-1}{2}^+).
\end{align*}}
All of the above operators are injective up to the $\mathcal{V}_{\eta(1)}$-isotropic spaces by \eqref{eq-tau}, Proposition \ref{inj_v_x}, Proposition \ref{type_A_inj} and Lemma \ref{lem-eta1}.
\end{example}

\begin{theorem}\label{key_reduction}
Let $\pi = J(M_{1/2}\cup M_{1/2}^h, N_{1/2} \cup N_{-1/2}^h)$ be as given in Equation \eqref{eq-bhalf}. Suppose
$\pi$ is unitary, then $N_{1/2}$ must be the union of sequences of the form
\begin{itemize}
    \item $\left(\dots, \frac{5}{2}, \frac{1}{2}, \frac{-3}{2}, \dots \right)$ or $\left( \frac{-3}{2}, \frac{-7}{2}, \dots \right)$ or $\left(\dots, \frac{1}{2}, \frac{-3}{2} \right)$, all appearing in $N^B_{1/2}$; or
    \item $\left(\dots, \frac{5}{2}, \frac{1}{2} \right)$, which
    appears in  $N^A_{1/2}$.
\end{itemize} 
Equivalently, 
\begin{equation}
\label{eq-bsigma}    
N_{1/2} = \begin{cases} \mathop{\bigcup}\limits_{i=1}^l (\frac{1}{2}+2(x_i-1), \dots, \frac{1}{2}, \frac{-3}{2}, \dots, \frac{1}{2} - 2y_i) \cup \mathop{\bigcup}\limits_{j=l+1}^k (\frac{-3}{2}, \dots, \frac{1}{2} - 2y_i),\ or\\
\mathop{\bigcup}\limits_{i=1}^l (\frac{1}{2}+2(x_i-1), \dots, \frac{1}{2}, \frac{-3}{2}, \dots, \frac{1}{2} - 2y_i) \cup \mathop{\bigcup}\limits_{j=l+1}^k (\frac{1}{2}+2(x_j-1), \dots, \frac{1}{2}) 
\end{cases}
 \end{equation}
for two descending chains of {\bf positive} integers $x_i \geq x_{i+1}$ and $y_i \geq y_{i+1}$. We will write
\begin{equation} \label{eq-sigmab}
\begin{pmatrix}
    x_1 & \dots & x_l & x_{l+1} & \dots & x_{k} \\
    y_1 & \dots & y_l & y_{l+1} & \dots & y_k
\end{pmatrix}
\end{equation}
with $x_i, y_i > 0$ for all $0 \leq i \leq l$, and $x_jy_j = 0$ for $l+1 \leq j \leq k$ for the $N_{1/2}$ parameter given \eqref{eq-bsigma}. 
\end{theorem}
\begin{proof}
Suppose on the contrary that $N^{-1/2}$ does not satisfy the hypothesis. Let 
$N^A_{1/2}$ as Proposition \ref{Theta_1}. Decompose $N^A_{1/2}$ as the union of descend subsequences as the way decomposing ``$\nu_1$" in subsection \ref{sec-pseudo}. By hypothesis, $N^A_{1/2}$ either contains a descend subsequence $(\dots,s,t,\cdots)$ with $t-s>2$, or contains a descend subsequence $(w, w-2, \dots)$ with $w \leq -\frac{7}{2}$, or 
$(\dots,w-2,w)$ with $w \geq \frac{5}{2}$.
By Lemma \ref{lem-dirac} and Theorem \ref{thm-comp}, the module
$J_A(M^A_{1/2}\cup M^A_{-1/2} , N^A_{1/2} \cup N^A_{-1/2})$ appearing in the induced module in Proposition \ref{Theta_1} must have indefinite form on the $K$-type with highest weight $(\frac{3}{2},\frac{1}{2},\dots,\frac{1}{2},-\frac{1}{2})$. By Proposition \ref{Theta_1} and induction in stages, $\pi$ has indefinite form
on $\mathcal{V}_{\eta(1)}$.
\end{proof}

\subsection{Main theorem}
\begin{theorem}\label{Main_odd}
Let $\pi = J(M_{1/2}\cup M_{1/2}^h, N_{1/2}\cup N_{1/2}^h)$ be as in Equation \eqref{eq-bhalf} and $N_{1/2}$ is given by Equation \eqref{eq-sigmab}. Then $\pi$ is unitary if and only if $N_{1/2}$ satisfies
\begin{equation}\label{condition_odd}
y_i+1 \geq x_i \quad \quad \text{and}\quad \quad x_{i}  \geq y_{i+1}
\end{equation}
for all $i$. Otherwise, its invariant Hermitian form is indefinite at some spin-relevant $K$-types. In particular, the parameters must be of the form 
\[\begin{pmatrix}
    x_1 & \dots & x_{l} & 0  \\
    y_1 & \dots & y_{l} & y_{l+1}
\end{pmatrix} \quad \quad \text{or} \quad \quad
\begin{pmatrix}
    x_1 & \dots & x_l & 1 & \dots & 1 \\
    y_1 & \dots & y_l & 0 & \dots & 0
\end{pmatrix}.\]
\end{theorem}

\begin{remark} \label{rmk-unipb}
Note that if the equality holds, then they give parameters for the Stein complementary series. Consequently, 
the arguments in Section \ref{sec-unitary} allows us to reduce our unitarity proof
to considering \eqref{condition_odd} with strict inequalities, whose $N_{1/2}$-parameter must be of the form:
\[\begin{pmatrix}
    x_1 & \dots & x_l & x_{l+1} \\
    y_1 & \dots & y_l & y_{l+1}
\end{pmatrix}, \quad y_{i} \geq x_i > y_{i+1} .\]
These representations are studied in \cite[Section 7.2]{B17}, which are unipotent representations corresponding to the orbit $\mathcal{B} = \bigcup_{i=1}^{l+1} (2y_i+1, 2y_i, 2x_i, 2x_i-1),$ where $(2x_{l+1},2x_{l+1}-1) = (0,0)$ if $x_{l+1} = 0$.
The unitarity of these modules is stated in \cite{B17}, and the proof is identical to that of Brega representations in type D.
\end{remark}
 The non-unitarity of the Hermitian module which doesn't satisfy the condition \eqref{condition_odd} is proved similarly as the type D case. We first deal with two fundamental cases (``Case I" and ``Case II"), and then reduce the general case to one of the two cases. We will only give details on when the proofs are different from the Type D case.

\subsection{Base Case I}
\begin{proposition}\label{U_case_odd}
Let $n = a+b$, and $N_{1/2} = \begin{pmatrix} a \\ b \end{pmatrix}$ with $a > b+1$. Then the module $\mathrm{Ind}_{A_{2n-1}}^{B_{2n}} \big(J_A(M_{1/2} \cup M_{1/2}^h, N_{1/2} \cup N_{1/2}^h)\big)$ and $J_{B_{2n}}(M_{1/2} \cup M_{1/2}^h, N_{1/2} \cup N_{1/2}^h)$ have the same multiplicity of the $K$-type $\mathcal{V}_{\eta(i)}$ for all $0\leq i\leq n$. 
\end{proposition}

\begin{proof}
As before, the proposition holds more generally for $a \geq b$ which we will prove below. As in Equation \eqref{eq-lema1} and \eqref{eq-lema12}, consider the intertwining operators:
\begin{align*}
\left((\frac{1}{2}-2b)^+, \dots, \frac{-3}{2}^+, \frac{1}{2}^+, \frac{5}{2}^+, \dots, (2a-\frac{3}{2})^+\right) 
&\xrightarrow{\omega}  (w_0\theta,w_0\phi) \\
\left((\frac{3}{2}-2a)^+, \dots, \frac{-5}{2}^+, \frac{-1}{2}^+, \frac{3}{2}^+, \dots, (2b-\frac{1}{2})^+\right) 
&\xrightarrow{\omega^h}  (w_0\theta^h,w_0\phi^h),\end{align*}
and show that they are injective over $\mathrm{Ind}_{A_{n-1}}^{B_{n}} \big(J_A(M_{1/2}, N_{1/2})\big)$ and $\mathrm{Ind}_{A_{n-1}}^{B_{n}} \big(J_A(M_{1/2}^h, N_{1/2}^h)\big)$ over the $\mathcal{V}_{\eta(i)}$-isotropic subspaces for all $ 0\leq i\leq n$.

We now study the injectivity of $\omega$. We split $\omega$ into the following operators:
\begin{align*}
\left((\frac{1}{2}-2b)^+, \dots, (2a-\frac{3}{2})^+\right) 
\to &\left((\frac{1}{2}-2b)^+, \dots, \frac{1}{2}^+,\frac{5}{2}^+\right) \left((\frac{3}{2}-2a)^-,\dots,\frac{-9}{2}^-\right) \\
\to & \left((\frac{1}{2}-2b)^+, \dots, \frac{-3}{2}^+\right) \left((\frac{3}{2}-2a)^-,\dots,\frac{-9}{2}^-\right) \left(\frac{1}{2}^+,\frac{5}{2}^+\right)\\
\to & \left((\frac{1}{2}-2b)^+, \dots, \frac{-3}{2}^+\right) \left((\frac{3}{2}-2a)^-,\dots,\frac{-9}{2}^-,\frac{-5}{2}^+,\frac{-1}{2}^-\right) \\
\to &(w_0\theta,w_0\phi),
\end{align*}
where the first $\to$ is a composition of $\tau$ and type A intertwining operators, all of which are injective by \eqref{eq-tau} and Proposition \ref{inj_v_x}. And the second and last $\to$ are injective by Proposition \ref{inj_v_x} and Proposition \ref{type_A_inj} respectively. So one is reduced to the case of 
$$\left(\frac{1}{2}^+,\frac{5}{2}^+\right) \longrightarrow \left(\frac{-5}{2}^-,\frac{-1}{2}^-\right),$$
that is, the case when $\begin{pmatrix} a \\ b
\end{pmatrix} = \begin{pmatrix} 2 \\ 0
\end{pmatrix}$.

As for $\omega^h$. Consider the following operators:
{\allowdisplaybreaks
\begin{align*}
&\ \ \left((\frac{3}{2}-2a)^+, \dots, (2b-\frac{1}{2})^+\right) \\
\to &\left((\frac{3}{2}-2a)^+, \dots, (2b-\frac{5}{2})^+\right)\left((\frac{1}{2}-2b)^-\right) \\
\to & \left((\frac{3}{2}-2a)^+, \dots, (\frac{-1}{2}-2b)^+,(\frac{1}{2}-2b)^-, (\frac{3}{2}-2b)^+, \dots, (2b-\frac{5}{2})^+\right) \\
\to & \dots \\
\to & \left((\frac{3}{2}-2a)^+, \dots, (\frac{-1}{2}-2b)^+,(\frac{1}{2}-2b)^-, (\frac{3}{2}-2b)^+, (\frac{5}{2}-2b)^- \dots, \frac{-1}{2}^+,\frac{3}{2}^+ \right) \\
\to & \left((\frac{3}{2}-2a)^+, \dots, (\frac{-1}{2}-2b)^+,(\frac{1}{2}-2b)^-, (\frac{3}{2}-2b)^+, (\frac{5}{2}-2b)^- \dots, \frac{-3}{2}^-,\frac{-1}{2}^+ \right) \\
= &(w_0\theta,w_0\phi),
\end{align*}}
As before, one applies \eqref{eq-tau} and Proposition \ref{type_A_inj} interchangeably to conclude that all $\to$ are injective except the last one:
$$\left(\frac{-1}{2}^+,\frac{3}{2}^+ \right) \longrightarrow \left(\frac{-3}{2}^-,\frac{-1}{2}^+ \right) $$
which is the case when $\begin{pmatrix} a \\ b
\end{pmatrix} = \begin{pmatrix} 1 \\ 1
\end{pmatrix}$.

In conclusion, one is reduced to proving the proposition when $\begin{pmatrix} a \\ b
\end{pmatrix} = \begin{pmatrix} 2 \\ 0
\end{pmatrix}$ or $\begin{pmatrix} 1 \\ 1
\end{pmatrix}$. Both of them can be checked by \texttt{atlas} as outlined in Appendix A.
\end{proof}

As in Corollary \ref{cor-u}, we have the following:
\begin{corollary} \label{cor-bu}
Let $\pi = J(M_{1/2} \cup M_{1/2}^h,N_{1/2} \cup N_{1/2}^h)$, where $N_{1/2}$ is as given in Proposition \ref{U_case_odd}. Then $\pi$ has indefinite signature at $\mathcal{V}_{\eta(2b+2)}$.
\end{corollary}

\subsection{Base Case II}
\begin{proposition}
    \label{V_case_odd}
Let $n = c+d+e+f$, and $N_{1/2} \longleftrightarrow\begin{pmatrix} c & d\\ e & f\end{pmatrix}$ with $f>c$. Let $A_{1/2} \longleftrightarrow\begin{pmatrix} c\\ f \end{pmatrix}$ and $D_{1/2} \longleftrightarrow\begin{pmatrix} d\\ e\end{pmatrix}$.
Then $J(M_{1/2}\cup M_{1/2}^h, N_{1/2}\cup N_{1/2}^h)$ is the lowest $K$-type subquotient of the induced representation:
\[\mathrm{Ind}_{A_{2(c+f)-1}\times B_{2(d+e)}}^G\big(J_A(A_{1/2} \cup A_{1/2}^h) \otimes J_B(B_{1/2} \cup B_{1/2}^h)\big),\]
both having the same multiplicity of $\mathcal{V}_{\eta(i)}$ for all $0\leq i\leq n$. 
\end{proposition}
\begin{proof}
We will prove the more general case when $f \geq c-1$. As in Proposition \ref{two_string_mult}, consider the intertwining operators:
\begin{align*}
\left((\frac{1}{2}-2f)^+,\dots,(2c-\frac{3}{2})^+\right)\left(
(\frac{1}{2}-2e)^+,\dots,(\frac{1}{2}-2d)^+,(\frac{3}{2}-2d)^-,\dots,\frac{-3}{2}^+,\frac{-1}{2}^{-} \right) &\xrightarrow{\omega} (w_0\theta,w_0\phi),\\
\left((\frac{3}{2}-2c)^+,\dots,(2f-\frac{1}{2})^+\right) \left((\frac{1}{2}-2e)^-,\dots, (\frac{1}{2}-2d)^-,(\frac{3}{2}-2d)^+,\dots,\frac{-3}{2}^-,\frac{-1}{2}^{+}\right) &\xrightarrow{\omega^h} (w_0\theta^h,w_0\phi^h)
\end{align*}
By a similar argument as in Step 1 - Step 3 in the proof of Proposition \ref{two_string_mult}, one can reduce to the cases when $N_{1/2} \longleftrightarrow
\begin{pmatrix}
 1  &  0\\
 1 &  0
\end{pmatrix} = \begin{pmatrix}
 1 \\
 1 
\end{pmatrix}$ (for $\omega$) and $\begin{pmatrix}
 0 & 0 \\
 1 & 1
\end{pmatrix}$ (for $\omega^h$), where one can verify the proposition by applying \texttt{atlas} as before. 
\end{proof}

\begin{corollary} \label{cor-bv}
Let $\pi$ be as given in Proposition \ref{V_case_odd}. Then the Hermitian form of $\pi$ is indefinite at the $K$-type $\mathcal{V}_{\eta(2c+1)}$.
\end{corollary}

\subsection{General Case}
Similar to Section \ref{sec-comp}, one has the same results (with identical proofs) for induction of complementary series. Namely, one has the following result analogous to Proposition \ref{prop-nonunit2}:
\begin{proposition} \label{prop-nonunitb}
Let $\pi := J_{B_{2n}}(M_{1/2} \cup M_{1/2}^h,N_{1/2} \cup N_{1/2}^h)$, where
$$N_{1/2} \longleftrightarrow \begin{pmatrix} s_1 & \dots & s_p & * & s_{p+1} & \dots & s_q  \\ t_1 & \dots & t_p & ** & t_{p+1} & \dots & t_q \end{pmatrix}$$
satisfies
\begin{itemize}
\item $s_i - t_i = 0$ or $1$; and
\item $\begin{pmatrix} * \\ ** \end{pmatrix} = \begin{pmatrix} a \\ b \end{pmatrix}$ or $\begin{pmatrix} c & d \\ e & f \end{pmatrix}$ are as given in Corollary \ref{cor-bu} and Corollary \ref{cor-bv}.
\end{itemize}
Consider the induced module
\[
\mathrm{Ind}_{\prod_{i = 1}^{p} A_{2j_i-1} \times A_{2j-1} \times B_m}^{B_{2n}}\begin{pmatrix} \bigotimes_{i=1}^p \mathrm{Ind}_{A_{j_i-1}  \times A_{j_i-1}}^{A_{2j_i-1}}\left(\mathrm{comp}_{1/2}(j_i, \frac{1}{2}) \otimes \mathrm{comp}_{1/2}(j_i, \frac{1}{2})^h\right) \otimes \\
J_A (\widetilde{A}_{1/2} \cup \widetilde{A}_{1/2}^h)\ \otimes\   J_B(\widetilde{B}_{1/2} \cup \widetilde{B}_{1/2}^h)\end{pmatrix},
\]
where $j_i := s_i+t_i$, and 
\begin{align*}
\left(\widetilde{A}_{1/2},\widetilde{B}_{1/2}\right) &\longleftrightarrow \begin{cases}
    \left(\begin{pmatrix} a \\ b \end{pmatrix},\ \begin{pmatrix} s_{p+1} & \dots & s_q \\ t_{p+1} & \dots & t_q \end{pmatrix}\right) &\text{if}\ \begin{pmatrix} * \\ ** \end{pmatrix} = \begin{pmatrix} a \\ b \end{pmatrix}\\
    \left(\begin{pmatrix} c \\ f \end{pmatrix}, \begin{pmatrix} d & s_{p+1} & \dots & s_q \\ e & t_{p+1} & \dots & t_q \end{pmatrix}\right) &\text{if}\ \begin{pmatrix} * \\ ** \end{pmatrix} = \begin{pmatrix} c & d \\ e & f \end{pmatrix} \end{cases}
\end{align*}
Then $\pi$ and the above induced module have the same multiplicities for all $K$-types with highest weight $\eta(i)$ for all $0\leq i\leq n$.
\end{proposition}

In view of the proposition above, the non-unitarity proof for Type B is identical to that of type D -- namely, for any module $\pi$ in \eqref{eq-bhalf} not satisfying the hypothesis of Theorem \ref{Main_odd}, then by inducing some suitable complementary series to $\pi$, the $N_{1/2}$-parameter of the induced module will of the form given in Proposition \ref{prop-nonunitb}, whose lowest $K$-type irreducible subquotient is not unitary. This implies the induced module and $\pi$ are not unitary. \qed

\appendix

\section{Some atlas calcuations}
In this section, we finish the proof of Proposition \ref{two_string_mult} by the following:
\begin{lemma}\label{two_string_v2=-1/2_case}
Let $G = Spin(8,\mathbb{C})$, then Proposition \ref{two_string_mult} holds for the $(\mathfrak{g},K)$-modules $J(M_{1/2} \cup M_{1/2}^h, N_{1/2} \cup N_{1/2}^h)$, with
$$N_{1/2} \longleftrightarrow \begin{pmatrix}
2& 2\\ 0 & 0
\end{pmatrix},\quad 
\begin{pmatrix}
2 & 1\\  1 & 0
\end{pmatrix}.$$
\end{lemma}
\begin{proof}
This can be done by direct calculations similar to that in Appendix A of \cite{BW24}. Alternatively one can
verify it by \texttt{atlas}. We present the case of $N_{1/2} = (\frac{5}{2},\frac{1}{2},\frac{5}{2},\frac{1}{2}) \longleftrightarrow \begin{pmatrix}
2& 2\\ 0 & 0
\end{pmatrix}$. 

In this case, the induced module in the proposition is
$$\mathrm{Ind}_{A_1 \times D_2}^{D_4}\left(J\left(\frac{5}{2}^+,\frac{1}{2}^+\right) \otimes J\left(\frac{5}{2}^+,\frac{1}{2}^+\right)\right)$$
and its multiplicities for the spin-relevant $K$-types $\mathcal{V}_{\eta(i)}$ can be obtained by Frobenius reciprocity:
$$[\mathcal{V}_{\eta(i)}:J(M_{1/2},N_{1/2})] = 1\ (0 \leq i \leq 4,\ i \neq 2), \quad \quad
[\mathcal{V}_{\eta(2)}:J(M_{1/2},N_{1/2})] = 2.$$ 

On the other hand, \texttt{atlas} gives the multiplicities of $\mathcal{V}_{\eta(i)}$ for $J(M_{1/2},N_{1/2})$. More explicitly, the \texttt{atlas} parameter of $J(M_{1/2},N_{1/2})$ is given by:
\begin{verbatim}
set G = complexification(Spin(8,0))
set p = parameter(G,163,[-1,4,-1,0,1,1,1,1]/1,[0,2,0,1,0,3,0,-1]/2))
\end{verbatim}
One can check the $(\mu,\nu)$-parameter of \texttt{p} is:
\begin{center}
\texttt{mu\_C(p)} $= [0,0,0,1] = (\frac{1}{2}, \frac{1}{2}, \frac{1}{2}, \frac{1}{2})$;\quad
\texttt{nu\_C(p)} $= [0,2,0,1] = (\frac{5}{2}, \frac{5}{2}, \frac{1}{2}, \frac{1}{2}).$
\end{center}
Now we look at some $K$-types of \texttt{p}:
\begin{verbatim}
atlas> branch_irr(p,18)
Value:
1* K_type(x=0, lambda=[0,0,0,0,0,0,0,1]/1) [6]
1* K_type(x=0, lambda=[0,0,0,0,1,0,1,0]/1) [12]
2* K_type(x=0, lambda=[0,0,0,0,0,1,0,1]/1) [16]
1* K_type(x=0, lambda=[0,0,0,1,0,0,0,2]/1) [18]
1* K_type(x=0, lambda=[0,0,1,0,0,0,1,1]/1) [18]
\end{verbatim}
So the result follows.
\end{proof}
For completeness, we record all \texttt{atlas} parameters for the modules needed for the proofs of the propositions in the previous sections:
\begin{itemize}
    \item $G = Spin(8,\mathbb{C})$:
    \begin{center} $N_{1/2} \longleftrightarrow \begin{pmatrix} 2 & 1 \\ 1& 0\end{pmatrix}$: \texttt{parameter(G,118,[2,0,2,-1,0,2,0,0],[2,0,2,-1,-1,3,-1,0]/2)}\end{center}
    \item $G = Spin(5,\mathbb{C})$:
    \begin{align*}
    N_{1/2} \longleftrightarrow \begin{pmatrix} 2 \\ 0\end{pmatrix}:\ &\texttt{parameter(G,5,[2,0,1,1]/1,[2,1,3,-1]/2)} \\ 
    N_{1/2} \longleftrightarrow \begin{pmatrix} 1 \\ 1 \end{pmatrix}:\ &\texttt{parameter(G,3,[2,-2,1,1]/1,[2,-1,-1,3]/2)} \\ 
N_{1/2} \longleftrightarrow \begin{pmatrix} 0 & 0\\ 1& 1\end{pmatrix}:\ &\texttt{parameter(G,6,[-1,3,1,2]/1,[0,3,0,3]/2)}
    \end{align*}
\end{itemize}

\section{Unipotent representations} \label{sec-unipotent}
As mentioned in the introduction, \cite{B89} and \cite{B17} considered a larger
class of representations called {\bf unipotent representations} for all classical groups, which generalizes the notion of special unipotent representations given in \cite{BV85}. More precisely,
for each classical nilpotent orbit $\mathcal{O}$, an element $\lambda_{\mathcal{O}} \in \mathfrak{h}^*$ is assigned to $\mathcal{O}$. Then $\pi$ is a unipotent representation attached to $\mathcal{O}$ if it has infinitesimal character $(\lambda_{\mathcal{O}}, \lambda_{\mathcal{O}})$, and its annihilator ideal $\mathrm{Ann}(\pi) \subseteq U(\mathfrak{g})$ is maximal subject to all irreducible representations with infinitesimal character
$(\lambda_{\mathcal{O}}, \lambda_{\mathcal{O}})$.
For instance, the oscillator representations in $Sp(2n,\mathbb{C})$ are (non-special) unipotent representations attached to the minimal nilpotent orbit.

\medskip
In this section, we only deal with genuine unipotent representations of type D (the results for type B is stated in \ref{rmk-unipb}). Using \cite[Section 2.8]{B17}, one can directly check the following holds:
\begin{lemma} \label{lem-b1}
Let $N_{1/2} \longleftrightarrow \begin{pmatrix} x_1 & x_2 & \dots & x_k \\ y_1 & y_2 & \dots & y_k \end{pmatrix}$, where
$$x_i > y_i\quad \text{and} \quad y_i + 1 > x_{i+1}$$
satisfies the hypothesis of Remark \ref{rmk-brega}. Then $\lambda_L  = \left(\frac{1}{2}(M_{1/2} + N_{1/2}); \frac{1}{2}(M_{1/2}^h + N_{1/2}^h)\right)$ is a Weyl conjugate of $\lambda_{\mathcal{B}}$, where
\[\mathcal{B} = \mathop{\bigcup}_{i=1}^k(2x_i,2x_i-1,2y_i+1,2y_i)\]
is a nilpotent orbit of $\mathfrak{so}(4n,\mathbb{C})$ (here a nilpotent orbit is described by the column sizes
of its corresponding partition).
\end{lemma}

By Remark \ref{rmk-brega}, the above lemma implies that all Brega representations
are of the form $J_{D_{2n}}(M_{1/2} \cup M_{1/2}^h, N_{1/2} \cup N_{1/2}^h) = J_{D_{2n}}(\lambda_{\mathcal{B}};w\lambda_{\mathcal{B}})$
for some Weyl group element $w$. As discussed above, one can conclude that
it is a unipotent representation if the following holds:
\begin{theorem}
Suppose $N_{1/2} $ satisfies the hypothesis of Remark \ref{rmk-brega}, so that
$$\pi = J_{D_{2n}}(M_{1/2} \cup M_{1/2}^h, N_{1/2} \cup N_{1/2}^h) = \mathrm{Ind}_{D_n \times D_n}^{D_{2n}}\left(J_{D_{n}}(M_{1/2},N_{1/2}) \otimes J_{D_{n}}(M_{1/2}^h, N_{1/2}^h)\right).$$
is a Brega representation. Then its associated variety is equal to $AV(\mathrm{Ann}(\pi)) = \overline{\mathcal{B}}$, where $\mathcal{B}$ is as given in Lemma \ref{lem-b1}.
\end{theorem}
\begin{proof}
For all nilpotent orbits $\mathcal{O}$, let $\mathrm{Spr}(\mathcal{O})$ be its corresponding Weyl group representation under the Springer correspondence. In particular, the {\it fake degree}
of $b(\mathrm{Spr}(\mathcal{O}))$ of $\mathrm{Spr}(\mathcal{O})$, i.e. the smallest degree in which $\mathrm{Spr}(\mathcal{O})$ appears in the harmonic polynomials on the Cartan subalgebra $\mathfrak{h}$ is equal to the {\it co}-dimension $\dim(\mathcal{N}) - \dim(\mathcal{O})$ of the nilpotent cone $\mathcal{N}$ of $\mathfrak{g}$.

By Lemma 3.1 of \cite{Br99}, the associated varieties of the annihilators of
$J_{D_{n}}(M_{1/2}, N_{1/2})$ and $J_{D_{n}}(M_{1/2}^h, N_{1/2}^h)$
are both equal to the closure of the special nilpotent orbit $\mathcal{P} := \mathop{\bigcup}_{i=1}^k(2x_i,2y_i)$.
By the Kazhdan-Lusztig conjectures for non-integral infinitesimal characters (for instance \cite[Theorem 5.1]{McG94}), the character theory
(and hence the associated variety) of $\pi = J_{D_{2n}}(M_{1/2} \cup M_{1/2}^h, N_{1/2} \cup N_{1/2}^h)$
can be obtained from that of $J_{D_{n}}(M_{1/2}, N_{1/2})$ and $J_{D_{n}}(M_{1/2}^h, N_{1/2}^h)$. In particular, the associated variety
of $\pi$ satisfies
\begin{equation} \label{eq-av}
\mathrm{Spr}(AV(\mathrm{Ann}(\pi))) = j_{D_n \times D_n}^{D_{2n}}\left(\mathrm{Spr}(\mathcal{P}) \otimes \mathrm{Spr}(\mathcal{P})\right),\end{equation}
where $j_{W_0}^{W}$ is the {\it truncated induction}, which (by definition) preserves the fake degree of representations.

\medskip
We are now in the position to prove the theorem: By \cite[Corollary 5.18]{BV85}, $AV(\pi)$ is bounded below by
$\overline{\mathcal{B}}$. Also, \eqref{eq-av} and the first paragraph of the proof imply that
\begin{align*}
\dim(\mathcal{N}_{\mathfrak{so}(4n,\mathbb{C})}) - \dim(AV(\mathrm{Ann}(\pi))) &=  b(\mathrm{Spr}(AV(\mathrm{Ann}(\pi)))) \\
&= b(\mathrm{Spr}(\mathcal{P})) + b(\mathrm{Spr}(\mathcal{P}))\\
&= 2\left(\dim(\mathcal{N}_{\mathfrak{so}(2n,\mathbb{C})}) - \dim(\mathcal{P})\right)
\end{align*}
On the other hand, by the formula of the dimension of nilpotent orbits \cite[Theorem 6.1.3]{CM93}, we know that
\begin{align*}
\dim(\mathcal{B}) = & \ \frac{4n(4n-1)}{2}-\Big(\frac{1}{2}\sum_{i=1}^k\big((2x_i)^2+(2x_i-1)^2+(2y_i+1)^2+(2y_i)^2\big)-\frac{1}{2}\sum_{i=1}^k(1+1)\Big), \\
\dim(\mathcal{P})= & \  \frac{2n(2n-1)}{2}-\Big(\frac{1}{2}\sum_{i=1}^k\big((2x_i)^2+(2y_i)^2\big)-\frac{1}{2}\sum_{i=1}^k\big(2x_i-2y_i\big)\Big),
\end{align*}
Using the above formulas, one can check directly that
$$\dim(\mathcal{N}_{\mathfrak{so}(4n,\mathbb{C})}) -\dim(\mathcal{B}) =2\left(\dim(\mathcal{N}_{\mathfrak{so}(2n,\mathbb{C})})-\dim(\mathcal{P})\right).$$
So $\dim(AV(\mathrm{Ann}(\pi))) = \dim(\mathcal{B})$, and hence $AV(\mathrm{Ann}(\pi)) = \overline{\mathcal{B}}$ attains its lower bound.
\end{proof}

\begin{example}
Let $N_{1/2} \longleftrightarrow \begin{pmatrix} n \\ 0 \end{pmatrix}$. Then
\begin{align*} \lambda_L  &= \frac{1}{2}\left((\frac{1}{2}, \dots,\frac{1}{2}, \frac{1}{2})+(\frac{4n-3}{2}, \dots,\frac{5}{2}, \frac{1}{2});\ ((\frac{1}{2}, \dots,\frac{1}{2}, \frac{1}{2})+(\frac{-(4n-3)}{2}, \dots,\frac{-5}{2}, \frac{-1}{2}) \right) \\
&= \left(\frac{2n-1}{2}, \frac{2n-3}{2}, \dots, \frac{1}{2};\ 0, -1, \dots, -(n-1)\right) \sim \frac{\rho}{2},
\end{align*}
where $\rho = \rho_{\mathfrak{so}(4n,\mathbb{C})}$ is half the sum of all positive roots of $\mathfrak{so}(4n,\mathbb{C})$. By Lemma \ref{lem-b1}, $\lambda_L = \lambda_{\mathcal{B}}$, where $\mathcal{B} = [2n,2n-1,1]$ is the {\bf model orbit} of dimension $4n^2$, and
$$J_{D_{2n}}(M_{1/2} \cup M_{1/2}^h, N_{1/2} \cup N_{1/2}^h) = \mathrm{Ind}_{D_n \times D_n}^{D_{2n}}\left(J_{D_{n}}(M_{1/2}, N_{1/2}) \otimes J_{D_{n}}(M_{1/2}^h, N_{1/2}^h)\right) = J_{D_{2n}}(\frac{\rho}{2}; \frac{w\rho}{2})$$
for some $w \in W(D_{2n})$.

On the other hand,
the representations $J_{D_n}(M_{1/2}, N_{1/2}) = S^+ \otimes \mathrm{triv}$ and $J_{D_n}(M_{1/2}^h, N_{1/2}^h)  = \mathrm{triv} \otimes S^+$ are both finite-dimensional, both having a single $K$-type equal to the spinor module $S^+$. Therefore, their associated varieties $\mathcal{P}$ are both the zero orbit. Consequently,
$$\dim(\mathcal{N}_{\mathfrak{so}(4n,\mathbb{C})}) - \dim(\mathcal{B}) = \frac{4n(4n-1)}{2} - 4n^2 = 2\left(\dim(\mathcal{N}_{\mathfrak{so}(2n,\mathbb{C})}) - \dim(\mathcal{P})\right)$$
and hence $AV(\mathrm{Ann}(\pi)) = \overline{\mathcal{B}} = [2n,2n-1,1]$.
\end{example}

\end{document}